%% file: logGeometry_arXiv.tex
\documentclass[arxiv]{agn_article}

\usepackage{agn_all}
\usepackage[babel]{csquotes}
\usepackage{wrapfig}
\usepackage{graphicx}
\usepackage{footmisc}
\usepackage{caption}
\usepackage{enumitem}

\usepackage{tablefootnote}

\usepackage{etoolbox}
\AtBeginEnvironment{remark}{\pushQED{\ifmmode \tag*{$\square$}\else\hfill$\square$\fi}} %
\AtEndEnvironment{remark}{\popQED}

\newcommand{\measure}{\omega}

\newcommand{\isomeas}{\omega_{\mathrm{iso}}}
\newcommand{\volmeas}{\omega_{\mathrm{vol}}}
\newcommand{\Fiso}{F_{\mathrm{iso}}}
\newcommand{\Fvol}{F_{\mathrm{vol}}}
\newcommand{\dgiso}{\dist_{\mathrm{geod,}\SLn}}
\newcommand{\dgisotwo}{\dist_{\mathrm{geod,}\SL(2)}}
\newcommand{\dgisothree}{\dist_{\mathrm{geod,}\SL(3)}}
\newcommand{\dgvol}{\dist_{\mathrm{geod,}\,\R^+\!\cdot\id}}
\newcommand{\dgfull}{\dist_{\mathrm{geod,}\GLpn}}
\newcommand{\dgparam}{\dist_{\mathrm{geod},\mu,\muc,\kappa}}
\newcommand{\dgparamtilde}{\dist_{\mathrm{geod},\mu,\mutilde_c,\kappa}}
\newcommand{\dgso}{\dist_{\mathrm{geod,}\SOn}}
\newcommand{\dgcso}{\dist_{\mathrm{geod,}\CSOn}}
\newcommand{\lenparam}{\len_{\mu,\muc,\kappa}}
\newcommand{\lenparamtilde}{\len_{\mu,\mutilde_c,\kappa}}
\newcommand{\isonormtilde}[1]{{\norm{#1}}_{\mu, \mutilde_c, \kappa}}
\newcommand{\gammaiso}{\gamma_{\mathrm{iso}}}
\newcommand{\gammavol}{\gamma_{\mathrm{vol}}}
\newcommand{\gammaisodot}{{\dot\gamma}_{\mathrm{iso}}}
\newcommand{\gammavoldot}{{\dot\gamma}_{\mathrm{vol}}}
\DeclareMathOperator{\CSO}{CSO}
\newcommand{\CSOn}{\CSO(n)}
\newcommand{\sigmaeH}{\sigma_{\mathrm{eH}}}
\newcommand{\sigmaH}{\sigma_{\mathrm{H}}}
\newcommand{\fulldisteuclid}{\dist_{\mathrm{Euclid,}\mu,\muc,\kappa}}
\newcommand{\fulldistgeod}{\dist_{\mathrm{geod,}\mu,\muc,\kappa}}
\newcommand{\fulldistgeodPsym}{\dist_{\mathrm{geod,}\PSymn,\mu,\kappa}}
\newcommand{\fulldisteuclidSym}{\dist_{\mathrm{Euclid,}\Symn,\mu,\kappa}}
\newcommand{\disteuclidPhiOmega}{\dist_{\mathrm{Euclid,}\varphi(\Omega)}}
\newcommand{\dgOmega}{\dist_{\mathrm{geod,}\Omega}}
\newcommand{\scalee}{\mathrm{e}}
\newcommand{\scaleetilde}{\tilde{\mathrm{e}}}
\newcommand{\dgright}{\dist_{\mathrm{geod,right}}}
\newcommand{\ddtdot}{\ddt^{\!\!\circ}}
\newcommand{\ddtsquare}{\ddt^{\!\!\square}}
\newcommand{\ddttriangle}{\ddt^{\!\!\vartriangle}}
\newcommand{\ddttriangledown}{\ddt^{\!\!\triangledown}}
\newcommand{\ddtlog}{\ddt^{\!\!\mathrm{log}}}
\newcommand{\scaleelog}{\scalee_{\mathrm{log}}}
\newcommand{\Biot}{T^{\mathrm{Biot}}}
\DeclareMathOperator{\Curl}{Curl}
\DeclareMathOperator{\curl}{curl}
\newcommand{\quoteref}[1]{\enquote{\emph{#1}}}
\newcommand{\quoteesc}[1]{{\rm[}#1{\rm]}}
\newcommand{\quotepar}{\\\vspace*{\parskip}\indent}
\newcommand{\hot}{\mathrm{h.o.t.}}
\newcommand{\gbad}{{\check{g}}}
\newcommand{\gbadtoo}{{\check{\check g}}}
\renewcommand{\afrac}[2]{#1\! /\!\!\; #2}

\input{tikz/tikzGlobal.tex}

\makeatletter
\let\@fnsymbol\@arabic
\makeatother

\begin{document}
\title{\vspace*{-1.4cm}\Huge Geometry of logarithmic strain measures in solid mechanics}
\date{\today\vspace*{-.7em}}
\author{Patrizio Neff\thanks{\;\;Corresponding author: Patrizio Neff, Head of Chair for Nonlinear Analysis and Modelling, Fakult\"at f\"ur Mathematik, Universit\"at Duisburg-Essen, Campus Essen, Thea-Leymann Stra{\ss}e 9, 45141 Essen, Germany, email: patrizio.neff@uni-due.de, Tel.:+49-201-183-4243},
\quad Bernhard Eidel\thanks{\;\;Bernhard Eidel, Chair of Computational Mechanics, Universit\"at Siegen, Paul-Bonatz-Stra{\ss}e 9-11, 57068 Siegen, Germany, email: bernhard.eidel@uni-siegen.de}
\ \quad and \quad Robert J. Martin\thanks{\;\;Robert J. Martin, Chair for Nonlinear Analysis and Modelling, Fakult\"at f\"ur Mathematik, Universit\"at Duisburg-Essen, Campus Essen, Thea-Leymann Stra{\ss}e 9, 45141 Essen, Germany, email: robert.martin@uni-due.de}
\\[1em]{\small Published in Arch.~Rational~Mech.~Anal., vol.\ 222 (2016), 507--572.}
\\{\small DOI: 10.1007/s00205-016-1007-x}
\\[1em]{\small In memory of Giuseppe Grioli (*10.4.1912 -- \textdagger4.3.2015), a true paragon of rational mechanics}
\\[1em]
}
\maketitle
\begin{abstract}
\noindent We consider the two logarithmic strain measures
\[
	\isomeas = \norm{\dev_n \log U} = \norm{\dev_n \log \sqrt{F^TF}} \quad\text{ and }\quad \volmeas = \abs{\tr(\log U)} = \abs{\tr(\log\sqrt{F^TF})} = \abs{\log(\det U)}\,,
\]
which are isotropic invariants of the Hencky strain tensor $\log U$, and show that they can be uniquely characterized by purely geometric methods based on the geodesic distance on the general linear group $\GLn$. Here, $F$ is the deformation gradient, $U=\sqrt{F^TF}$ is the right Biot-stretch tensor, $\log$ denotes the principal matrix logarithm, $\norm{\,.\,}$ is the Frobenius matrix norm, $\tr$ is the trace operator and $\dev_n X = X-\frac1n\,\tr(X)\cdot\id$ is the $n$-dimensional deviator of $X\in\Rnn$.
This characterization identifies the Hencky (or true) strain tensor as the natural nonlinear extension of the linear (infinitesimal) strain tensor $\eps=\sym\grad u$, which is the symmetric part of the displacement gradient $\grad u$, and reveals a close geometric relation between the classical quadratic isotropic energy potential
\[
	\mu \, \norm{\dev_n \sym\grad u}^2+\frac{\kappa}{2}\,[\tr(\sym\grad u)]^2 = \mu \, \norm{\dev_n \eps}^2+\frac{\kappa}{2}\,[\tr(\eps)]^2
\]
in linear elasticity and the geometrically nonlinear quadratic isotropic Hencky energy
\[
	\mu \, \norm{\dev_n\log U}^2+\frac{\kappa}{2}\,[\tr(\log U)]^2 = \mu\,\isomeas^2 + \frac\kappa2\,\volmeas^2\,,
\]
where $\mu$ is the shear modulus and $\kappa$ denotes the bulk modulus. Our deduction involves a new fundamental logarithmic minimization property of the orthogonal polar factor $R$, where $F=R\,U$ is the polar decomposition of $F$. We also contrast our approach with prior attempts to establish the logarithmic Hencky strain tensor directly as the preferred strain tensor in nonlinear isotropic elasticity.
\\[.63em]
\end{abstract}
\\[1.4em]
\thispagestyle{empty}

\newpage
\tableofcontents

\section{Introduction}
\subsection{What's in a strain?}
\label{section:actualIntroduction}
The concept of \emph{strain} is of fundamental importance in elasticity theory. In linearized elasticity, one assumes that the Cauchy stress tensor $\sigma$ is a linear function of the symmetric infinitesimal strain tensor
\[
	\eps = \sym \grad u = \sym(\grad \varphi - \id) = \sym(F-\id)\,,
\]
where $\varphi\col\Omega\to\R^n$ is the deformation of an elastic body with a given reference configuration $\Omega\subset\R^n$, $\varphi(x) = x+u(x)$ with the displacement $u$, $F=\grad\varphi$ is the deformation gradient\footnote{Although $F$ is widely known as the deformation \enquote{gradient}, $F=\grad\varphi=D\varphi$ actually denotes the \emph{first derivative} (or the \emph{Jacobian matrix}) of the deformation $\varphi$.}, $\sym\grad u = \frac12(\grad u + (\grad u)^T)$ is the symmetric part of the displacement gradient $\grad u$ and $\id\in\GLpn$ is the identity tensor in the group of invertible tensors with positive determinant. In geometrically nonlinear elasticity models, it is no longer necessary to postulate a linear connection between some stress and some strain. However, nonlinear strain tensors are often used in order to simplify the stress response function, and many constitutive laws are expressed in terms of linear relations between certain strains and stresses\footnote{%
In a short note \cite{brannon2011define}, R.~Brannon observes that \quoteref{usually, a researcher will select the strain measure for which the stress-strain curve is most \textbf{linear}}. In the same spirit, Bruhns \cite[p.~147]{bruhns2014history} states that \quoteref{we should \quoteesc{\ldots} always use the logarithmic Hencky strain measure in the description of finite deformations.}. Truesdell and Noll \cite[p.~347]{truesdell65} explain: \quoteref{Various authors \quoteesc{\ldots} have suggested that we should select the strain \quoteesc{tensor} afresh for each material in order to get a simple form of constitutive equation. \quoteesc{\ldots} \textbf{Every} invertible stress relation $T=f(B)$ for an isotropic elastic material is linear, trivially, in an appropriately defined, particular strain \quoteesc{tensor $f(B)$}.}
} \cite{batra1998linear,batra2001comparison,bertram2007rank} (cf.\ Appendix \ref{section:linearRelations} for examples).

There are many different definitions of what exactly the term \enquote{strain} encompasses: while Truesdell and Toupin \cite[p.~268]{truesdell60} consider \quoteref{any uniquely invertible isotropic second order tensor function of \quoteesc{the right Cauchy-Green deformation tensor $C=F^TF$}} to be a strain tensor, it is commonly assumed \cite[p.\ 230]{Hill68} (cf.\ \cite{hill1970, hill1978, bertram2008elasticity, norris2008higherDerivatives}) that a (material or Lagrangian\footnote{Similarly, a \emph{spatial} or \emph{Eulerian} strain tensor $\Ehat(V)$ depends on the left Biot-stretch tensor $V=\sqrt{FF^T}$ (cf.\ \cite{fosdick1968general}).}) strain takes the form of a \emph{primary matrix function} of the right Biot-stretch tensor $U=\sqrt{F^TF}$ of the deformation gradient $F\in\GLpn$, i.e.\ an isotropic tensor function $E\col\PSymn\to\Symn$ from the set of positive definite tensors to the set of symmetric tensors of the form
\begin{equation}
\label{eq:primaryMatrixFunctionDefinition}
	E(U) = \sum_{i=1}^n \scalee(\lambda_i) \cdot e_i\otimes e_i \quad\text{ for }\quad U = \sum_{i=1}^n \lambda_i \cdot e_i\otimes e_i
\end{equation}
with a \emph{scale function} $\scalee\col(0,\infty)\to\R$, where $\otimes$ denotes the tensor product, $\lambda_i$ are the eigenvalues and $e_i$ are the corresponding eigenvectors of $U$. However, there is no consensus on the exact conditions for the scale function $\scalee$; Hill (cf.\ \cite[p.~459]{hill1970} and \cite[p.~14]{hill1978}) requires $\scalee$ to be \enquote{suitably smooth} and monotone with $\scalee(1)=0$ and $\scalee'(1)=1$, whereas Ogden \cite[p.~118]{Ogden83} also requires $\scalee$ to be infinitely differentiable and $\scalee'>0$ to hold on all of $(0,\infty)$.

The general idea underlying these definitions is clear: strain is a measure of deformation (i.e.\ the change in form and size) of a body with respect to a chosen (arbitrary) reference configuration. Furthermore, the strain of the deformation gradient $F\in\GLpn$ should correspond only to the \emph{non-rotational} part of $F$. In particular, the strain must vanish if and only if $F$ is a pure rotation, i.e.\ if and only if $F\in\SOn$, where $\SOn=\{Q\in\GLn \setvert Q^TQ=\id,\, \det Q = 1\}$ denotes the special orthogonal group. This ensures that the only strain-free deformations are rigid body movements:
\begin{alignat}{2}
	F^TF\equiv\id \quad&\Longrightarrow\quad &\grad\varphi(x) &= F(x)=R(x)\in\SOn\\
	&\Longrightarrow\quad &\varphi(x) &= \overline{Q}\,x + \overline{b} \quad\text{for some fixed $\overline{Q}\in\SOn,\, \overline{b}\in\R^n$,} \nonumber
\end{alignat}
where the last implication is due to the rigidity \cite{resetnjak1967liouville} inequality $\norm{\Curl R}^2 \geq c^+\,\norm{\grad R}^2$ for $R\in\SOn$ (with a constant $c^+>0$), cf.\ \cite{Neff_curl06}. A similar connection between vanishing strain and rigid body movements holds for linear elasticity: if $\eps\equiv0$ for the linearized strain $\eps=\sym\grad u$, then $u$ is an \emph{infinitesimal rigid displacement} of the form
\[
	u(x) = \overline{A}\,x+\overline{b} \quad\text{with fixed $\overline{A}\in\son,\, \overline{b}\in\R^n$,}
\]
where $\son = \{A\in\Rnn : A^T = -A\}$ denotes the space of skew symmetric matrices. This is due to the inequality $\norm{\Curl A}^2 \geq c^+\,\norm{\grad A}^2$ for $A\in\son$, cf.\ \cite{Neff_curl06}.

In the following, we will use the term \emph{strain tensor} (or, more precisely, \emph{material strain tensor}) to refer to an injective isotropic tensor function $U \mapsto E(U)$ of the right Biot-stretch tensor $U$ mapping $\PSymn$ to $\Symn$ with
\begin{alignat*}{2}
	& E(Q^TU\,Q)=Q^TE(U)\,Q \qquad\text{for all $Q\in\On$} &\tag{isotropy}\\
	\text{and }\quad & E(U) = 0 \quad\Longleftrightarrow\quad U=\id\,;
\end{alignat*}
where $\On=\{Q\in\GLn \setvert Q^TQ=\id\}$ is the orthogonal group and $\id$ denotes the identity tensor. In particular, these conditions ensure that $\id=U=\sqrt{F^TF}$ if and only if $F\in\SOn$. Note that we do not require the mapping to be of the form \eqref{eq:primaryMatrixFunctionDefinition}.

Among the most common examples of material strain tensors used in nonlinear elasticity is the \emph{Seth-Hill family}\footnote{Note that $\log U = \lim\limits_{r\to0} \frac{1}{2\,r}(U^{2r}-\id)$. Many more examples of strain tensors used throughout history can be found in \cite{curnier1991} and \cite{doyle1956nonlinear}, cf.\ \cite[p.\ 132]{bigoni2012nonlinear}.} \cite{seth1961generalized}
\begin{equation}\label{eq:sethHillFamily}
	E_r(U) =
	\begin{cases}
		\frac{1}{2\,r}(U^{2r}-\id) \quad&:\; r\in\R\setminus\{0\}\\ %
		\log U &:\; r=0
	\end{cases}
\end{equation}
\begin{wrapfigure}{r}{0.45\textwidth}
	\centering
	\tikzsetnextfilename{strainPlot}
	\begin{tikzpicture}[scale=1.26]
		\input{tikz/strainPlot.tex}
	\end{tikzpicture}
	\caption{\label{fig:strainPlot}Scale functions $\scalee_r,\scaleetilde_r$ associated with the strain tensors $E_r$ and $\widetilde{E}_r=\frac12(E_r-E_{-r})$ via eigenvalue $\lambda$.}
\end{wrapfigure}
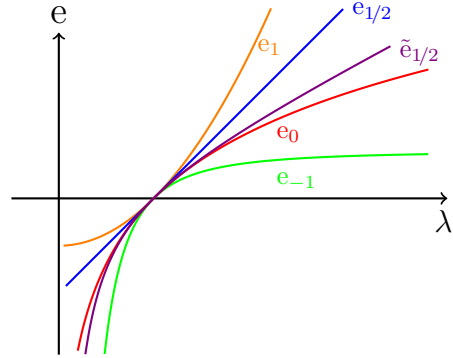
of material strain tensors\footnote{%
The corresponding family of spatial strain tensors
\[
	\Ehat_r(V) =
	\begin{cases}
		\frac{1}{2\,r}(V^{2r}-\id) \quad&:\; r\neq0\\
		\log V &:\; r=0
	\end{cases}
\]
includes the \emph{Almansi-Hamel strain tensor} $\Ehat_{\afrac12}(V)=V-\id$ as well as the \emph{Euler-Almansi strain tensor} $\Ehat_{-1}(V)=\frac12(\id-B\inv)$, where $B=FF^T=V^2$ is the Finger tensor \cite{finger1894}.%
}, which includes the \emph{Biot strain tensor} $E_{\afrac12}(U)=U-\id$, the \emph{Green-Lagrangian strain tensor} $E_1(U) = \frac12(C-\id) = \frac12(U^2-\id)$, where $C=F^TF=U^2$ is the right Cauchy-Green deformation tensor, the \emph{(material) Almansi strain tensor} \cite{almansi1911sulle} $E_{-1}(U)=\frac12(\id-C\inv)$ and the \emph{(material) Hencky strain tensor} $E_0(U)=\log U$, where $\log\col\PSymn\to\Symn$ is the \emph{principal matrix logarithm} \cite[p.~20]{higham2008} on the set $\PSymn$ of positive definite symmetric matrices. The Hencky (or logarithmic) strain tensor has often been considered the \emph{natural} or \emph{true} strain in nonlinear elasticity \cite{tarantola2009stress,Tarantola06,freed1995natural,hanin1956isotropic}. It is also of great importance to so-called hypoelastic models, as is discussed in \cite{xiao1997logarithmic,freed2014hencky} (cf.\ Section \ref{section:hypoelasticity}).\footnote{%
Bruhns \cite[p.\ 41--42]{bruhns2015ohno} emphasizes the advantages of the Hencky strain tensor over the other Seth-Hill strain tensors in the one-dimensional case:
\quoteref{The significant advantage of this logarithmic (Hencky) measure lies in the fact that it tends to infinity as $F$ tends to zero, thus in a very natural way bounding the regime of applicability to the case $F>0$. This behavior can also be observed for strain \quoteesc{tensors} with negative exponent $n$. Compared with the latter, however, the logarithmic measure also goes to infinity as $F$ does, whereas it is evident that for negative values of $n$ the strain \quoteesc{$\frac1n\,(F^n-1)$} is bound to the limit \nolinebreak$-\frac1n$.
\quotepar
All measures with positive values of $n$ including the Green strain share the reasonable property of the logarithmic strain for $F$ going to infinity. For $F$ going to zero, however, these measures arrive at finite values for the specific strains, e.g. at $-\frac12$ for $n=2$, which would mean that interpreted from physics a total compression of the rod (to zero length) is related to a finite value of the strain. This awkward result would not agree with our observation - at least what concerns the behavior of solid materials.}%
}
A very useful approximation of the material Hencky strain tensor was given by Ba\v{z}ant \cite{bazant1998} (cf.\ \cite{ortiz2001computation,al2013computing,Darijani2010223}):
\begin{equation}
\label{eq:bazantDefinition}
	\widetilde{E}_{\afrac12}(U) \;\colonequals\; \tfrac12 \,[E_{\afrac12}(U) + E_{-\afrac12}(U)] \;=\; \tfrac12\,(U-U\inv)\,.
\end{equation}

Additional motivations of the logarithmic strain tensor were also given by Vall\'ee \cite{vallee1978,vallee2008dual}, Roug{\'e}e \cite[p.~302]{rougee1997mecanique} and Murphy \cite{murphy2007linear}. An extensive overview of the properties of the logarithmic strain tensor and its applications can be found in \cite{xiao2005} and \cite{agn_neff2015exponentiatedI}.

All strain tensors, by the definition employed here, can be seen as \emph{equivalent}: since the mapping $U\mapsto E(U)$ is injective, for every pair $E,E'$ of strain tensors there exists a mapping $\psi\col\Symn\to\Symn$ such that $E'(U)=\psi(E(U))$ for all $U\in\PSymn$.
Therefore, every constitutive law of elasticity can -- in principle -- be expressed in terms of any strain tensor\footnote{According to Truesdell and Toupin \cite[p.~268]{truesdell60}, \quoteref{\ldots any \quoteesc{tensor} sufficient to determine the directions of the principal axes of strain and the magnitude of the principal stretches may be employed and is fully general}. Truesdell and Noll \cite[p.~348]{truesdell65} argue that there \quoteref{is no basis in experiment or logic for supposing nature prefers one strain \quoteesc{tensor} to another}.} and no strain tensor can be inherently superior to any other strain tensor.\footnote{Nevertheless, \quoteref{\quoteesc{in} spite of this equivalence, one strain \quoteesc{tensor} may present definite technical advantages over another one} \cite[p.~467]{curnier1991}. For example, there is one and only one spatial strain tensor $\Ehat$ together with a unique objective \emph{and} corotational rate $\ddtsquare$ such that $\ddtsquare \Ehat = \sym(\dot{F}F\inv)=D$. Here, $\ddtsquare=\ddtlog$ is the logarithmic rate, $D$ is the unique rate of stretching and $\Ehat$ is the spatial Hencky strain tensor $\Ehat_0=\log V$; cf.\ Section \ref{section:hypoelasticity} and \cite{Bruhns01,xiao1997logarithmic,norris2008eulerian,zhilin2013material,gurtin1983relationship}.}
Note that this invertibility property also holds if the definition by Hill or Ogden is used: if the strain is given via a \hbox{scale function $\scalee$}, the strict monotonicity of $\scalee$ implies that the mapping $U\mapsto E(U)$ is strictly monotone \cite{agn_martin2015some}, i.e.
\[
	\innerproduct{E(U_1)-E(U_2), U_1 - U_2} > 0
\]
for all $U_1,U_2\in\PSymn$ with $U_1\neq U_2$, where $\iprod{X,Y}=\tr(X^TY)$ denotes the \emph{Frobenius inner product} on $\Symn$ and $\tr(X) = \sum_{i=1}^n X_{i,i}$ is the \emph{trace} of $X\in\Rnn$. This monotonicity in turn ensures that the mapping $U\mapsto E(U)$ is injective.

In contrast to \emph{strain} or \emph{strain tensor}, we use the term \emph{\textbf{strain measure}} to refer to a nonnegative real-valued function $\measure\col\GLpn\to[0,\infty)$ depending on the deformation gradient which vanishes if and only if $F$ is a pure rotation, i.e.\ $\measure(F)=0$ if and only if $F\in\SOn$.

Note that the terms \enquote{strain tensor} and \enquote{strain measure} are sometimes used interchangeably in the literature (e.g.\ \cite{hill1978, norris2008higherDerivatives}). A simple example of a strain measure in the above sense is the mapping $F\mapsto \norm{E(\sqrt{F^TF})}$ of $F$ to an orthogonally invariant norm of any strain tensor $E$.

There is a close connection between strain measures and \emph{energy functions} in isotropic hyperelasticity: an isotropic energy potential \cite{grioli1966thermodynamic} is a function $W$ depending on the deformation gradient $F$ such that
\begin{align}
	W(F)&\geq 0\,, &\tag{normalization}\\
	W(QF)&=W(F)\,, &\tag{frame-indifference}\\
	W(FQ)&=W(F) &\tag{material symmetry: isotropy}
\end{align}
for all $F\in\GLpn,\,Q\in\SOn$ and
\begin{equation}
	W(F) = 0 \quad\text{ if and only if }\quad F\in\SOn\,. \tag{stress-free reference configuration}
\end{equation}
While every such energy function can be taken as a strain measure\label{sectionContains:energyAsStrainMeasure}, many additional conditions for \enquote{proper} energy functions are discussed in the literature, such as constitutive inequalities \cite{truesdell1956,Hill68,hill1970,ball1977constitutive,Ciarlet1988,marsden1994foundations}, generalized convexity conditions \cite{ball1976convexity,ball2002openProblems} or monotonicity conditions to ensure that \enquote{stress increases with strain} \cite[Section 2.2]{agn_neff2015exponentiatedI}. Apart from that, the main difference between strain measures and energy functions is that the former are purely mathematical expressions used to quantitatively assess the extent of strain in a deformation, whereas the latter postulate some \emph{physical} behaviour of materials in a condensed form: an elastic energy potential, interpreted as the elastic energy per unit volume in the undeformed configuration, induces a specific stress response function\footnote{The specific elasticity tensor further depends on the particular choice of a strain and a stress tensor in which to express the constitutive law.}, and therefore completely determines the physical behaviour of the modelled hyperelastic material. The connection between \enquote{natural} strain measures and energy functions will be further discussed later on.

In particular, we  will be interested in energy potentials which can be expressed \emph{in terms of} certain strain measures. Note carefully that, in contrast to strain tensors, strain measures cannot simply be used interchangeably: for two different strain measures (as defined above) $\measure_1,\measure_2$, there is generally no function $f\col\R^+\to\R^+$ such that $\measure_2(F)=f(\measure_1(F))$ for all $F\in\GLpn$. Compared to \enquote{full} strain tensors, this can be interpreted as an unavoidable \emph{loss of information} %
for strain measures (which are only scalar quantities).

Sometimes a strain measure is employed only for a particular kind of deformation. For example, on the group of simple shear deformations (in a fixed plane) consisting of all $F_\gamma\in\GLp(3)$ of the form
\[
	F_\gamma = \matrs{1&\gamma&0\\0&1&0\\0&0&1}\,, \quad \gamma\in\R\,,
\]
we could consider the mappings
\[
	F_\gamma \mapsto \frac12\,\gamma^2 \,, \qquad F_\gamma \mapsto \frac{1}{\sqrt3}\,\abs{\gamma} \qquad\text{ or }\qquad F_\gamma \mapsto \frac{2}{\sqrt3}\,\ln\left(\frac\gamma2 + \sqrt{1+\frac{\gamma^2}{4}}\right);
\]
We will come back to these \emph{partial strain measures} in Section \ref{section:mainResults}.

In the following we consider the question of what strain measures are appropriate for the theory of nonlinear isotropic elasticity. Since, by our definition, a strain measure attains zero if and only if $F\in\SOn$, a simple geometric approach is to consider a \emph{distance function} on the group $\GLpn$ of admissible deformation gradients, i.e.\ a function $\dist\col\GLpn\times\GLpn\to[0,\infty)$ with $\dist(A,B)=\dist(B,A)$ which satisfies the triangle inequality and vanishes if and only if its arguments are identical.\footnote{A distance function is more commonly known as a metric of a metric space. The term \enquote{distance} is used here and throughout the article in order to avoid confusion with the Riemannian metric introduced later on.} Such a distance function induces a \enquote{natural} strain measure on $\GLpn$ by means of the distance to the special orthogonal group $\SOn$:
\begin{equation}
\label{eq:measureDefinedViaDistance}
	\measure(F) \colonequals \dist(F,\SOn) \colonequals \inf_{Q\in\SOn} \dist(F,Q)\,.
\end{equation}
In this way, the search for an appropriate strain measure reduces to the task of finding a \emph{natural, intrinsic distance function} on $\GLpn$.

\subsection{The search for appropriate strain measures}

The remainder of this article is dedicated to this task: after some simple (Euclidean) examples in Section \ref{section:euclideanStrainMeasures}, we consider the \emph{geodesic distance} on $\GLpn$ in Section \ref{section:riemannianStrainMeasureInNonlinearElasticity}. Our main result is stated in Theorem \ref{theorem:mainResult}: if the distance on $\GLpn$ is induced by a left-$\GLn$-invariant, right-$\On$-invariant Riemannian metric on $\GLn$, then the distance of $F\in\GLpn$ to $\SOn$ is given by
\[
	\dg^2(F,\SOn) = \dg^2(F,R) = \mu\,\norm{\dev_n\log U}^2 + \frac{\kappa}{2}\,[\tr(\log U)]^2\,,
\]
where $F=R\,U$ with $U=\sqrt{F^TF}\in\PSymn$ and $R\in\SOn$ is the polar decomposition of $F$. Section \ref{section:riemannianStrainMeasureInNonlinearElasticity} also contains some additional remarks and corollaries which further expand upon this Riemannian strain measure.

In Section \ref{section:alternativeMotivations}, we discuss a number of different approaches towards motivating the use of logarithmic strain measures and strain tensors, whereas applications of our results and further research topics are indicated in Section \ref{section:applicationsAndOngoingResearch}.

Our main result (Theorem \ref{theorem:mainResult}) has previously been announced in a Comptes Rendus M\'ecanique article \cite{Neff_Eidel_Osterbrink_2013} as well as in Proceedings in Applied Mathematics and Mechanics \cite{neff2013henckyPAMM}.

The idea for this paper was conceived in late 2006. However, a number of technical difficulties had to be overcome (cf.\ \cite{agn_birsan2013sum,agn_neff2014logarithmic,agn_lankeit2014minimization,agn_martin2014minimal,neff2013video}) in order to prove our results. The completion of this article might have taken more time than was originally foreseen, but we adhere to the old German saying: \emph{Gut Ding will Weile haben.}

\section{Euclidean strain measures}
\label{section:euclideanStrainMeasures}
\subsection{The Euclidean strain measure in linear isotropic elasticity}
\label{section:euclideanStrainMeasureInLinearElasticity}
An approach similar to the definition of strain measures via distance functions on $\GLpn$, as stated in equation \eqref{eq:measureDefinedViaDistance}, can be employed in linearized elasticity theory: let $\varphi(x) = x+u(x)$ with the displacement $u$. Then the \emph{infinitesimal strain measure} may be obtained by taking the distance of the displacement gradient $\grad u \in \Rnn$ to the set of \emph{linearized rotations} $\son = \{A\in\Rnn : A^T = -A\}$, which is the vector space\footnote{Note that $\son$ also corresponds to the Lie algebra of the special orthogonal group $\SOn$.} of skew symmetric matrices. An obvious choice for a distance measure on the linear space $\Rnn \cong \R^{n^2}$ of $n\times n$-matrices is the \emph{Euclidean distance} induced by the canonical Frobenius norm
\[
	\norm{X} = \sqrt{\tr(X^TX)} = \sqrt{\smash{\sum_{i,j=1}^n}\vphantom{\sum^n} X_{ij}^2\,.} \vphantom{\sum_{i,j=1}^n}
\]
We use the more general weighted norm defined by\label{sectionContains:weightedNormIntroduction}
\begin{equation}
\label{eq:isoNormDefinition}
	\isonorm{X}^2 = \mu\,\norm{\dev_n\sym X}^2 + \muc\,\norm{\skew X}^2 + \frac\kappa2\, [\tr(X)]^2\,, \quad \mu,\muc,\kappa > 0\,,
\end{equation}
which separately weights the \emph{deviatoric} (or \emph{trace free}) \emph{symmetric part} $\dev_n\sym X = \sym X - \frac{1}{n} \tr(\sym X)\cdot \id$, the \emph{spherical part} $\frac{1}{n} \tr(X) \cdot \id$, and the \emph{skew symmetric part} $\skew X = \frac12(X-X^T)$ of $X$; note that $\isonorm{X} = \norm{X}$ for $\mu=\muc=1, \kappa=\frac2n$, and that $\isonorm{\,.\,}$ is induced by the inner product\footnote{%
The family \eqref{eq:isoProdDefinitions} of inner products on $\Rnn$ is based on the Cartan-orthogonal decomposition
\[
	\gln = \Big(\sln\cap\Symn\Big) \oplus \son \oplus \R\cdot\id
\]
of the Lie algebra $\gln=\Rnn$. Here, $\sln=\{X\in\gln \setvert \tr X = 0\}$ denotes the Lie algebra corresponding to the special linear group $\SLn=\{A\in\GLn \setvert \det A = 1\}$.
}
\begin{equation}
\label{eq:isoProdDefinitions}
	\isoprod{X,Y} = \mu\,\iprod{\dev_n\sym X, \dev_n\sym Y} + \muc\,\iprod{\skew X, \skew Y} + \tfrac{\kappa}{2}\tr(X) \tr(Y)
\end{equation}
on $\Rnn$, where $\iprod{X,Y}=\tr(X^TY)$ denotes the canonical inner product. In fact, every isotropic inner product on $\Rnn$, i.e.\ every inner product $\iprod{\cdot,\cdot}_{\mathrm{iso}}$ with
\[
	\iprod{Q^TX\,Q,\;Q^TY\,Q}_{\mathrm{iso}} = \iprod{X,Y}_{\mathrm{iso}}
\]
for all $X,Y\in\Rnn$ and all $Q\in\On$, is of the form \eqref{eq:isoProdDefinitions}, cf.\ \cite{Boor1985}. The suggestive choice of variables $\mu$ and $\kappa$, which represent the \emph{shear modulus} and the \emph{bulk modulus}, respectively, will prove to be justified later on. The remaining parameter $\muc$ will be called the \emph{spin modulus}.

Of course, the element of best approximation in $\son$ to $\grad u$ with respect to the weighted Euclidean distance $\fulldisteuclid(X,Y)=\isonorm{X-Y}$ is given by the associated orthogonal projection of $\grad u$ to $\son$, cf.\ Figure \ref{fig:euclideanDistanceToSkew}. Since $\son$ and the space $\Symn$ of symmetric matrices are orthogonal with respect to $\isoprod{\cdot,\cdot}$, this projection is given by the \emph{continuum rotation}, i.e.\ the skew symmetric part $\skew \grad u = \frac12(\grad u - (\grad u)^T)$ of $\grad u$, the axial vector of which is $\curl u$. Thus the distance is\footnote{
	The distance can also be computed directly: since
	\begin{align*}
		\isonorm{\grad u - A}^2 &= \mu\,\norm{\dev_n\sym (\grad u - A)}^2 + \muc\,\norm{\skew (\grad u - A)}^2 + \frac\kappa2\, [\tr(\grad u - A)]^2\\
		&= \mu\,\norm{\dev_n\sym \grad u}^2 + \muc\,\norm{(\skew \grad u) - A}^2 + \frac\kappa2\, [\tr(\grad u)]^2
	\end{align*}
	for all $A\in\son$, the infimum $\inf\limits_{A\in\son}\isonorm{\grad u - A}=\mu\,\norm{\dev_n\sym \grad u}^2 + \frac\kappa2\, [\tr(\grad u)]^2$ is obviously uniquely attained at $A=\skew\grad u$.
}
\begin{align}
	\fulldisteuclid(\grad u, \son) \ratio&= \inf_{A\in\son} \isonorm{\grad u - A}\nnl
	&= \isonorm{\grad u - \skew \grad u} = \isonorm{\sym \grad u}\,.\label{eq:skewSymmetricPartAsArgmin}
\end{align}
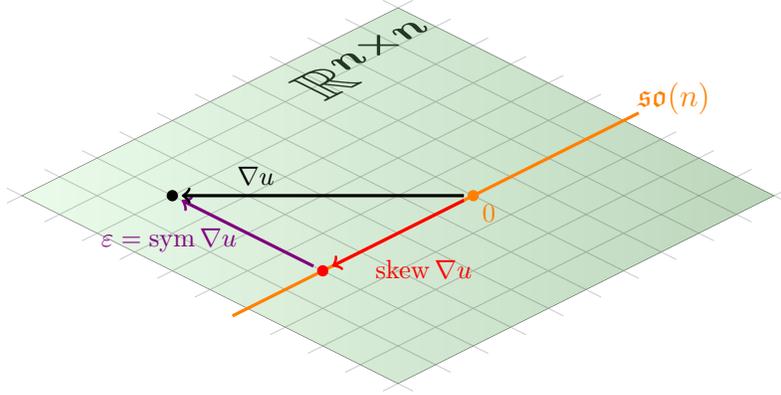
\begin{figure}
	\centering
	\tikzsetnextfilename{euclideanDistanceToSkew}
	\begin{tikzpicture}
		\input{tikz/euclideanDistanceToSkew.tex}
	\end{tikzpicture}
	\caption{\label{fig:euclideanDistanceToSkew}The Euclidean distance $\fulldisteuclid^2(\grad u, \son)=\mu\,\norm{\dev_n \eps}^2 + \frac\kappa2\, [\tr(\eps)]^2$ of $\grad u$ to $\son$ in $\Rnn$ in the infinitesimal strain setting. The strain tensor $\eps=\sym\grad u$ is orthogonal to the infinitesimal continuum rotation $\skew\grad u$.}
\end{figure}
\noindent We therefore find
\begin{align*}
	\fulldisteuclid^2(\grad u, \son) &= \isonorm{\sym \grad u}^2\\
	&= \mu\,\norm{\dev_n \sym \grad u}^2 + \frac\kappa2\, [\tr(\sym \grad u)]^2\\
	&= \mu\,\norm{\dev_n \eps}^2 + \frac\kappa2\, [\tr(\eps)]^2 = \Wlin(\grad u)
\end{align*}
for the linear strain tensor $\eps=\sym\grad u$, which is the quadratic isotropic elastic energy, i.e.\ the canonical model of isotropic linear elasticity with
\begin{equation}
	\sigma = D_{\grad u}\Wlin(\grad u) = 2\mu\,\dev_n\eps + \kappa\tr(\eps)\cdot\id\,.
\end{equation}
This shows the aforementioned close connection of the energy potential to geometrically motivated measures of strain. Note also that the so computed distance to $\son$ is independent of the parameter $\muc$, the \emph{spin modulus}, weighting the skew-symmetric part in the quadratic form \eqref{eq:isoNormDefinition}. We will encounter the (lack of) influence of the parameter $\muc$ subsequently again.

Furthermore, this approach \emph{motivates} the symmetric part $\eps=\sym\grad u$ of the displacement gradient as the strain tensor in the linear case: instead of \emph{postulating} that our strain measure should depend only on $\eps$, the above computations \emph{deductively characterize} $\eps$ as the infinitesimal strain tensor from simple geometric assumptions alone.

\subsection{The Euclidean strain measure in nonlinear isotropic elasticity}
In order to obtain a strain measure in the geometrically nonlinear case, we must compute the distance
\[
	\dist(\grad\varphi, \SO(n)) = \dist(F, \SO(n)) = \inf_{Q\in\SOn} \dist(F,Q)
\]
of the deformation gradient $F=\grad\varphi\in\GLpn$ to the actual set of pure rotations $\SOn\subset\GLpn$. It is therefore necessary to choose a distance function on $\GLpn$; an obvious choice is the restriction of the Euclidean distance on $\Rnn$ to $\GLpn$. For the canonical Frobenius norm $\norm{\,.\,}$, the Euclidean distance between $F,P\in\GLpn$ is
\[
	\disteuclid(F,P) = \norm{F-P} = \sqrt{\tr[(F-P)^T(F-P)]}\,.
\]
Now let $Q\in\SOn$. Since $\norm{\,.\,}$ is orthogonally invariant, i.e.\ $\norm{\Qhat X}=\norm{X\Qhat}=\norm{X}$ for all $X\in\Rnn$, $\Qhat\in\On$, we find
\begin{equation}
	\disteuclid(F,Q) = \norm{F-Q} = \norm{Q^T (F-Q)} = \norm{Q^TF-\id}\,.
\end{equation}
Thus the computation of the strain measure induced by the Euclidean distance on $\GLpn$ reduces to the \emph{matrix nearness problem} \cite{higham1988matrix}
\[
	\disteuclid(F, \SOn) = \inf_{Q\in\SOn} \norm{F-Q} = \min_{Q\in\SOn} \norm{Q^TF-\id}\,.
\]
By a well-known optimality result discovered by Giuseppe Grioli \cite{Grioli40} (cf.\ \cite{grioli2013, grioli1962equilibrium, martins1979variational, bouby2005direct}), also called \enquote{Grioli's Theorem} by Truesdell and Toupin \cite[p.~290]{truesdell60}, this minimum is attained for the orthogonal polar factor $R$.\label{sectionContains:griolisTheorem}
\begin{theorem}[Grioli's Theorem \cite{Grioli40,grioli2013,truesdell60}]\label{theorem:grioli} Let $F\in\GLpn$. Then
\[
	\min_{Q\in\SOn} \norm{Q^TF-\id} = \norm{R^TF-\id} = \norm{\sqrt{F^TF}-\id} = \norm{U-\id}\,,
\]
where $F=R\,U$ is the polar decomposition of $F$ with $R=\polar(F)\in\SOn$ and $U=\sqrt{F^TF}\in\PSymn$. The minimum is uniquely attained at the orthogonal polar factor $R$.
\end{theorem}
\begin{remark}
The minimization property stated in Theorem \ref{theorem:grioli} is equivalent to \cite{Guidugli80}
\[
	\max_{Q\in\SOn} \tr(Q^TF) = \max_{Q\in\SOn} \iprod{Q^TF,\id} = \iprod{R^TF,\id} = \iprod{U,\id}\,.\qedhere
\]
\end{remark}
\noindent Thus for nonlinear elasticity, the restriction of the Euclidean distance to $\GLpn$ yields the strain measure
\[
	\disteuclid(F, \SOn) = \norm{U-\id}\,.
\]
\begin{figure}[h]
	\centering
	\tikzsetnextfilename{flatGLn}
	\begin{tikzpicture}
		\input{tikz/flatGLn.tex}
	\end{tikzpicture}
	\caption{\label{fig:flatGLn}The \enquote{flat} interpretation of $\GLpn\subset\Rnn$ endowed with the Euclidean distance. Note that $\norm{F-R}=\norm{R\,(U-\id)}=\norm{U-\id}$ by orthogonal invariance of the Frobenius norm, where $F=R\,U$ is the polar decomposition of $F$.}
\end{figure}
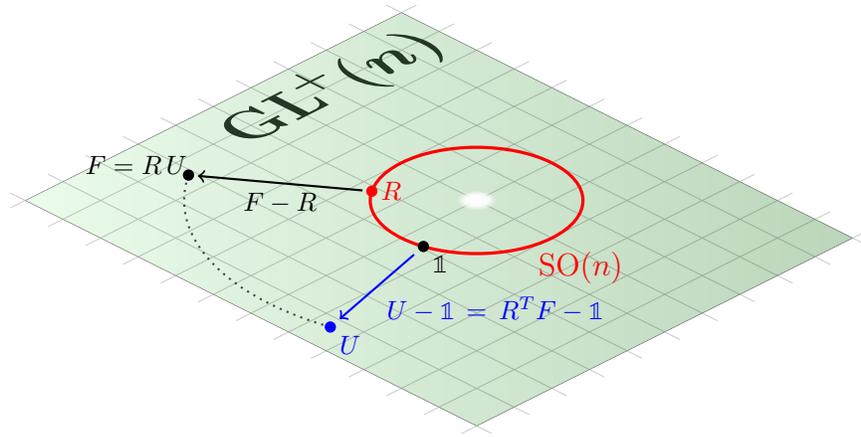
In analogy to the linear case, we obtain
\begin{equation}
	\disteuclid^2(F, \SOn) = \norm{U-\id}^2 = \norm{E_{\afrac12}}^2\,,
\end{equation}
where $E_{\afrac12}=U-\id$ is the Biot strain tensor. Note the similarity between this expression and the \emph{Saint-Venant-Kirchhoff} energy \cite{kirchhoff1852gleichungen}
\begin{equation}
\label{eq:SVKdefinition}
	\isonorm{E_1}^2 = \mu\,\norm{\dev_3 E_1}^2 + \frac\kappa2\, [\tr(E_1)]^2\,,
\end{equation}
where $E_1=\frac12(C-\id)=\frac12(U^2-\id)$ is the Green-Lagrangian strain.

The squared Euclidean distance of $F$ to $\SOn$ is often used as a lower bound for more general elastic energy potentials. Friesecke, James and M\"uller \cite{friesecke2002rigidity}, for example, show that if there exists a constant $C>0$ such that
\begin{equation}
\label{eq:euclideanDistanceAsLowerBoundForEnergy}
	W(F) \geq C \cdot \disteuclid^2(F,\SO(3))
\end{equation}
for all $F\in\GLp(3)$ in a large neighbourhood of $\id$, then the elastic energy $W$ shows some desirable properties which do not otherwise depend on the specific form of $W$. As a starting point for nonlinear theories of bending plates, Friesecke et al.\ also use the weighted squared norm
\[
	\isonorm{\sqrt{F^TF}-\id}^2 = \mu\,\norm{\dev_3(U-\id)}^2 + \frac\kappa2\,[\tr(U-\id)]^2 = \mu\,\norm{U-\id}^2 + \frac\lambda2\,[\tr(U-\id)]^2\,,
\]
where $\lambda$ is the first Lam\'e parameter, as an energy function satisfying \eqref{eq:euclideanDistanceAsLowerBoundForEnergy}. The same energy, also called the \emph{Biot energy} \cite{Neff_Biot07}, has been recently motivated by applications in digital geometry processing \cite{chao2010simple}.

However, the resulting strain measure $\measure(U)=\disteuclid(F, \SOn)=\norm{U-\id}$ does not truly seem appropriate for finite elasticity theory: for $U\to 0$ we find $\norm{U-\id}\to\norm{\id}=\sqrt{n}<\infty$, thus singular deformations do not necessarily correspond to an infinite measure $\omega$. Furthermore, the above computations are not compatible with the weighted norm introduced in Section \ref{sectionContains:weightedNormIntroduction}: in general \cite{Neff_Biot07,agn_fischle2015geometricallyI,agn_fischle2015geometricallyII},
\begin{equation}
	\min_{Q\in\SOn} \isonorm{F-Q}^2 \neq \min_{Q\in\SOn} \isonorm{Q^TF-\id}^2 \neq \isonorm{\sqrt{F^TF}-\id}^2\,, \label{eq:tensoriallyNotViable}
\end{equation}
thus the Euclidean distance of $F$ to $\SOn$ with respect to $\isonorm{\,.\,}$ does not equal $\isonorm{\sqrt{F^TF}-\id}$ in general. In these cases, the element of best approximation is not the  orthogonal polar factor $R=\polar(F)$.

In fact, the expression on the left-hand side of \eqref{eq:tensoriallyNotViable} is not even well defined in terms of linear mappings $F$ and $Q$ \cite{Neff_Biot07}: the deformation gradient $F=\grad\varphi$ at a point $x\in\Omega$ is a \emph{two-point tensor} and hence, in particular, a linear mapping between the tangent spaces $T_x\Omega$ and $T_{\varphi(x)}\varphi(\Omega)$. Since taking the norm
\[
	\isonorm{X} = \mu\,\norm{\dev_n\sym X}^2 + \muc\,\norm{\skew X}^2 + \frac\kappa2\,[\tr(X)]^2
\]
of $X$ requires the decomposition of $X$ into its symmetric and its skew symmetric part, it is only well defined if $X$ is an endomorphism on a single linear space.\footnote{If $X\col V_1\to V_2$ is a mapping between two different linear spaces $V_1,V_2$, then $X^T$ is a mapping from $V_2$ to $V_1$, hence $\sym X = \frac12\,(X+X^T)$ is not well-defined.} Therefore $\isonorm{F-Q}$, while being a valid expression for arbitrary \emph{matrices} $F,Q\in\Rnn$, is not an admissible term in the setting of finite elasticity.

We also observe that the Euclidean distance is not an \emph{intrinsic} distance measure on $\GLpn$: in general, $A-B\notin\GLpn$ for $A,B\in\GLpn$, hence the term $\norm{A-B}$ depends on the underlying linear structure of $\Rnn$. Since it is not a closed subset of $\Rnn$, $\GLpn$ is also not complete with respect to $\disteuc$; for example, the sequence $\left(\frac1n\cdot\id\right)_{n\in\N}$ is a Cauchy sequence which does not converge.

Most importantly, because $\GLpn$ is not convex, the straight line $\{A+t\,(B-A) \setvert t\in[0,1]\}$ connecting $A$ and $B$ is not necessarily contained\footnote{The straight line connecting $F\in\GLpn$ to its orthogonal polar factor $R$ (i.e.\ the shortest connecting line from $F$ to $\SOn$), however, lies in $\GLpn$, which easily follows from the convexity of $\PSymn$: for all $t\in[0,1]$,\: $t\,U+(1-t)\,\id\in\PSymn$ and thus \[R+t(F-R)=R\,(t\,U+(1-t)\,\id)\in R\cdot\PSymn\subset\GLpn\,.\]} in $\GLpn$, which shows that the characterization of the Euclidean distance as the length of a shortest connecting curve is also not possible in a way intrinsic to $\GLpn$, as the intuitive sketches\footnote{Note that the representation of $\GLpn$ as a sphere only serves to visualize the curved nature of the manifold and that further geometric properties of $\GLpn$ should not be inferred from the figures. In particular, $\GLpn$ is not compact and the geodesics are generally not closed.} in Figures \ref{fig:euclideanDistanceToSOn} and \ref{fig:euclideanAndRiemannianDistanceOnGLn} indicate.
\begin{figure}
	\centering
	\tikzsetnextfilename{euclideanDistanceToSOn}
	\begin{tikzpicture}
		\input{tikz/euclideanDistanceToSOn.tex}
	\end{tikzpicture}
	\caption{\label{fig:euclideanDistanceToSOn}The Euclidean distance as an extrinsic measure on $\GLpn$.}
\end{figure}
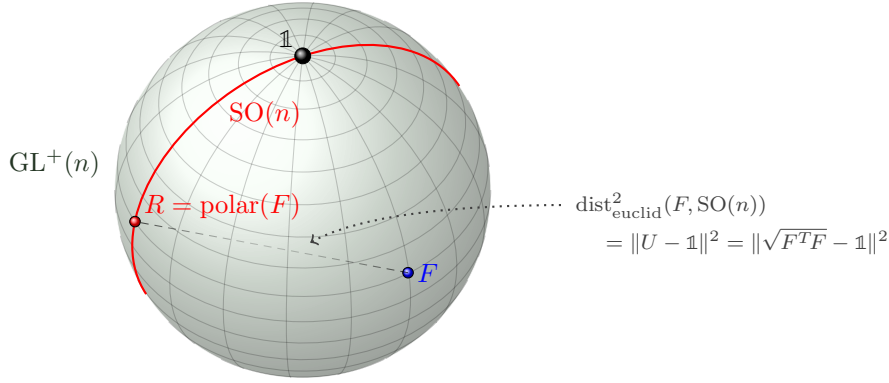

These issues amply demonstrate that the Euclidean distance can only be regarded as an \emph{extrinsic} distance measure on the general linear group. We therefore need to expand our view to allow for a more appropriate, truly \emph{intrinsic} distance measure on $\GLpn$.

\section{The Riemannian strain measure in nonlinear isotropic elasticity}
\label{section:riemannianStrainMeasureInNonlinearElasticity}
\subsection[$\GLpn$ as a Riemannian manifold]{\boldmath $\GLpn$ as a Riemannian manifold}
\label{section:manifoldIntroduction}
In order to find an intrinsic distance function on $\GLpn$ that alleviates the drawbacks of the Euclidean distance, we endow $\GLn$ with a \emph{Riemannian metric}.\footnote{For technical reasons, we define $g$ on all of $\GLn$ instead of its connected component $\GLpn$; for more details, we refer to \cite{agn_martin2014minimal}, where a more thorough introduction to geodesics on $\GLn$ can be found. Of course, our strain measure depends only on the restriction of $g$ to $\GLpn$.} Such a metric $g$ is defined by an inner product
\[
	g_A\col T_A\GLn \times T_A\GLn \to \R
\]
on each tangent space $T_A\GLn$, $A\in\GLn$. Then the length of a sufficiently smooth curve $\gamma\col[0,1]\to\GLn$ is given by
\[
	\len(\gamma) = \int_0^1 \sqrt{g_{\gamma(t)} (\gammadot(t),\gammadot(t))} \;\dt\,,
\]
where $\gammadot(t)=\ddt\,\gamma(t)$, and the \emph{geodesic distance} (cf.\ Figure \ref{fig:euclideanAndRiemannianDistanceOnGLn}) between $A,B\in\GLpn$ is defined as the infimum over the lengths of all (twice continuously differentiable) curves connecting $A$ to $B$:
\[
	\dg(A,B) = \inf\{\len(\gamma) \,|\, \gamma\in C^2([0,1];\GLpn),\; \gamma(0)=A,\; \gamma(1)=B\}\,.
\]
\begin{figure}[h]
	\centering
	\tikzsetnextfilename{euclideanAndRiemannianDistanceOnGLn}
	\begin{tikzpicture}
		\input{tikz/euclideanAndRiemannianDistanceOnGLn.tex}
	\end{tikzpicture}
	\caption{\label{fig:euclideanAndRiemannianDistanceOnGLn}The geodesic (intrinsic) distance compared to the Euclidean (extrinsic) distance.}
\end{figure}
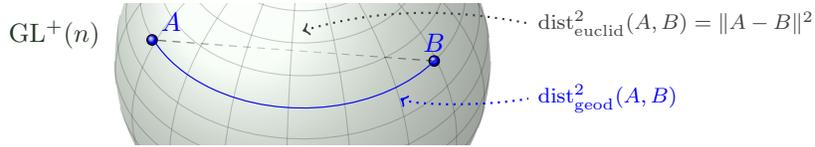
Our search for an appropriate strain measure is thereby reduced to the task of finding an appropriate Riemannian metric on $\GLn$. Although it might appear as an obvious choice, the metric $\gbad$ with
\begin{equation}
\label{eq:gbadDefinition}
	\gbad_A(X,Y) \colonequals \iprod{X,Y} \quad\text{ for all }A\in\GLpn,\; X,Y\in\Rnn
\end{equation}
provides no improvement over the already discussed Euclidean distance on $\GLpn$: since the length of a curve $\gamma$ with respect to $\gbad$ is the classical (Euclidean) length
\[
	\len(\gamma) = \int_0^1 \sqrt{\gbad_{\gamma(t)}(\gammadot(t),\gammadot(t))} \,\dt = \int_0^1 \norm{\gammadot(t)} \,\dt\,,
\]
the shortest connecting curves with respect to $\gbad$ are straight lines of the form $t\mapsto A+t(B-A)$ with $A,B\in\GLpn$. Locally, the geodesic distance induced by $\gbad$ is therefore equal to the Euclidean distance. However, as discussed in the previous section, not all straight lines connecting arbitrary $A,B\in\GLpn$ are contained within $\GLpn$, thus length minimizing curves with respect to $\gbad$ do not necessarily exist (cf.\ Figure \ref{fig:flatGLnEuclideanGeodesics}). Many of the shortcomings of the Euclidean distance therefore apply to the geodesic distance induced by $\gbad$ as well.
\begin{figure}[h]
	\centering
	\tikzsetnextfilename{flatGLnEuclideanGeodesics}
	\begin{tikzpicture}
		\input{tikz/flatGLnEuclideanGeodesics.tex}
	\end{tikzpicture}
	\caption{\label{fig:flatGLnEuclideanGeodesics}The shortest connecting (\emph{geodesic}) curves in $\GLpn$ with respect to the Euclidean metric are straight lines, thus not every pair $A,B\in\GLpn$ can be connected by curves of minimal length. The length of the straight line $\gamma\col t\mapsto A+t(B-A)$ connecting $A$ to $B$ is given by $\int_0^1 \sqrt{\gbad_{\gamma(t)}(\gammadot(t),\gammadot(t))}\,\dt=\norm{B-A}$, whereas the curve $\gammahat$ connecting $A$ to $C$ is not contained in $\GLpn$; its length is therefore not well defined.}
\end{figure}
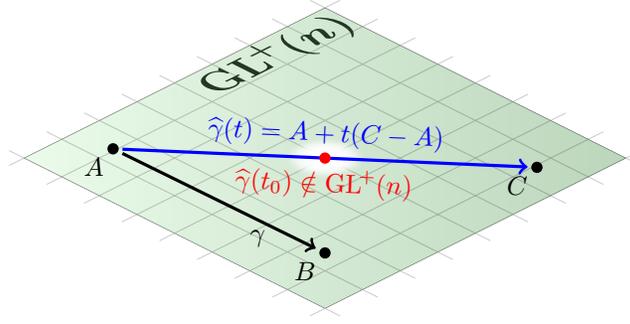

In order to find a more viable Riemannian metric $g$ on $\GLn$, we consider the mechanical interpretation of the induced geodesic distance $\dg$:
while our focus lies on the strain measure induced by $g$, that is the geodesic distance of the deformation gradient $F$ to the special orthogonal group $\SOn$, the distance $\dg(F_1,F_2)$ between two deformation gradients $F_1,F_2$ can also be motivated directly as a \emph{measure of difference} between two linear (or \emph{homogeneous}) deformations $F_1,F_2$ of the same body $\Omega$. More generally, we can define a difference measure between two inhomogeneous deformations $\varphi_1,\varphi_2:\Omega\subset \R^n\to\R^n$ via
\begin{equation}
\label{eq:deformationDistanceDefinition}
	\dist(\varphi_1,\varphi_2) \colonequals \int_\Omega \dg(\grad\varphi_1(x),\grad\varphi_2(x)) \,\dx
\end{equation}
under suitable regularity conditions for $\varphi_1,\varphi_2$ (e.g.\ if $\varphi_1,\varphi_2$ are sufficiently smooth with $\det\grad\varphi_i>0$ up to the boundary). This extension of the distance to inhomogeneous deformations is visualized in Figure \ref{fig:deformationDistanceExplanation}.
\begin{figure}[h]
	\centering
	\tikzsetnextfilename{deformationDistanceExplanation}
	\begin{tikzpicture}
		\input{tikz/deformationDistanceExplanation_alternative.tex}
	\end{tikzpicture}
	\caption{\label{fig:deformationDistanceExplanation}The distance $\dist(\varphi_1,\varphi_2)\colonequals\int_\Omega \dg(\grad\varphi_1(x),\grad\varphi_2(x))\,\dx$ measures how much two deformations $\varphi_1,\varphi_2$ of a body $\Omega$ differ from each other via integration over the pointwise geodesic distances between $\grad\varphi_1(x)$ and $\grad\varphi_2(x)$.}
\end{figure}
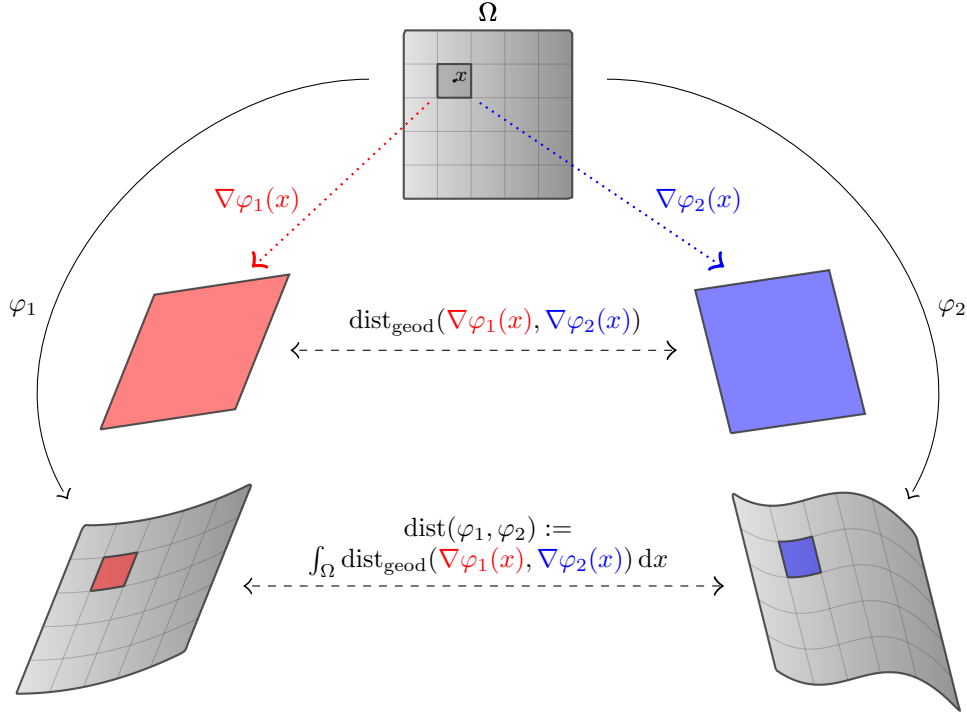

In order to find an appropriate Riemannian metric $g$ on $\GLn$, we must discuss the required properties of this \enquote{difference measure}. First, the requirements of objectivity (left-invariance) and isotropy (right-invariance) suggest that the metric $g$ should be \emph{bi-$\On$-invariant}, i.e.\ satisfy
\begin{equation}
\label{eq:OnBiInvariance}
	\mathrlap{\underbrace{g_{QA}(QX,QY) = g_A(X,Y)}_{\text{objectivity}}} \phantom{g_{QA}(QX,QY) =}\overbrace{\phantom{g_A(X,Y)} = g_{AQ}(XQ,YQ)}^{\smash{\text{isotropy}}}
\end{equation}
for all $Q\in\On$, $A\in\GLn$ and $X,Y\in T_A\GLn$, to ensure that $\dg(A,B)=\dg(Q\,A,Q\,B)=\dg(A\,Q,B\,Q)$.

However, these requirements do not sufficiently determine a specific Riemannian metric. For example, \eqref{eq:OnBiInvariance} is satisfied by the metric $\gbad$ defined in \eqref{eq:gbadDefinition} as well as by the metric $\gbadtoo$ with $\gbadtoo_A(X,Y)=\iprod{A^T\,X,A^T\,Y}$. In order to rule out unsuitable metrics, we need to impose further restrictions on $g$. If we consider the distance measure $\dist(\varphi_1,\varphi_2)$ between two deformations $\varphi_1,\varphi_2$ introduced in \eqref{eq:deformationDistanceDefinition}, a number of further invariances can be motivated: if we require that the distance is not changed by the superposition of a homogeneous deformation, i.e.\ that
\[
	\dist(B\cdot\varphi_1,B\cdot\varphi_2) = \dist(\varphi_1,\varphi_2)
\]
for all constant $B\in\GLn$, then $g$ must be \emph{left-$\GLn$-invariant}, i.e.
\begin{equation}
	g_{BA}(B\,X,B\,Y) = g_A(X,Y)
\end{equation}
for all $A,B\in\GLn$ and $X,Y\in T_A\GLn$. The physical interpretation of this invariance requirement is readily visualized in Figure \ref{fig:invarianceExplanation}.
\begin{figure}[h]
	\centering
	\tikzsetnextfilename{invarianceExplanation}
	\begin{tikzpicture}
		\input{tikz/invarianceExplanation.tex}
	\end{tikzpicture}
	\caption{\label{fig:invarianceExplanation}The distance between two deformations should not be changed by the composition with an additional homogeneous transformation $B$: $\dist(\varphi_1,\varphi_2) = \dist(B\cdot\varphi_1, B\cdot\varphi_2)$.}
\end{figure}
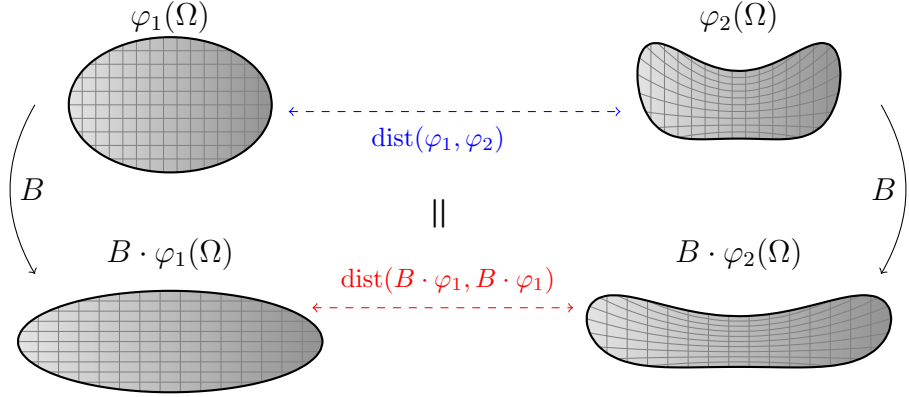

It can easily be shown \cite{agn_martin2014minimal} that a Riemannian metric $g$ is left-$\GLn$-invariant\footnote{Of course, the left-$\GLn$-invariance of a metric also implies the left-$\On$-invariance.} as well as right-$\On$-invariant if and only if $g$ is of the form
\begin{equation}
\label{eq:invariantRiemannianMetricDefinition}
	g_A(X,Y) = \isoprod{A\inv X, A\inv Y}\,,
\end{equation}
where $\isoprod{\cdot,\cdot}$ is the fixed inner product on the tangent space $\gln=T_\id\GLn=\Rnn$ at the identity with\label{sectionContains:MetricDefinition}
\begin{align}
	\isoprod{X,Y} &= \mu\,\iprod{\dev_n\sym X, \dev_n\sym Y} + \mu_c\iprod{\skew X, \skew Y} + \tfrac{\kappa}{2}\tr(X) \tr(Y) \label{eq:isotropicInnerProduct}%
\end{align}
for constant positive parameters $\mu,\muc,\kappa>0$, and where $\iprod{X,Y}=\tr(X^TY)$ denotes the canonical inner product on $\gln=\Rnn$.\footnote{If $\mu=\muc=1$ and $\kappa=\frac2n$, then the inner product $\isoprod{\cdot,\cdot}$ is the canonical inner product, and the corresponding metric $g$ is the \emph{canonical left-invariant metric} on $\GLn$ with $g_A(X,Y) = \iprod{A\inv X, A\inv Y}=\tr(X^TA^{-T}A\inv Y)$. Note that this metric differs from the \emph{trace metric} $\gtilde_A(X,Y)=\tr(A\inv X A\inv Y)$, cf.\ \cite{Dolcetti2016}.} A Riemannian metric $g$ defined in this way behaves in the same way on all tangent spaces: for every $A\in\GLpn$, $g$ transforms the tangent space $T_A\GLpn$ at $A$ to the tangent space $T_\id \GLpn=\gln$ at the identity via the left-hand multiplication with $A\inv$ and applies the fixed inner product $\isoprod{\cdot,\cdot}$ on $\gln$ to the transformed tangents, cf.\ Figure \ref{fig:tangentSpaceTransformation}.

In the following, we will always assume that $\GLn$ is endowed with a Riemannian metric of the form \eqref{eq:invariantRiemannianMetricDefinition} unless indicated otherwise.
\begin{figure}
	\centering
	\tikzsetnextfilename{tangentSpaceTransformation}
	\begin{tikzpicture}
		\input{tikz/tangentSpaceTransformation.tex}
	\end{tikzpicture}
	\caption{\label{fig:tangentSpaceTransformation}A left-$\GLn$-invariant Riemannian metric on $\GLn$ transforms the tangent space at $A\in\GLpn$ to the tangent space $T_\id \GLpn=\gln$ at the identity and applies a fixed inner product on $\gln$ to the transformed tangents. Thus no tangent space is treated preferentially.}
\end{figure}
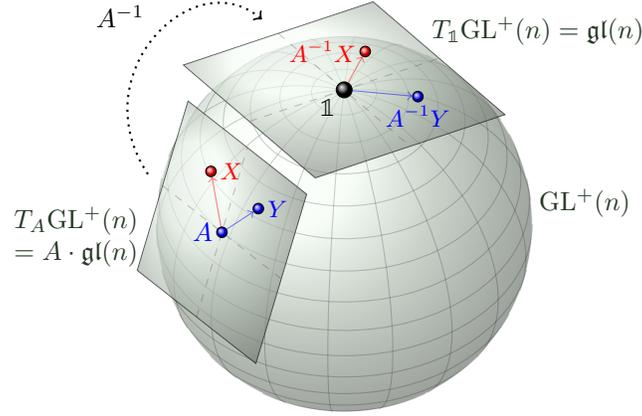

In order to find the geodesic distance
\[
	\dg(F,\SOn) = \inf_{Q\in\SOn} \dg(F,Q)
\]
of $F\in\GLpn$ to $\SOn$, we need to consider the \emph{geodesic curves} on $\GLpn$. It has been shown \cite{agn_martin2014minimal,Mielke2002,hackl2003dissipation,Andruchow2011} that every geodesic on $\GLpn$ with respect to the left-$\GLn$-invariant Riemannian metric induced by the inner product \eqref{eq:isotropicInnerProduct} is of the form %
\begin{equation}
\label{eq:generalFormOfGeodesics}
	\gamma_{F}^{\xi}(t) = F \, \exp(t(\sym\xi - \omegatfrac \skew\xi)) \, \exp(t(1+\omegatfrac)\skew\xi)
\end{equation}
with $F\in\GLpn$ and some $\xi\in\gln$, where $\exp$ denotes the matrix exponential.\footnote{The mapping $\xi\mapsto\exp_{\mathrm{geod}}(\xi) \colonequals \gamma_F^\xi(1) = F \, \exp(\sym\xi - \omegatfrac \skew\xi) \, \exp((1+\omegatfrac)\skew\xi)$ is also known as the \emph{geodesic exponential} function at $F$. Note that in general $\exp_{\mathrm{geod}}(\xi) \neq F\cdot\exp(\xi)$ if $\xi$ is not normal (i.e.\ if $\xi\xi^T\neq \xi^T\xi$), thus the geodesic curves are generally not one-parameter groups of the form $t\mapsto F\,\exp(t\,\xi)$, in contrast to bi-invariant metrics on Lie groups (e.g.\ $\SOn$ with the canonical bi-invariant metric \cite{Moakher2002}).} These curves are characterized by the \emph{geodesic equation}
\begin{align}
		\dot \zeta = \tfrac{\mu+\mu_c}{2\mu}(\zeta^T\zeta - \zeta\zeta^T)\,, \qquad
		\zeta \colonequals \gamma^{-1}\dot \gamma\,.
\end{align}
Since the geodesic curves are defined globally, $\GLpn$ is \emph{geodesically complete} with respect to the metric $g$. We can therefore apply the Hopf-Rinow theorem \cite{hopf1931begriff,agn_martin2014minimal} to find that for all $F,P\in\GLpn$ there exists a \emph{length minimizing geodesic $\gamma_{F}^{\xi}$} connecting $F$ and $P$. Without loss of generality, we can assume that $\gamma_{F}^{\xi}$ is defined on the interval $[0,1]$. Then the end points of $\gamma_{F}^{\xi}$ are
\[
	\gamma_{F}^{\xi}(0)=F \quad\text{ and }\quad P = \gamma_{F}^{\xi}(1) = F \, \exp(\sym\xi - \omegatfrac \skew\xi) \, \exp((1+\omegatfrac)\skew\xi)\,,
\]
and the length of the geodesic $\gamma_F^\xi$ starting in $F$ with initial tangent $F\,\xi\in T_F\GLpn$ (cf.\ \eqref{eq:generalFormOfGeodesics} and Figure \ref{fig:riemannianDistanceToSOn}) is given by \cite{agn_martin2014minimal} \label{sectionContains:ConsiderationsOnMinimizingGeodesics}
\[
	\len(\gamma_{F}^{\xi}) = \isonorm{\xi}\,.
\]
The geodesic distance between $F$ and $P$ can therefore be characterized as
\[
	\dg(F,P) = \min\{\isonorm{\xi} \setvert\; \xi\in\gln:\, \gamma_{F}^{\xi}(1) = P\}\,,
\]
that is the minimum of $\isonorm{\xi}$ over all $\xi\in\gln$ which connect $F$ and $P$, i.e.\ satisfy
\begin{equation}
\label{eq:geodesicEndPoint}
	\exp(\sym\xi - \omegatfrac \skew\xi) \, \exp((1+\omegatfrac)\skew\xi) = F\inv P\,.
\end{equation}

Although some numerical computations have been employed \cite{zacur2013multivariate} to approximate the geodesic distance in the special case of the canonical left-$\GLn$-invariant metric, i.e.\ for $\mu=\muc=1$, $\kappa=\frac2n$, there is no known closed form solution to the highly nonlinear system \eqref{eq:geodesicEndPoint} in terms of $\xi$ for given $F,P\in\GLpn$ and thus no known method of directly computing $\dg(F,P)$ in the general case exists. However, this parametrization of the geodesic curves will still allow us to obtain a lower bound on the distance of $F$ to $\SOn$.

\subsection[The geodesic distance to $\SOn$]{\boldmath The geodesic distance to $\SOn$}
\label{section:mainResults}
Having defined the geodesic distance on $\GLpn$, we can now consider the geodesic strain measure, which is the geodesic distance of the deformation gradient $F$ to $\SOn$:
\begin{equation}
	\dg(F,\SOn) = \inf_{Q\in\SOn} \dg(F,Q)\,.
\end{equation}

Without explicit computation of this distance, the left-$\GLn$-invariance and the right-$\On$-invariance of the metric $g$ immediately allow us to show the \emph{inverse deformation symmetry} of the geodesic strain measure:
\begin{alignat}{2}
	\dg(F,\SOn) &= \inf_{Q\in\SOn} \dg(F,Q) &&= \inf_{Q\in\SOn} \dg(F\inv F,F\inv Q)\nnl
	&= \inf_{Q\in\SOn} \dg(\id,F\inv Q)\; &&= \inf_{Q\in\SOn} \dg(Q^TQ,F\inv Q)\nnl
	&= \inf_{Q\in\SOn} \dg(Q^T,F\inv) &&= \;\dg(F\inv,\SOn)\,. \label{eq:inverseDeformationSymmetry}
\end{alignat}
This symmetry property demonstrates at once that the \emph{Eulerian} (spatial) and the \emph{Lagrangian} (referential) points of view are equivalent with respect to the geodesic strain measure: in the Eulerian setting, the inverse $F\inv$ of the deformation gradient appears more naturally\footnote{Note that Cauchy originally introduced the tensors $C\inv$ and $B\inv$ in his investigations of the nonlinear strain \cite{cauchy1827strain,cauchy1841memoire,freed2014soft,rougee1997mecanique}, where $C=F^TF=U^2$ is the right Cauchy-Green deformation tensor \cite{green1841propagation,freed2014soft} and $B=FF^T=V^2$ is the Finger tensor. Piola also formulated an early nonlinear elastic law in terms of $C\inv$, cf.\ \cite[p.\ 347]{truesdell65}.}, whereas $F$ is used in the Lagrangian frame (cf.\ Figure \ref{fig:lagrangeEuler}). Equality \eqref{eq:inverseDeformationSymmetry} shows that both points of view can equivalently be taken if the geodesic strain measure is used. As we will see later on (Remark \ref{remark:mainResultLeftBiotStretchFormulation}), the equality $\dg(B,\SOn)=\dg(C,\SOn)$ also holds for the right Cauchy-Green deformation tensor $C=F^TF=U^2$ and the Finger tensor $B=FF^T=V^2$, further indicating the independence of the geodesic strain measure from the chosen frame of reference. This property is, however, not unique to geodesic (or logarithmic) strain measures; for example, the Frobenius norm
\[
	\norm{\widetilde{E}_{\afrac12}(U)} = \tfrac12\norm{U-U\inv} = \tfrac12\norm{V-V\inv}
\]
of the Ba\v{z}ant approximation $\widetilde{E}_{\afrac12} = \frac12\,(U-U\inv)$, cf.\ \eqref{eq:bazantDefinition} and \cite{bazant1998}, which can be considered a \enquote{quasilogarithmic} strain measure, fulfils the inverse deformation symmetry as well.\footnote{The quantity $\frac{1}{\sqrt{2}}\norm{U-U\inv}$ is suggested as a measure of \emph{strain magnitude} by Truesdell and Toupin \cite[p.\ 266]{truesdell60}.} However, it is not satisfied for the Euclidean distance to $\SOn$: in general,
\begin{equation}
	\norm{U-\id} = \disteuc(F,\SOn) \neq \disteuc(F\inv,\SOn) = \norm{V\inv - \id}\,.
\end{equation}

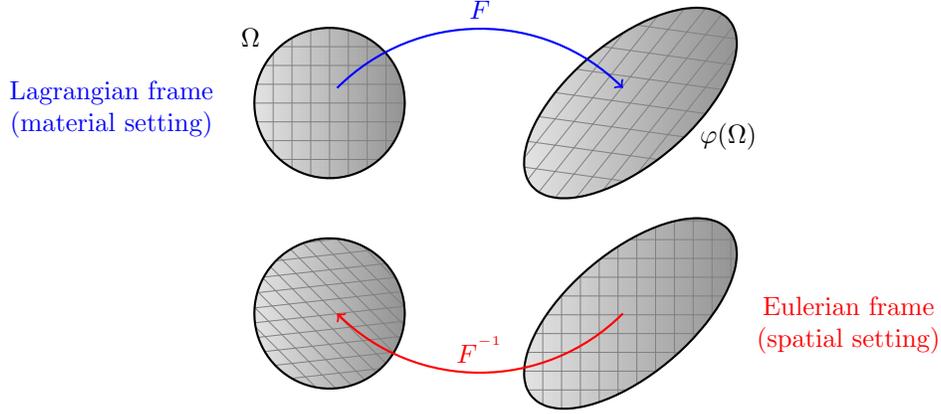
\begin{figure}
	\centering
	\tikzsetnextfilename{lagrangeEuler}
	\begin{tikzpicture}
		\input{tikz/lagrangeEuler.tex}
	\end{tikzpicture}
	\caption{\label{fig:lagrangeEuler}The Lagrangian and the Eulerian point of view are equivalently represented by the geodesic strain measure: $\dg(F,\SOn) = \dg(F\inv,\SOn)$.}
\end{figure}

Now, let $F=R\,U$ denote the polar decomposition of $F$ with $U\in\PSymn$ and $R\in\SOn$. In order to establish a simple upper bound on the geodesic distance $\dg(F,\SOn)$, we construct a particular curve $\gamma_R$ connecting $F$ to its orthogonal factor $R\in\SOn$ and compute its length $\len(\gamma_R)$. For
\[
	\gamma_R(t) \colonequals R\,\exp((1-t)\,\log U)\,,
\]
where $\log U\in\Symn$ is the principal matrix logarithm of $U$, we find
\[
	\gamma_R(0) = R\,\exp(\log U) = R\,U = F \qquad\text{ and }\qquad \gamma_R(1) = R\,\exp(0) = R \in \SOn\,.
\]
It is easy to confirm that $\gamma_R$ is in fact a geodesic as given in \eqref{eq:generalFormOfGeodesics} with $\xi=\log U\in\Symn$.
Since
\[
	\gamma_R\inv(t) \gammadot_R(t) = (R\,\exp((1-t)\,\log U))\inv \;R\,\exp((1-t)\,\log U)\cdot(-\log U) = -\log U\,,
\]
the length of $\gamma_R$ is given by
\begin{align}
	\len(\gamma_R) &= \int_0^1 \sqrt{g_{\gamma_R(t)} (\gammadot_R(t),\gammadot_R(t))} \;\dt\\
	&= \int_0^1 \sqrt{\isoprod{\gamma_R(t)\inv \gammadot_R(t), \gamma_R(t)\inv \gammadot_R(t)}} \;\dt\nnl
	&= \int_0^1 \sqrt{\isoprod{-\log U, -\log U}} \;\dt = \int_0^1 \isonorm{\log U} \,\dt = \isonorm{\log U}\,.\nonumber
\end{align}
We can thereby establish the \emph{upper bound}
\begin{align}
	\dg^2(F,\SOn) &= \inf_{Q\in\SOn} \dg^2(F,Q) \leq \dg^2(F,R) \label{eq:upperBoundFirstInequality}\\
	&\leq \len^2(\gamma_R) = \isonorm{\log U}^2 = \mu\,\norm{\dev_n\log U}^2 + \frac{\kappa}{2}\,[\tr(\log U)]^2 \label{eq:upperBound}
\end{align}
for the geodesic distance of $F$ to $\SOn$.

Our task in the remainder of this section is to show that the right hand side of inequality \eqref{eq:upperBound} is \emph{also a lower bound} for the (squared) geodesic strain measure, i.e.\ that, altogether,
\[
	\dg^2(F,\SOn) = \mu\,\norm{\dev_n\log U}^2 + \frac{\kappa}{2}\,[\tr(\log U)]^2\,.
\]

However, while the orthogonal polar factor $R$ is the element of best approximation in the Euclidean case (for $\mu=\mu_c=1$, $\kappa=\frac2n$) due to Grioli's Theorem, it is not clear whether $R$ is indeed the element in $\SOn$ with the shortest geodesic distance to $F$ (and thus whether equality holds in \eqref{eq:upperBoundFirstInequality}). Furthermore, it is not even immediately obvious that the geodesic distance between $F$ and $R$ is actually given by the right hand side of \eqref{eq:upperBound}, since a shorter connecting geodesic might exist (and hence inequality might hold in \eqref{eq:upperBound}).

Nonetheless, the following fundamental logarithmic minimization property\footnote{Of course, the application of such minimization properties to elasticity theory has a long tradition: Leonhard Euler, in the appendix \enquote{De curvis elasticis} to his 1744 book \enquote{Methodus inveniendi lineas curvas maximi minimive proprietate gaudentes sive solutio problematis isoperimetrici latissimo sensu accepti} \cite{euler1774methodus,oldfather1933leonhard}, already proclaimed that \quoteref{\quoteesc{\ldots} since the fabric of the universe is most perfect, and is the work of a most wise creator, nothing whatsoever takes place in the universe in which some rule of maximum and minimum does not appear.}} of the orthogonal polar factor, combined with the computations in Section \ref{section:manifoldIntroduction}, allows us to show that \eqref{eq:upperBound} is indeed also a lower bound for $\dg(F,\SOn)$.
\begin{proposition}
\label{prop:infLog}
Let $F=R\,\sqrt{F^TF}$ be the polar decomposition of $F\in\GLpn$ with $R\in\SOn$ and let $\norm{\,.\,}$ denote the Frobenius norm on $\Rnn$. Then
\[
	\inf_{Q\in\SOn} \norm{\sym \Log (Q^TF)} = \norm{\sym \log (R^TF)} = \norm{\log \sqrt{F^TF}}\,,
\]
where
\[
	\inf_{Q\in\SOn} \norm{\sym \Log (Q^TF)} \colonequals \inf_{Q\in\SOn} \inf \{\norm{\sym X} \;\setvert\; X\in\Rnn\,, \; \exp(X)=Q^TF\}
\]
is defined as the infimum of $\norm{\sym \,.\,}$ over \enquote{all real matrix logarithms} of $Q^TF$.
\end{proposition}
Proposition \ref{prop:infLog}, which can be seen as the natural logarithmic analogue of Grioli's Theorem (cf.\ Section \ref{sectionContains:griolisTheorem}), was first shown for dimensions $n=2,3$ by Neff et al.\ \cite{agn_neff2014logarithmic} using the so-called sum-of-squared-logarithms inequality \cite{agn_birsan2013sum,agn_pompe2015generalised,agn_dannan2014sum,agn_borisov2015sum}. A generalization to all unitarily invariant norms and complex logarithms for arbitrary dimension was given by Lankeit, Neff and Nakatsukasa \cite{agn_lankeit2014minimization}. We also require the following corollary involving the weighted Frobenius norm, which is not orthogonally invariant.\footnote{While $\isonorm{Q^TXQ}=\isonorm{X}$ for all $X\in\Rnn$ and $Q\in\On$, the orthogonal invariance requires the equalities $\isonorm{QX}=\isonorm{XQ}=\isonorm{X}$, which do not hold in general.}
\begin{corollary}
\label{cor:infLogWeighted}
Let
\[
	\isonorm{X}^2 = \mu\,\norm{\dev_n\sym X}^2 + \muc\,\norm{\skew X}^2 + \frac\kappa2\, [\tr(X)]^2\,, \qquad \mu,\muc,\kappa>0\,,
\]
for all $X\in\Rnn$, where $\norm{\,.\,}$ is the Frobenius matrix norm. Then
\[
	\inf_{Q\in\SOn} \isonorm{\sym \Log (Q^TF)} = \isonorm{\log \sqrt{F^TF}}\,.
\]
\end{corollary}
\begin{proof}
We first note that the equality $\det\exp(X)=e^{\tr(X)}$ holds for all $X\in\Rnn$. Since $\det Q = 1$ for all $Q\in\SOn$, this implies that for all $X\in\Rnn$ with $\exp(X)=Q^TF$,
\[
	\tr(\sym X) = \tr(X) = \ln(\det(\exp(X))) = \ln(\det(Q^TF)) = \ln(\det F)\,.
\]
Therefore\footnote{Observe that $\mu\,\norm{\dev_n Y}^2 + \frac\kappa2\, [\tr(Y)]^2 = \mu\,\norm{Y}^2 + \frac{n\,\kappa - 2\mu}{2n}\, [\tr(Y)]^2$ for all $Y\in\Rnn$.}
\begin{align*}
	\isonorm{\sym X}^2 &= \mu\,\norm{\dev_n \sym X}^2 + \frac\kappa2\, [\tr(\sym X)]^2\\
	&= \mu\,\norm{\sym X}^2 + \frac{n\,\kappa - 2\mu}{2n}\, [\tr(\sym X)]^2 = \mu\,\norm{\sym X}^2 + \frac{n\,\kappa - 2\mu}{2n}\, (\ln(\det F))^2
\end{align*}
and finally
\begin{align}
	&\hspace{-1.96em}\inf_{Q\in\SOn} \isonorm{\sym \Log (Q^TF)}^2\\
	&= \inf_{Q\in\SOn} \inf \{\isonorm{\sym X}^2 \setvert X\in\Rnn\,, \; \exp(X)=Q^TF\}\nnl
	&= \inf_{Q\in\SOn} \inf \{\mu\,\norm{\sym X}^2 + \frac{n\,\kappa - 2\mu}{2n} \,(\ln(\det F))^2 \setvert X\in\Rnn\,, \; \exp(X)=Q^TF\}\nnl
	&= \mu \inf_{Q\in\SOn} \inf \{\norm{\sym X}^2 \setvert X\in\Rnn\,, \; \exp(X)=Q^TF\} + \frac{n\,\kappa - 2\mu}{2n}\, (\ln(\det F))^2\nnl
	&= \mu \norm{\log \sqrt{F^TF}}^2 + \frac{n\,\kappa - 2\mu}{2n}\, (\ln(\det F))^2\nnl
	&= \mu \norm{\log \sqrt{F^TF}}^2 + \frac{n\,\kappa - 2\mu}{2n}\, [\tr(\log\sqrt{F^TF})]^2\nnl
	&= \mu\,\norm{\dev_n \log \sqrt{F^TF}}^2 + \frac\kappa2\, [\tr(\log \sqrt{F^TF})]^2 = \isonorm{\log \sqrt{F^TF}}^2\,. \nonumber\qedhere
\end{align}
\end{proof}
Note that Corollary \ref{cor:infLogWeighted} also implies the slightly weaker statement
\[
	\inf_{Q\in\SOn} \isonorm{\Log (Q^TF)} = \isonorm{\log \sqrt{F^TF}}
\]
by using the simple estimate $\isonorm{X}^2 \geq \isonorm{\sym X}^2$.\\[1em]
We are now ready to prove our main result.
\begin{theorem}
\label{theorem:mainResult}
Let $g$ be the left-$\GLn$-invariant, right-$\On$-invariant Riemannian metric on $\GL(n)$ defined by
\[
	g_A(X,Y) = \isoprod{A\inv X, A\inv Y}\,, \qquad \mu,\muc,\kappa>0\,,
\]
for $A\in\GLn$ and $X,Y\in\Rnn$, where
\begin{equation}
\label{eq:isoprodInMainResult}
	\isoprod{X,Y} = \mu\,\iprod{\dev_n\sym X, \dev_n\sym Y} + \mu_c\iprod{\skew X, \skew Y} + \tfrac{\kappa}{2}\tr(X) \tr(Y)\,.
\end{equation}
Then for all $F\in\GLpn$, the geodesic distance of $F$ to the special orthogonal group $\SOn$ induced by $g$ is given by
\begin{align}
	\dg^2(F,\SO(n)) = \mu\,\norm{\dev_n\log U}^2 + \frac{\kappa}{2}\,[\tr(\log U)]^2\,, \label{eq:mainResultEquation}
\end{align}
where $\log$ is the principal matrix logarithm, $\tr(X) = \sum_{i=1}^n X_{i,i}$ denotes the trace and $\dev_n X = X - \frac1n\tr(X)\cdot\id$ is the $n$-dimensional deviatoric part of $X\in\Rnn$. The orthogonal factor $R\in\SOn$ of the polar decomposition $F=R\,U$ is the unique element of best approximation in $\SOn$, i.e.
\[
	\dg(F,\SOn) = \dg(F,R) = \dg(R^TF,\id) = \dg(U,\id)\,.
\]
In particular, the geodesic distance does not depend on the \emph{spin modulus} $\muc$.
\end{theorem}
\begin{remark}[Uniqueness of the metric]
We remark once more that the Riemannian metric considered in Theorem \ref{theorem:mainResult} is not chosen arbitrarily: every left-$\GLn$-invariant, right-$\On$-invariant Riemannian metric on $\GL(n)$ is of the form given in \eqref{eq:isoprodInMainResult} for some choice of parameters $\mu,\muc,\kappa>0$ \cite{agn_martin2014minimal}.
\end{remark}
\begin{remark}\label{remark:mainResultLeftBiotStretchFormulation}
Since the weighted Frobenius norm on the right hand side of equation \eqref{eq:mainResultEquation} only depends on the eigenvalues of $U=\sqrt{F^TF}$, the result can also be expressed in terms of the left Biot-stretch tensor $V=\sqrt{FF^T}$, which has the same eigenvalues as $U$:
\begin{align}
	\dg^2(F,\SO(n)) = \mu\,\norm{\dev_n\log V}^2 + \frac{\kappa}{2}\,[\tr(\log V)]^2\,.
\end{align}
Applying the above formula to the case $F=P$ with $P\in\PSymn$, we find $\sqrt{P^TP}=\sqrt{PP^T}=P$ and therefore
\begin{equation}
\label{eq:logDistanceToIdFromPSym}
	\dist^2(P,\SOn)=\dist^2(P,\,\id)=\mu\,\norm{\dev_n\log P}^2 + \frac{\kappa}{2}\,[\tr(\log P)]^2\,,
\end{equation}
since $\id$ is the orthogonal polar factor of $P$. For the tensors $U$ and $V$, the right Cauchy-Green deformation tensor $C=F^TF=U^2$ and the Finger tensor $B=FF^T=V^2$, we thereby obtain the equalities
\begin{alignat}{4}
	&\dg(B,\SOn) &&= \dg(B,\id) &&= \dg(B\inv,\id)\\
	&&&= \dg(C,\id) &&= \dg(C\inv,\id) &&= \dg(C,\SOn)\nnl
	\text{and}\quad &\dg(V,\SOn) &&= \dg(V,\id) &&= \dg(V\inv,\id)\\
		&&&= \dg(U,\id) &&= \dg(U\inv,\id) &&= \dg(U,\SOn)\,. \nonumber
\end{alignat}
Note carefully that, although \eqref{eq:logDistanceToIdFromPSym} for $P\in\PSymn$ immediately follows from Theorem \ref{theorem:mainResult}, it is not trivial to compute the distance $\dg(P,\,\id)$ directly: while the curve given by $\exp(t\,\log P)$ for $t\in[0,1]$ is in fact a geodesic \cite{hackl2003dissipation} connecting $\id$ to $P$ with squared length $\mu\,\norm{\dev_n\log P}^2 + \frac{\kappa}{2}\,[\tr(\log P)]^2$, it is not obvious whether or not a shorter connecting geodesic might exist. Our result ensures that this is in fact not the case.
\end{remark}
\begin{proof}[Proof of Theorem \ref{theorem:mainResult}]
Let $F\in\GLpn$ and $\Qhat\in\SOn$. Then according to our previous considerations (cf.\ Section \ref{sectionContains:ConsiderationsOnMinimizingGeodesics}) there exists $\xi\in\gln$ with
\begin{equation}
\label{geodesicEndPointQ}
	\exp(\sym\xi - \omegatfrac \skew\xi) \, \exp((1+\omegatfrac)\skew\xi) = F\inv \Qhat
\end{equation}
and
\begin{equation}
\label{geodesicLengthQxi}
	\isonorm{\xi} = \dg(F,\Qhat)\,.
\end{equation}
In order to find a lower estimate on $\isonorm{\xi}$ (and thus on $\dg(F,\Qhat)$), we compute
\begin{alignat*}{2}
	&&\exp(\sym\xi - \omegatfrac \skew\xi) \, \exp((1+\omegatfrac)\skew\xi) &= F\inv \Qhat\\
	&\Longrightarrow\quad &\exp((1+\omegatfrac)\skew\xi)\inv \, \exp(\sym\xi - \omegatfrac \skew\xi)\inv &= \Qhat^TF\\
	&\Longrightarrow\quad &\exp(-\sym\xi + \omegatfrac \skew\xi) &= \exp(\,\underbrace{(1+\omegatfrac)\skew\xi}_{\in\son}\,)\,\Qhat^TF\,.
\end{alignat*}
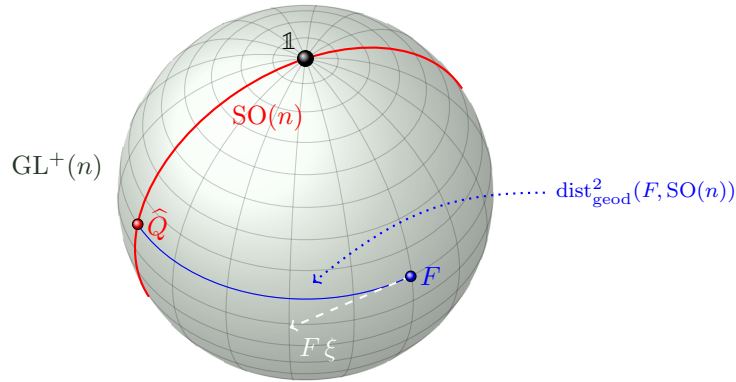
\begin{figure}[h]
	\centering
	\tikzsetnextfilename{riemannianDistanceToSOn}
	\begin{tikzpicture}
		\input{tikz/riemannianDistanceToSOn.tex}
	\end{tikzpicture}
	\caption{\label{fig:riemannianDistanceToSOn}The geodesic (intrinsic) distance to $\SOn$; neither the element $\Qhat$ of best approximation nor the initial tangent $F\,\xi\in T_F \GLpn$ of the connecting geodesic is known beforehand.}
\end{figure}
Since $\exp(W)\in\SOn$ for all skew symmetric $W\in\son$, we find
\begin{equation}
\label{eq:termIsLogOfQTxiF}
	\exp(\underbrace{-\sym\xi + \omegatfrac \skew\xi}_{\equalscolon Y}) = Q_\xi^TF
\end{equation}
with $Q_\xi = \Qhat\,\exp(-(1+\omegatfrac)\skew\xi\,) \in\SOn$; note that $\sym Y = -\sym\xi$. According to \eqref{eq:termIsLogOfQTxiF}, $Y=-\sym\xi + \omegatfrac \skew\xi$\quad is \enquote{a logarithm}\footnote{Loosely speaking, we use the term \enquote{a logarithm of $A\in\GLpn$} to denote any (real) solution $X$ of the matrix equation $\exp X = A$.} of $Q_\xi^TF$. The weighted Frobenius norm of the symmetric part of \;$Y=-\sym\xi + \omegatfrac \skew\xi$\quad is therefore bounded below by the infimum of $\isonorm{\sym X}$ over \enquote{all logarithms} $X$ of $Q_\xi^TF$:
\begin{align}
	\isonorm{\sym \xi} %
	&= \isonorm{\sym Y}\nnl
	&\overset{\mathclap{\eqref{eq:termIsLogOfQTxiF}}}{\geq} \hphantom{\inf_{Q\in\SOn}}\inf \{\isonorm{\sym X} \setvert X\in\Rnn\,, \; \exp(X)=Q_\xi^TF\}\nnl
	&\geq \inf_{Q\in\SOn} \inf \{\isonorm{\sym X} \setvert X\in\Rnn\,, \; \exp(X)=Q^TF\}\nnl
	&= \inf_{Q\in\SOn} \isonorm{\sym \Log (Q^TF)}\,.\label{eq:inequalityForSymXi}
\end{align}
We can now apply Corollary \ref{cor:infLogWeighted} to find
\begin{align}
	\dg^2(F,\Qhat) = \isonorm{\xi}^2 &= \mu\,\norm{\dev_n \sym \xi}^2 + \muc\,\norm{\skew \xi}^2 + \frac\kappa2\, [\tr(\sym \xi)]^2\nnl
	&\geq \mu\,\norm{\dev_n \sym \xi}^2 + \frac\kappa2\, [\tr(\sym \xi)]^2\\ %
	&= \isonorm{\sym \xi}^2\nnl
	&\overset{\mathclap{\eqref{eq:inequalityForSymXi}}}{\geq} \inf_{Q\in\SOn} \isonorm{\sym \Log (Q^TF)}^2\nnl
	&\overset{\mathclap{\text{Corollary \ref{cor:infLogWeighted}}}}{=} \qquad\mu\, \isonorm{\log \sqrt{F^TF}}^2 = \mu\,\norm{\dev_n\log U}^2 + \frac{\kappa}{2}\,[\tr(\log U)]^2\nonumber
\end{align}
for $U=\sqrt{F^TF}$. Since this inequality is independent of $\Qhat$ and holds for all $\Qhat\in\SOn$, we obtain the desired lower bound
\[
	\dg^2(F,\SOn) = \inf_{\Qhat\in\SOn} \dg^2(F,\Qhat) \geq \mu\,\norm{\dev_n\log U}^2 + \frac{\kappa}{2}\,[\tr(\log U)]^2
\]
on the geodesic distance of $F$ to $\SOn$. Together with the upper bound
\[
	\dg^2(F,\SOn) \leq \dg^2(F,R) \leq \mu\,\norm{\dev_n\log U}^2 + \frac{\kappa}{2}\,[\tr(\log U)]^2
\]
already established in \eqref{eq:upperBound}, we finally find
\begin{equation}
\label{eq:mainResultProofComputedDistance}
	\dg^2(F,\SOn) = \dg^2(F,R) = \mu\,\norm{\dev_n\log U}^2 + \frac{\kappa}{2}\,[\tr(\log U)]^2\,.
\end{equation}

By equation \eqref{eq:mainResultProofComputedDistance}, apart from computing the geodesic distance of $F$ to $\SOn$, we have shown that the orthogonal polar factor $R=\polar(F)$ is an element of best approximation to $F$ in $\SOn$. However, it is not yet clear whether there exists another element of best approximation, i.e.\ whether there is a $\Qhat\in\SOn$ with $\Qhat\neq R$ and $\dg(F,\Qhat)=\dg(F,R)=\dg(F,\SOn)$.
For this purpose, we need to compare geodesic distances corresponding to different parameters $\mu,\muc,\kappa$. We therefore introduce the following notation: for fixed $\mu,\muc,\kappa>0$, let $\dgparam$ denote the geodesic distance on $\GLpn$ induced by the left-$\GLn$-invariant, right-$\On$-invariant Riemannian metric $g$ (as introduced in \eqref{eq:invariantRiemannianMetricDefinition}) with parameters $\mu,\muc,\kappa$. Furthermore, the length of a curve $\gamma$ with respect to this metric will be denoted by $\lenparam(\gamma)$.

Assume that $\Qhat\in\SOn$ is an element of best approximation to $F$ with respect to $g$ for some fixed parameters $\mu,\muc,\kappa>0$. Then there exists a length minimizing geodesic $\gamma\col[0,1]\to\GLpn$ connecting $\Qhat$ to $F$ of the form
\[
	\gamma(t) = \Qhat \, \exp(t(\sym\xi - \omegatfrac \skew\xi)) \, \exp(t(1+\omegatfrac)\skew\xi)
\]
with $\xi\in\Rnn$, and the length of $\gamma$ is given by
\[
	\lenparam^2(\gamma) = \isonorm{\xi}^2 = \mu\,\norm{\dev_n\sym \xi}^2 + \muc\,\norm{\skew \xi}^2 + \frac\kappa2\, [\tr(\xi)]^2\,.
\]
We first assume that $\skew\xi\neq0$. We choose $\mutilde_c>0$ with $\mutilde_c<\muc$ and find %
\begin{align}
	\dgparamtilde^2(F,\SOn) &= \inf_{Q\in\SOn} \dgparamtilde^2(F,Q) \label{eq:uniquenessStepOne}\\
	&\leq \dgparamtilde^2(F,\Qhat) \;\leq\; \lenparamtilde^2(\gamma)\,, \nonumber
\end{align}
since $\gamma$ is a curve connecting $F$ to $\Qhat\in\SOn$; note that although $\gamma$ is a shortest connecting geodesic with respect to parameters $\mu,\muc,\kappa$ by assumption, it must not necessarily be a length minimizing curve with respect to parameters $\mu,\mutilde_c,\kappa$. Obviously, $\isonormtilde{\xi} < \isonorm{\xi}$ if $\skew\xi\neq0$, and therefore
\begin{equation}
\label{eq:uniquenessStepTwo}
	\lenparamtilde^2(\gamma) = \isonormtilde{\xi}^2 < \isonorm{\xi}^2 = \lenparam^2(\gamma) = \dgparam^2(F,\Qhat)\,.
\end{equation}
By assumption, $\Qhat$ is an element of best approximation to $F$ in $\SOn$ for parameters $\mu,\muc,\kappa$, thus
\begin{align}
	\dgparam^2(F,\Qhat) &= \dgparam^2(F,\SOn) \label{eq:uniquenessStepThree}\\
	&= \mu \, \norm{\dev_n\log U}^2+\frac{\kappa}{2}\,[\tr(\log U)]^2 = \dgparamtilde^2(F,\SOn)\,,\nonumber
\end{align}
where the last equality utilizes the fact that the distance from $F$ to $\SOn$ is independent of the second parameter ($\muc$ or $\mutilde_c$).
Combining \eqref{eq:uniquenessStepOne}, \eqref{eq:uniquenessStepTwo} and \eqref{eq:uniquenessStepThree}, we thereby obtain the contradiction
\[
	\dgparamtilde^2(F,\SOn) \leq \lenparamtilde^2(\gamma) < \dgparam^2(F,\Qhat) = \dgparamtilde^2(F,\SOn)\,,
\]
hence we must have $\skew\xi=0$. But then
\[
	\gamma(1) = \Qhat \, \exp(\sym\xi - \omegatfrac \skew\xi) \, \exp((1+\omegatfrac)\skew\xi) = \Qhat\,\exp(\sym\xi)\,,
\]
and since $\exp(\sym\xi)\in\PSymn$, the uniqueness of the polar decomposition $F=R\,U$ yields $\exp(\sym\xi)=U$ and, finally, $\Qhat=R$.
\end{proof}
The fact that the orthogonal polar factor $R=\polar(F)$ is the unique element of best approximation to $F$ in $\SOn$ with respect to the geodesic distance
corresponds directly to the linear case (cf.\ equality \eqref{eq:skewSymmetricPartAsArgmin} in Section \ref{section:euclideanStrainMeasureInLinearElasticity}), where the skew symmetric part $\skew\grad u$ of the displacement gradient $\grad u$ is the element of best approximation with respect to the Euclidean distance: for $F=\id+\grad u$ we have
\[
	U = \id + \sym\grad u + \mathcal{O}(\norm{\grad u}^2) \qquad\text{and}\qquad R = \id + \skew\grad u + \mathcal{O}(\norm{\grad u}^2)\,,
\]
hence the linear approximation of the orthogonal and the positive definite factor in the polar decomposition is given by $\skew\grad u$ and $\sym\grad u$, respectively. The geometric connection between the geodesic distance on $\GLpn$ and the Euclidean distance on the tangent space $\Rnn=\gl(n)$ at $\id$ is illustrated in Figure \ref{fig:allDistances}.
\begin{figure}[h]
	\centering
	\tikzsetnextfilename{allDistances}
	\begin{tikzpicture}
		\input{tikz/allDistances.tex}
	\end{tikzpicture}
	\caption{\label{fig:allDistances}The isotropic Hencky energy of $F$ measures the geodesic distance between $F$ and $\SOn$. The linear Euclidean strain measure is obtained as the linearization via the tangent space $\gln$ at $\id$.}
\end{figure}
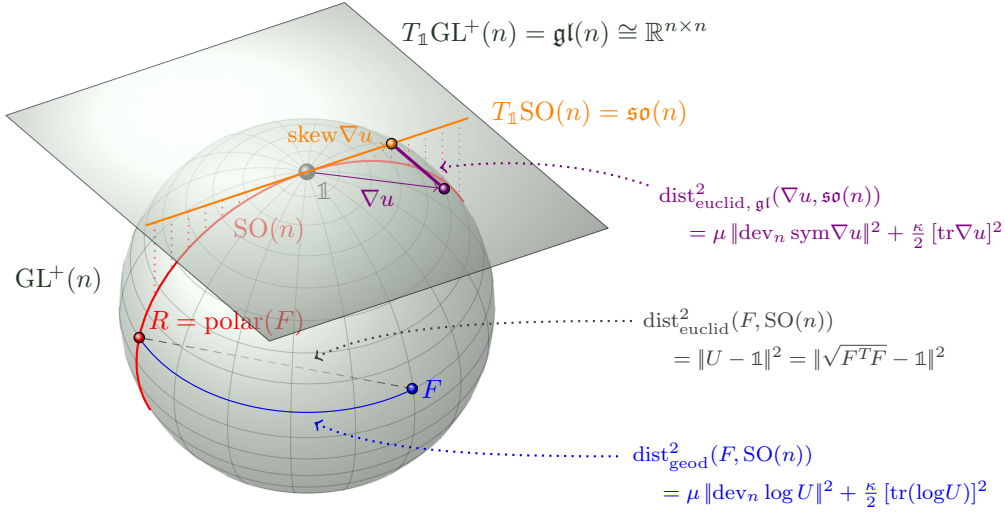
\begin{remark}
Using a similar proof, exactly the same result can be shown for the geodesic distance $\dgright$ induced by the \emph{right-$\GLn$-invariant, left-$\On$-invariant} Riemannian metric \cite{Vandereycken2010}
\[
	g^{\mathrm{right}}_A(X,Y) = \isoprod{XA\inv,YA\inv}
\]
on $\GLn$:
\[
	\dgright^2(F,\SOn) = \dg^2(F,\SOn) = \mu\,\norm{\dev_n\log U}^2 + \frac{\kappa}{2}\,[\tr(\log U)]^2\,.
\]
The right-$\GLn$-invariant Riemannian metric can be motivated in a way similar to the left-$\GLn$-invariant case: it corresponds to the requirement that the distance between two deformations $F_1$ and $F_2$ should not depend on the initial shape of $\Omega$, i.e.\ should not be changed if $\Omega$ is homogeneously deformed beforehand (cf.\ Figure \ref{fig:invarianceExplanationRightInvariance}). A similar independence from prior deformations (and so-called \enquote{pre-stresses}), called \enquote{elastic determinacy} by L.\ Prandtl \cite{prandtl1924}, was postulated by H.\ Hencky in the deduction of his elasticity model; cf.\ \cite[p.~618]{hencky1929super}, \cite[p.~19]{henckyTranslation} and Section \ref{section:mechanicalMotivations}.
\end{remark}
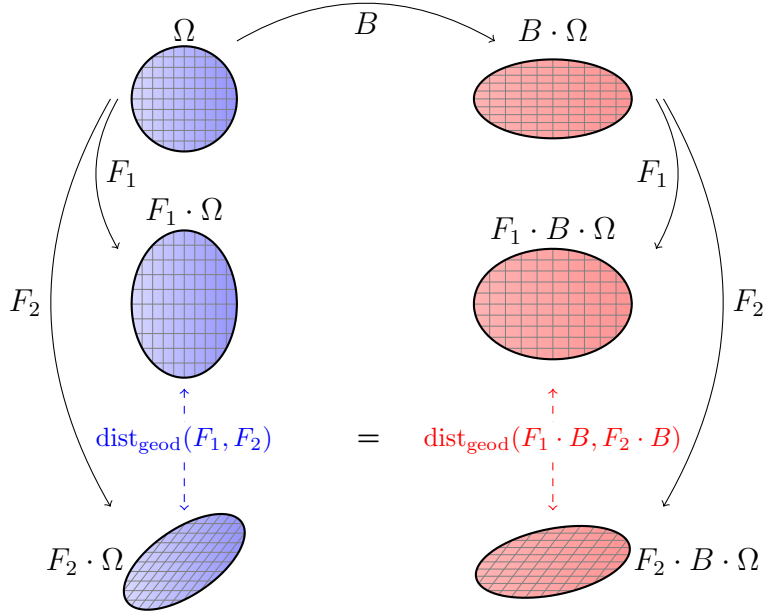
\begin{figure}
	\centering
	\tikzsetnextfilename{invarianceExplanationRightInvariance}
	\begin{tikzpicture}
		\input{tikz/invarianceExplanation_rightInvariance.tex}
	\end{tikzpicture}
	\caption{\label{fig:invarianceExplanationRightInvariance}The right-$\GLn$-invariance of a distance measure on $\GLn$: the distance between two homogeneous deformations $F_1,F_2$ is not changed by a prior homogeneous deformation $B$, i.e.\ $\dg(F_1,F_2)=\dg(F_1\cdot B, F_2\cdot B)$.}
\end{figure}
According to Theorem \ref{theorem:mainResult}, the squared geodesic distance between $F$ and $\SOn$ with respect to any left-$\GLn$-invariant, right-$\On$-invariant Riemannian metric on $\GLn$ is the \emph{isotropic quadratic Hencky energy}
\[
	\WH(F) = \mu\,\norm{\dev_n\log U}^2 + \frac{\kappa}{2}\,[\tr(\log U)]^2\,,
\]
where the parameters $\mu,\kappa>0$ represent the shear modulus and the bulk modulus, respectively. The Hencky energy function was introduced in 1929 by H.\ Hencky \cite{Hencky1929}, who derived it from geometrical considerations as well: his deduction\footnote{Hencky's approach is often misrepresented as empirically motivated. Truesdell claims that \quoteref{Hencky himself does not give a systematic treatement} in introducing the logarithmic strain tensor \cite[p.~144]{truesdell1952} and attributes the axiomatic approach to Richter \cite{richter1949verzerrung} instead \cite[p.~270]{truesdell60}. Richter's resulting deviatoric strain tensors $\dev_3 \log U$ and $\dev_3 \log V$ are disqualified as \quoteref{complicated algebraic functions} by Truesdell and Toupin \cite[p.~270]{truesdell60}.} was based on a set of axioms including a law of superposition (cf.\ Section \ref{section:mechanicalMotivations}) for the stress response function \cite{henckyTranslation}, an approach previously employed by G.\,F.\ Becker \cite{becker1893,neff2014becker} in 1893 and later followed in a more general context by H.\ Richter \cite{richter1949verzerrung}, cf.\ \cite{richter1949log,richter1948,richter1952elastizitatstheorie}. A different constitutive model for uniaxial deformations based on logarithmic strain had previously been proposed by Imbert \cite{imbert1880} and Hartig \cite{hartig1893}. While Ludwik is often credited with the introduction of the uniaxial logarithmic strain, his ubiquitously cited article \cite{ludwik1909} (which is even referenced by Hencky himself \cite[p.\ 175]{hencky1931}) does not provide a systematic introduction of such a strain measure.\vspace*{.5em}

While the energy function $\WH(F)=\dg^2(F,\SOn)$ already defines a measure of strain as described in Section \ref{sectionContains:energyAsStrainMeasure}, we are also interested in characterizing the two terms $\norm{\dev_n\log U}$ and $\abs{\tr(\log U)}$ as separate \emph{partial strain measures}.
\begin{theorem}[Partial strain measures]
\label{theorem:separateMeasures}
Let
\[
	\isomeas(F)\colonequals\norm{\dev_n\log\sqrt{F^TF}} \quad\text{ and }\quad \volmeas(F)\colonequals\abs{\tr(\log \sqrt{F^TF})}\,.
\]
Then
\begin{align*}
	\isomeas(F) &= \dgiso \left( \frac{F}{\det F^{\afrac1n}},\; \SOn \right)
\intertext{and}
	\volmeas(F) &= \sqrt{n}\cdot \dgvol \left( (\det F)^{\afrac1n}\cdot\id,\; \id \right)\,,
\end{align*}
where the geodesic distances $\dgiso$ and $\dgvol$ on the Lie groups $\SLn=\{A\in\GLn \setvert \det A = 1\}$ and $\R^+\cdot\id$ are induced by the canonical left-invariant metric
\[
	\bar{g}_A(X,Y) = \innerproduct{A\inv X, A\inv Y} = \tr(X^TA^{-T}A\inv Y)\,.
\]
\end{theorem}
\begin{remark}
Theorem \ref{theorem:separateMeasures} states that $\isomeas$ and $\volmeas$ appear as natural measures of the \emph{isochoric} and \emph{volumetric} strain, respectively: if $F=\Fiso\, \Fvol$ is decomposed multiplicatively \cite{flory1961thermodynamic} into an isochoric part $\Fiso=(\det F)^{-\afrac1n}\cdot F$ and a volumetric part $\Fvol=(\det F)^{\afrac1n}\cdot\id$, then $\isomeas(F)$ measures the $\SLn$-geodesic distance of $\Fiso$ to $\SOn$, whereas $\frac{1}{\sqrt{n}}\,\volmeas(F)$ gives the geodesic distance of $\Fvol$ to the identity $\id$ in the group $\R^+\cdot\id$ of purely volumetric deformations.
\end{remark}
\begin{proof}
First, observe that the canonical left-invariant metrics on $\SLn$ and $\R^+\cdot\id$ are obtained by choosing $\mu=\muc=1$ and $\kappa=\frac2n$ and restricting the corresponding metric $g$ on $\GLpn$ to the submanifolds $\SLn$, $\R^+\cdot\id$ and their respective tangent spaces. Then for this choice of parameters, every curve in $\SLn$ or $\R^+\cdot\id$ is a curve of equal length in $\GLpn$ with respect to $g$. Since the geodesic distance is defined as the infimal length of connecting curves, this immediately implies
\[
	\dgiso \left( \Fiso,\; \SOn \right) \geq \dgfull \left( \Fiso,\; \SOn \right)
\]
as well as
\[
	\dgvol \left( \Fvol,\; \id \right) \geq \dgfull \left( \Fvol,\; \id \right) \geq \dgfull \left( \Fvol,\; \SOn \right)
\]
for $\Fiso \colonequals (\det F)^{-\afrac1n}\cdot F$ and $\Fvol \colonequals (\det F)^{\afrac1n}\cdot\id$.
We can therefore use Theorem \ref{theorem:mainResult} to obtain the lower bounds\footnote{For some of the rules of computation employed here involving the matrix logarithm, we refer to Lemma \ref{lemma:logRules} in the appendix.}
\begin{align}
	&\hspace{-1.96em}\dgiso^2 \left( \Fiso,\; \SOn \right)\nnl
	&\geq \dgfull^2 \left( \Fiso,\; \SOn \right)\nnl
	&= \norm{\dev_n\log \left(\sqrt{\Fiso^T\Fiso}\right) }^2 + \frac1n\left[\tr\left(\log\sqrt{\Fiso^T\Fiso}\right)\right]^2\nnl
	&= \norm{\log \left(\left(\det\sqrt{\Fiso^T\Fiso}\right)^{-\afrac1n}\sqrt{\Fiso^T\Fiso}\right) }^2 + \frac1n\,\left[\ln\left(\;\smash{\overbrace{\det\sqrt{\Fiso^T\Fiso}}^{=1}} \vphantom{\det\sqrt{\Fiso^T\Fiso}}\;\right)\right]^2\nnl
	&= \norm{\log \left(\sqrt{\Fiso^T\Fiso}\right) }^2 = \norm{\log \left((\det F)^{-\afrac1n}\sqrt{F^TF}\right)}^2 = \isomeas^2(F)
\end{align}
and
\begin{align}
	\dgvol^2 \left( \Fvol,\; \id \right)
	&\geq \dgfull^2 \left( \Fvol,\; \SOn \right)\nnl
	&= \norm{\dev_n\log \left(\sqrt{\Fvol^T\Fvol}\right) }^2 + \frac1n\,[\tr(\log \left(\sqrt{\Fvol^T\Fvol}\right) )]^2\\
	&= \norm{\dev_n\left(\ln((\det F)^{\afrac1n}) \cdot\id\right) }^2 + \frac1n\,[\ln(\det\left((\det F)^{\afrac1n}\cdot\id\right) )]^2\nnl&
	= \frac1n\,[\ln(\det \sqrt{F^TF})]^2 = \frac1n\,[\tr(\log\sqrt{F^TF})]^2 = \frac1n\,\volmeas^2(F)\,.\nonumber
\end{align}
To obtain an upper bound on the geodesic distances, we define the two curves
\begin{alignat*}{2}
	&\gammaiso\col [0,1]\to\SLn\,,\quad &&\gammaiso(t)=R\,\exp(t\,\dev_n\log U)
\intertext{and}
	&\gammavol\col [0,1]\to\R^+\cdot\id\,,\quad &&\gammavol(t)=e^{\frac{t}{n}\tr(\log U)} \cdot \id\,,
\end{alignat*}
where $F=R\,U$ with $R\in\SOn$ and $U\in\PSymn$ is the polar decomposition of $F$. Then $\gammaiso$ connects $(\det F)^{-\afrac1n}\cdot F$ to $\SOn$:
\begin{alignat*}{2}
	\gammaiso(0) &= R \in\SOn\,,\\
	\gammaiso(1) &= R\,\exp(\dev_n\log U) &&= R\,\exp(\log U - \tfrac{\tr(\log U)}{n}\cdot\id) \\
	&&&= R\,\exp(\log U)\,\exp(-\tfrac{\tr(\log U)}{n}\cdot\id)\\
	&&&= R\,U\, \exp(-\tfrac{\ln\det U}{n}\cdot\id)
	= (\det U)^{-\afrac1n}\cdot F = (\det F)^{-\afrac1n}\cdot F\,,
\end{alignat*}
while $\gammavol$ connects $(\det F)^{\afrac1n}\cdot \id$ and $\id$:
\begin{alignat*}{2}
	\gammavol(0) &= \id\,, \quad \gammavol(1) &= e^{\frac{1}{n}\tr(\log U)}\cdot\id = e^{\frac{1}{n}\ln(\det U)}\cdot\id = (\det U)^{\afrac1n}\cdot\id = (\det F)^{\afrac1n}\cdot\id\,.
\end{alignat*}
The lengths of the curves compute to
\begin{align}
	\len(\gammaiso) &= \int_0^1 \norm{\gammaiso(t)\inv\gammaisodot(t)}\,\dt\\
	&= \int_0^1 \norm{(R\,\exp(t\,\dev_n\log U))\inv\,R\,\exp(t\,\dev_n\log U)\,\dev_n\log U}\,\dt\nnl
	&= \int_0^1 \norm{\dev_n\log U}\,\dt = \norm{\dev_n\log \sqrt{F^TF}} = \isomeas(F)\nonumber
\intertext{as well as}
	\len(\gammavol) &= \int_0^1 \norm{\gammavol(t)\inv\gammavoldot(t)}\,\dt\\
	&= \int_0^1 \norm{(e^{\frac{t}{n}\tr(\log U)} \cdot \id)\inv \cdot \tfrac{\tr(\log U)}{n}\cdot e^{\frac{t}{n}\tr(\log U)} \cdot \id}\,\dt\nnl
	&= \int_0^1 \norm{\tfrac{\tr(\log U)}{n} \cdot \id}\,\dt = \frac{\abs{\tr(\log U)}}{n} \cdot \norm{\id} = \frac{1}{\sqrt{n}}\;\abs{\tr(\log \sqrt{F^TF})} = \frac{1}{\sqrt{n}}\;\volmeas(F)\,,\nonumber
\end{align}
showing that
\begin{align*}
	\dgiso^2 \left( (\det F)^{-\afrac1n}\cdot F,\; \SOn \right) &\leq \len^2(\gammaiso) = \isomeas^2(F)
\intertext{and}
	\dgvol^2 \left( (\det F)^{\afrac1n}\cdot\id,\; \id \right) &\leq \len^2(\gammavol) = \frac1n\cdot\volmeas^2(F)\,,
\end{align*}
which completes the proof.
\end{proof}
\begin{remark}
In addition to the isochoric (distortional) part $\Fiso=(\det F)^{-\afrac1n}\cdot F$ and the volumetric part $\Fvol=(\det F)^{\afrac1n}\cdot\id$, we may also consider the \emph{cofactor} $\Cof F = (\det F) \cdot F^{-T}$ of $F\in\GLpn$. Theorem \ref{theorem:mainResult} allows us to directly compute (cf.\ Appendix \ref{appendix:cofactorComputation}) the distance
\[
	\dg^2(\Cof F, \SOn) = \mu \, \norm{\dev_n\log U}^2+\frac{\kappa\,(n-1)^2}{2}\,[\tr(\log U)]^2\,.\qedhere
\]
\end{remark}

\section{Alternative motivations for the logarithmic strain}
\label{section:alternativeMotivations}
\subsection[Riemannian geometry applied to $\PSymn$]{\boldmath Riemannian geometry applied to $\PSymn$}
Extensive work on the use of Lie group theory and differential geometry in continuum mechanics has already been done by Roug{\'e}e \cite{rougee1992intrinsic,rougee1991new,rougee1997mecanique,rougee2006intrinsic}, Moakher \cite{moakher2005differential,Moakher2010}, Bhatia \cite{bhatia2006riemannian} and, more recently, by Fiala \cite{fiala2004time,fiala2011geometrical,fiala2014evolution,fiala2015discussion,Fiala2016} (cf.\ \cite{latorre2014interpretation,latorre2015response,Ohara1996,pennec2006,Pennec2009}). They all endowed the convex cone $\PSym(3)$ of positive definite symmetric $(3\times3)$-tensors with the Riemannian metric\footnote{Note the subtle difference with our metric $g_C(X,Y)=\iprod{C\inv X, C\inv Y}$. Pennec \cite[p.\ 368]{Pennec2009} generalizes \eqref{eq:PSymMetricDefinition} by using the weighted inner product $\iprod{X,Y}_\ast=\iprod{X,Y}+\beta\,\tr(X)\,\tr(Y)$ with $\beta>-\frac1n$.}
\begin{equation}
\label{eq:PSymMetricDefinition}
	\gtilde_C(X,Y) = \tr(C\inv X C\inv Y) = \iprod{XC\inv, C\inv Y} = \iprod{C^{-\afrac12} \,X\, C^{-\afrac12}, C^{-\afrac12} \,Y\, C^{-\afrac12}}\,,
\end{equation}
where $C\in\PSym(3)$ and $X,Y\in\Sym(3)=T_C\PSym(3)$. Fiala and Roug{\'e}e deduced a motivation of the logarithmic strain tensor $\log U$ via geodesic curves connecting elements of $\PSymn$. However, their approach differs markedly from our method employed in the previous sections: the manifold $\PSymn$ already corresponds to \emph{metric states} $C=F^TF$, whereas we consider the full set $\GLpn$ of deformation gradients $F$ (cf.\ Appendix \ref{section:tensorsAndTangentSpaces} and Table \ref{table:summary} in Section \ref{section:conclusion}). This restriction can be viewed as the nonlinear analogue of the a priori restriction to $\eps=\sym\grad u$ in the linear case, i.e.\ the nature of the strain measure is not deduced but postulated. Note also that the metric $\gtilde$ %
cannot be obtained by restricting our left-$\GL(3)$-invariant, right-$\OO(3)$-invariant metric $g$ to $\PSym(3)$.\footnote{Since $\PSymn$ is not a Lie group with respect to matrix multiplication, the metric $\gtilde$ itself cannot be left- or right-invariant in any suitable sense.} Furthermore, while Fiala and Roug{\'e}e aim to motivate the Hencky strain tensor $\log U$ directly, our focus lies on the strain measures $\isomeas$, $\volmeas$ and the isotropic Hencky strain energy $\WH$. %

The geodesic curves on $\PSymn$ with respect to $\gtilde$ are of the simple form\footnote{While Moakher gives the parametrization stated here, Roug{\'e}e writes the geodesics in the form $\gamma(t)=\exp(t\cdot\Log(C_2 C_1\inv))\,C_1$ with $C_1,C_2\in\PSymn$, which can also be written as $\gamma(t)= (C_2 C_1\inv)^t\,C_1$; a similar formulation is given by Tarantola \cite[eq.\ (2.78)]{Tarantola06}. For a suitable definition of a matrix logarithm $\Log$ on $\GLpn$, these representations are equivalent to \eqref{eq:PSymGeodesicCurves} with $M=\log(C_2^{-\afrac12} \,C_1\, C_2^{-\afrac12})\in\Symn$.}
\begin{equation}
\label{eq:PSymGeodesicCurves}
	\gamma(t) \;=\; C_1^{\afrac12}\,\exp(t\cdot C_1^{-\afrac12} \,M\, C_1^{-\afrac12})\,C_1^{\afrac12}
\end{equation}
with $C_1\in\PSymn$ and $M\in\Symn=T_{C_1}\!\PSymn$. These geodesics are defined globally, i.e.\ $\PSymn$ is geodesically complete. Furthermore, for given $C_1,C_2\in\PSymn$, there exists a \emph{unique} geodesic curve connecting them; this easily follows from the representation formula \eqref{eq:PSymGeodesicCurves} or from the fact that the curvature of $\PSymn$ with $\gtilde$ is constant and negative \cite{fiala2011geometrical,Jost1998,bhatia2009positive}. Note that this implies that, in contrast to $\GLpn$ with our metric $g$, there are no closed geodesics on $\PSymn$.

An explicit formula for the corresponding geodesic distance was given by Moakher:\footnote{Moakher \cite[eq.\ (2.9)]{moakher2005differential} writes this result as $\norm{\Log (C_2\inv C_1)}=\sqrt{\sum_{i=1}^n \ln^2 \lambda_i}$, where $\lambda_i$ are the eigenvalues of $C_2\inv C_1$. The right hand side of this equation is identical to the result stated in \eqref{eq:moakherGeodesicDistanceOnPsym}. However, since $C_2\inv C_1$ is not necessarily normal, there is in general no logarithm $\Log(C_2\inv C_1)$ whose Frobenius norm satisfies this equality. Note that the eigenvalues of the matrix $C_2\inv C_1$ are real and positive due to its similarity to $C_2^{\afrac12}(C_2\inv C_1)C_2^{-\afrac12}=C_2^{-\afrac12}C_1C_2^{-\afrac12}\in\PSymn$.}
\begin{equation}
\label{eq:moakherGeodesicDistanceOnPsym}
	\dist_{\mathrm{geod,}\PSymn} (C_1,C_2) = \norm{\log(C_2^{-\afrac12} \,C_1\, C_2^{-\afrac12})}\,. %
\end{equation}
In the special case $C_2=\id$, this distance measure is equal to our geodesic distance on $\GLpn$ induced by the canonical inner product: Theorem \ref{theorem:mainResult}, applied with parameters $\mu=\muc=1$ and $\kappa=\frac2n$ to $R=\id$ and $U=C_1$, shows that
\[
	\dist_{\mathrm{geod,}\GLpn}(C_1,\id) = \norm{\log C_1} = \dist_{\mathrm{geod,}\PSymn} (C_1,\id)\,. 
\]
More generally, assume that the two metric states $C_1,C_2\in\PSymn$ commute. Then $C_2\inv C_1\in\PSymn$, and the left-$\GLn$-invariance of the geodesic distance implies
\begin{align}
	\dist_{\mathrm{geod,}\GLpn}(C_1,C_2) &= \dist_{\mathrm{geod,}\GLpn}(C_2\inv C_1,\id) = \norm{\log (C_2\inv C_1)} \nnl
	&= \norm{\log(C_2^{-\afrac12}\,C_2^{-\afrac12} \,C_1)} = \norm{\log(C_2^{-\afrac12} \,C_1\, C_2^{-\afrac12})} \label{eq:psymDistanceEquality}\\
	&= \dist_{\mathrm{geod,}\PSymn} (C_1,C_2)\,.\nonumber
\end{align}
However, since $C_2\inv C_1\notin\PSymn$ in general, this equality does not hold on all of $\PSymn$.

A different approach towards distance functions on the set $\PSymn$ was suggested by Arsigny et al. \cite{arsigny2005fast,arsigny2007geometric,arsigny2009fast} who, motivated by applications of geodesic and logarithmic distances in diffusion tensor imaging, directly define their \emph{Log-Euclidean metric} on $\PSymn$ by
\begin{equation}
	\dist_{\text{\rm{Log-Euclid}}}(C_1,C_2) \colonequals \norm{\log C_1 - \log C_2}\,,
\end{equation}
where $\norm{\,.\,}$ is the Frobenius matrix norm. If $C_1$ and $C_2$ commute, this distance equals the geodesic distance on $\GLpn$ as well:
\begin{align}
	\dist_{\mathrm{geod,}\GLpn}(C_1,C_2) &= \norm{\log (C_2\inv C_1)} \nnl
	&= \norm{\log (C_2\inv) + \log(C_1)}\label{eq:logEuclideanEquality}\\
	&= \norm{\log C_1 - \log C_2} = \dist_{\text{\rm{Log-Euclid}}}(C_1,C_2)\,, \nonumber
\end{align}
where equality in \eqref{eq:logEuclideanEquality} holds due to the fact that $C_1$ and $C_2\inv$ commute. Again, this equality does not hold for arbitrary $C_1$ and $C_2$.

Using a similar Riemannian metric, geodesic distance measures can also be applied to the set of positive definite symmetric fourth-order elasticity tensors, which can be identified with $\PSym(6)$. Norris and Moakher applied such a distance function in order to find an isotropic elasticity tensor $\C\col \Sym(3) \to \Sym(3)$ which best approximates a given anisotropic tensor \cite{moakher2006closest,norris2006isotropic}.

The connection between geodesic distances on the metric states in $\PSymn$ and logarithmic distance measures was also investigated extensively by the late Albert Tarantola \cite{Tarantola06}, a lifelong advocate of logarithmic measures in physics. In his view \cite[4.3.1]{Tarantola06}, \quoteref{\ldots the configuration space is the Lie group $\GLp(3)$, and the only possible measure of strain (as the geodesics of the space) is logarithmic.}

\begin{samepage}
\subsection{Further mechanical motivations for the quadratic isotropic Hencky model based on logarithmic strain tensors}
\label{section:mechanicalMotivations}
\begin{quote}
	\quoteref{At the foundation of all elastic theories lies the definition of strain, and before introducing a new law of elasticity we must explain how finite strain is to be measured.}
\end{quote}
{\raggedleft\scriptsize {Heinrich Hencky: The elastic behavior of vulcanized rubber \cite{hencky1933elastic}.} \qquad\quad\\[2em]}
\end{samepage}
Apart from the geometric considerations laid out in the previous sections, the Hencky strain tensor $E_0 = \log U$ can be characterized via a number of unique properties. 

For example, the Hencky strain is the only strain tensor (for a suitably narrow definition, cf.\ \cite{neff2014becker}) that satisfies the \emph{law of superposition} for coaxial deformations:
\begin{equation}
\label{eq:decompositionOfHenckyStrain}
	E_0(U_1\cdot U_2) = E_0(U_1) + E_0(U_2)
\end{equation}
for all coaxial stretches $U_1$ and $U_2$, i.e.\ $U_1,U_2\in\PSymn$ such that $U_1\cdot U_2 = U_2\cdot U_1$. This characterization was used by Heinrich Hencky \cite{tanner2003heinrich,hencky1923,hencky1931,hencky1933elastic} in his original introduction of the logarithmic strain tensor \cite{Hencky1928,Hencky1929,hencky1929super,henckyTranslation} and, indeed much earlier, by the geologist George Ferdinand Becker \cite{merrill1927biographical}, who postulated a similar law of superposition in order to deduce a logarithmic constitutive law of nonlinear elasticity \cite{becker1893,neff2014becker} (cf.\ Appendix \ref{section:linearRelations}).

In the case $n=1$, this superposition principle simply amounts to the fact that the logarithm function $f=\log$ satisfies Cauchy's \cite{cauchy1821cours} well-known functional equation
\begin{equation}
\label{eq:cauchysFunctionalEquation}
	f(\lambda_1\cdot\lambda_2)=f(\lambda_1)+f(\lambda_2)\,,
\end{equation}
i.e.\ that the logarithm is an isomorphism between the multiplicative group $(\Rp,\cdot)$ and the additive group $(\R,+)$. This means that for a sequence of incremental one-dimensional deformations, the logarithmic strains $\scaleelog^i$ can be added in order to obtain the total logarithmic strain $\scaleelog^{\mathrm{tot}}$ of the composed deformation \cite{fitzgerald1980tensorial}:
\[
	\scaleelog^1 + \scaleelog^2 + \ldots + \scaleelog^n = \log\frac{L_1}{L_0}+\log\frac{L_2}{L_1} + \ldots + \log\frac{L_{n}}{L_{n-1}} = \log\frac{L_n}{L_0}=\scaleelog^{\mathrm{tot}}\,,
\]
where $L_i$ denotes the length of the (one-dimensional) body after the $i$-th elongation. This property uniquely characterizes the logarithmic strain $\scaleelog$ among all differentiable one-dimensional strain mappings $\scalee\col\R^+\to\R$ with $\scalee'(1)=1$. %

Since purely volumetric deformations of the form $\lambda\cdot\id$ with $\lambda>0$ are coaxial to every stretch $U\in\PSymn$, the decomposition property \eqref{eq:decompositionOfHenckyStrain} allows for a simple \emph{additive volumetric-isochoric split} of the Hencky strain tensor \cite{richter1949verzerrung}:
\newcommand{\TEMPisoVolHeight}{\vphantom{\frac{U}{(\det U)^{\afrac1n}}}}
\begin{align*}
	\log U =\, \log \bigg[\underbrace{\frac{U}{(\det U)^{\afrac1n}}}_{\text{isochoric}} \:\cdot\: \underbrace{\TEMPisoVolHeight(\det U)^{\afrac1n} \cdot \id}_{\text{volumetric}}\bigg]
	&=\, \log \bigg[\frac{U}{(\det U)^{\afrac1n}}\bigg] \:+\: \log \Big[(\det U)^{\afrac1n} \cdot \id\Big]\\
	&=\, \underbrace{\TEMPisoVolHeight\dev_n\log U}_{\text{isochoric}} \:+\: \underbrace{\TEMPisoVolHeight\frac1n\tr(\log U)\cdot\id}_{\text{volumetric}}\,.
\end{align*}
In particular, the incompressibility condition $\det F = 1$ can be easily expressed as $\tr(\log U) = 0$ in terms of the logarithmic strain tensor.

\subsubsection{From Truesdell's hypoelasticity to Hencky's hyperelastic model}
\label{section:hypoelasticity}
As indicated in Section \ref{section:actualIntroduction}, the quadratic Hencky energy is also of great importance to the concept of \emph{hypoelasticity} \cite[Chapter IX]{grioli1962equilibrium}. It was 
found that the Truesdell equation\footnote{It is telling to see that equation \eqref{eq:truesdellEquation} had already been proposed by Hencky himself in \cite{hencky1929super} for the \emph{Zaremba-Jaumann stress rate} (cf.\ \eqref{eq:jaumannRate}). Hencky's work, however, contains a typographical error \hbox{\cite[eq.~(10) and eq.~(11e)]{hencky1929super}} changing the order of indices in his equations (cf.\ \cite{bruhns2014history}). The strong point of writing \eqref{eq:truesdellEquation} is that no discussion of any suitable strain tensor is necessary.%
} \cite{truesdell1952,truesdell1955simplest,truesdell1955hypo,freed2014hencky}
\begin{equation}
\label{eq:truesdellEquation}
	\ddtsquare[\tau] = 2\,\mu \,D + \lambda\,\tr(D)\cdot\id\,,\qquad D= \sym(\dot F\, F^{-1})\,,
\end{equation}
with constant Lam\'e coefficients $\mu,\lambda>0$, under the assumption that the stress rate $\ddtsquare$ is \emph{objective}\footnote{%
A rate $\ddtsquare$ is called objective if $\ddtsquare \big[S(QB\Qdot^T)\big] = Q\,(\ddtsquare [S(B)])Q^T$ for all (not necessarily constant) $Q=Q(t)\in\On$, where $S$ is any objective stress tensor, and if $\ddtsquare[S]=0\,\Leftrightarrow\,S=0$, i.e.\ the motion is rigid if and only if $\ddtsquare[S]\equiv0$.}
\textbf{and} \emph{corotational},
is satisfied if and only if $\ddtsquare$ is the so-called logarithmic corotational rate $\ddtlog$ and $\tau = 2\,\mu \,\log V + \lambda\,\tr(\log V)\cdot\id$
\cite{xiao1999existence,xiao2005,norris2008eulerian,reinhardt1995eulerian,reinhardt1996application,xiao1997hypo,xiao2002hencky,xiao2003hencky}, i.e.\ if and only if the hypoelastic model is exactly Hencky's hyperelastic constitutive model. Here, $\tau=\det F \cdot \sigma(V)$ denotes the Kirchhoff stress tensor and $D$ is the unique rate of stretching tensor (i.e.\ the symmetric part of the velocity gradient in the spatial setting). A rate $\ddtsquare$ is called corotational if it is of the special form
\begin{equation}
	\ddtsquare[X] = \dot X - \Omega X + X\Omega \quad\text{ with $\Omega\in\so(3)$}\,,
\end{equation}
which means that the rate is computed with respect to a frame that is rotated.\footnote{Corotational rates are also special cases of Lie derivatives \cite{Hugh1977b,marsden1994foundations}.} This extra rate of rotation is defined only by the underlying spins of the problem.
Upon specialisation, for $\mu=1$, $\lambda=0$ we obtain\footnote{Cf.\ Xiao, Bruhns and Meyers \cite[p.~90]{xiao1997logarithmic}: \quoteref{\ldots the logarithmic strain \quoteesc{does} possess certain intrinsic far-reaching properties \quoteesc{which} establish its favoured position in all possible strain measures}.} \cite[eq.~71]{bruhns2014prandtl}
\[
	\ddtlog[\log V]=D
\]
as the unique solution to \eqref{eq:truesdellEquation} with a corotational rate. Note that this characterization of the spatial logarithmic strain tensor $\log V$ is by no means exceptional. For example, it is well known that \cite[p.~49, Theorem 1.8]{Haupt99} (cf.\ \cite{bruhns2004oldroyd})
\[
	\ddttriangle[A] = \dot{A} + L^TA + AL = D\,,
\]
where $A=\Ehat_{-1}=\frac12(\id-B\inv)$ is the spatial Almansi strain tensor and $\ddttriangle$ is the upper Oldroyd rate (as defined in \eqref{eq:oldroydRate}).

The quadratic Hencky model
\begin{equation}
\label{eq:Hencky-model}
	\tau = 2\,\mu \,\log V + \lambda\,\tr(\log V)\cdot\id  = D_{\log V} \WH(\log V)
\end{equation}
was generalized in Hill's generalized linear elasticity laws\footnote{\emph{Hooke's law} \cite{hooke1931oxford} (cf.\ \cite{moyer1977hooke}) famously states that the \emph{strain} in a deformation depends linearly on the occurring \emph{stress} (\enquote{ut tensio, sic vis}). However, for finite deformations, different constitutive laws of elasticity can be obtained from this assumption, depending on the choice of a stress/strain pair. An idealized version of such a linear relation is given by \eqref{eq:Hencky-model}, i.e.\ by choosing the spatial Hencky strain tensor $\log V$ and the Kirchhoff stress tensor $\tau$. Since, however, Hooke speaks of extension versus force, the correct interpretation of Hooke's law is $\Biot=2\,\mu\,(U-\id)+\lambda\tr(U-\id)\cdot\id$, i.e.\ the case $r=\frac12$ in \eqref{eq:Generalized-Energy}.} \cite[eq.\ (2.69)]{hill1978}
\begin{equation}
\label{eq:Generalized-Energy}
	T_r = 2\,\mu\, E_r + \lambda\,\tr(E_r)\cdot \id
\end{equation}
with work-conjugate pairs $(T_r,E_r)$ based on the Lagrangian strain measures given in \eqref{eq:sethHillFamily}; cf.\ Appendix \ref{section:linearRelations} for examples.
The concept of \emph{work-conjugacy} was introduced by Hill \cite{Hill68} via an invariance requirement; the spatial stress power 
must be equal to its Lagrangian counterpart:
\begin{equation}
	\det F \cdot \iprod{\sigma, D} = \iprod{T_r, \dot{E}_r}\,, \tag{work-conjugacy}
\end{equation}
by means of which a material 
stress tensor is uniquely linked to its (material rate) conjugate strain tensor. Hence it generalizes the virtual work principle 
and is the foundation of derived methods like the finite element method. 

{For the case of isotropic materials, Hill \cite[p.\ 242]{Hill68} (cf.\ \cite{Hoger87}) shows by spectral decomposition techniques that the work-conjugate
stress to $\log U$ is the back-rotated Cauchy stress $\sigma$ multiplied by $\det F$, hence $\iprod{\sigma, D} = \iprod{R^T\,\sigma \,R, \;\ddt \log U}$, 
which is a generalization of Hill's earlier work \cite{Hill68,hill1978}. 
Sansour \cite{Sansour2001} additionally found that the Eshelby-like stress tensor $\Sigma=CS_2$ is equally conjugate to $\log U$; here, $S_2$ denotes the second Piola-Kirchhoff stress tensor. 
For anisotropy, however, the conjugate stress exists but follows a much more complex format than for isotropy \cite{Hoger87}. The logarithm of the left stretch 
$\log V$ in contrast exhibits a work conjugate stress tensor only for isotropic materials, namely the Kirchhoff stress tensor $\tau = \det F \cdot \sigma$ 
\cite{Ogden83,Hoger87}.} 

While hyperelasticity in its potential format avoids rate equations, the use of \emph{stress rates} (i.e.\ stress increments in time) 
may be useful for the description of inelastic material behavior at finite strains. Since the material time derivative of an Eulerian 
stress tensor is not objective, rates for a tensor $X$ were developed, like 
the (objective and corotational) \emph{Zaremba-Jaumann rate}
\begin{equation}
\label{eq:jaumannRate}
	\ddtdot[X] = \dot X - WX + XW\,, \qquad W = \skew L\,, \quad L = \dot F F^{-1}\,,
\end{equation}
or the (objective but not corotational) lower and upper \emph{Oldroyd rates}
\begin{equation}
\label{eq:oldroydRate}
	\ddttriangledown[X] = \dot X + L^T X + XL \qquad\text{and}\qquad \ddttriangle[X] = \dot X - LX - XL^T\,,
\end{equation}
to name but a few (cf.\ \cite[Section 1.7]{Haupt99} and \cite{sansour1993study}). Which one of these or the great number of other objective rates should be used seems to be rather a matter of taste, 
hence of arbitrariness\footnote{Truesdell and Noll \cite[p.~404]{truesdell65} declared that \quoteref{various such 
stress rates have been used in the literature. Despite claims and whole papers to the contrary, any advantage claimed 
for one such rate over another is pure illusion}, and that \quoteref{the properties of a material are independent of the choice of flux \quoteesc{i.e.\ of the chosen rate}, which, like the choice of a \quoteesc{strain tensor}, is absolutely immaterial} \cite[p.~97]{truesdell65}.} or heuristics\footnote{For a shear test in Eulerian elasto-plasticity 
using the Zaremba-Jaumann rate \eqref{eq:jaumannRate}, an unphysical 
artefact of oscillatory shear stress was observed, first in \cite{Lehmann72}. A similar oscillatory behavior was observed for 
hypoelasticity in \cite{Dienes79}.}, but not a matter of theory. 

The concept of \emph{dual variables}\footnote{Hill \cite{hill1978} used the terms \emph{conjugate} and \emph{dual} as synonyms.} as introduced by 
Tsakmakis and Haupt in \cite{TsakmakisHaupt89} into continuum mechanics overcame the arbitrariness of the chosen rate in that it 
uniquely connects a particular (objective) strain rate to a stress tensor and, analogously, a stress rate to a strain tensor. The rational 
rule is that, when stress and strain tensors operate on configurations other than the reference configurations, the physically significant scalar products 
$\iprod{S_2, \dot E_1}$, $\iprod{\dot S_2, E_1}$, $\iprod{S_2, E_1}$ and $\iprod{{\dot S}_2, \dot E_1}$ (with the second Piola-Kirchhoff stress tensor 
$S_2$ and its work-conjugate Green strain tensor $E_1$) must remain invariant, see \cite{TsakmakisHaupt89,Haupt99}.

\subsubsection{Advantageous properties of the quadratic Hencky energy}
For modelling elastic material behavior there is no theoretical reason to prefer one strain tensor over another one, and the same is true for stress tensors. As discussed in Section \ref{section:actualIntroduction}, stress and strain are \emph{immaterial}.\footnote{Cf.\ Truesdell \cite[p.~145]{truesdell1952}: \quoteref{It is important to realize that since each of the several material tensors [\dots] is an isotropic function of any one of the others, an exact description of strain in terms of any one is equivalent to a description in terms of any other} or Antman \cite[p.~423]{antman2005nonlinear}: \quoteref{In place of $C$, any invertible tensor-valued function of $C$ can be used as a measure of strain.} Rivlin \cite{rivlin1950} states that strain need never be defined at all, cf.\ \cite[p.\ 122]{truesdell65}.} Primary experimental data (forces, displacements) in material testing are sufficient to calculate any strain tensor and any stress tensor and to display any combination thereof in stress-strain curves, while only work-conjugate pairs are physically meaningful. %

However, for modelling finite-strain elasticity, the quadratic Hencky model
\begin{align}
	\WH &= \mu\,\norm{\dev_n\log V}^2 + \frac{\kappa}{2}\,[\tr(\log V)]^2 = \mu\,\norm{\dev_n\log U}^2 + \frac{\kappa}{2}\,[\tr(\log U)]^2\,,\nnl
	\tau &= 2\,\mu\,\dev_n\log V + \kappa\,\tr(\log V)\, \id\,, \label{eq:Hencky-model_two}
\end{align}
exhibits a number of unique,
favorable properties, including its functional simplicity and its dependency on only two material parameters $\mu$ and $\kappa$ 
that are determined in the infinitesimal strain regime and remain constant over the entire strain range. In view of the linear dependency 
of stress from logarithmic strain in \eqref{eq:Hencky-model_two}, it is obvious that any nonlinearity in the stress-strain curves can 
only be captured in Hencky's model by virtue of the nonlinearity in the strain tensor itself. There is a surprisingly large 
number of different materials, where Hencky's elasticity relation provides a very good fit to experimental stress-strain data, which 
is true for different length scales and strain regimes. In the following we substantiate this claim with some examples.

\emph{Nonlinear elasticity on macroscopic scales for a variety of materials.} \quad 
Anand \cite{Anand79,Anand86} has shown that the Hencky model is in good agreement with experiments on a wide class of materials,
e.g.\ vulcanized natural rubber, for principal stretches between $0.7$ and $1.3$. More precisely, this refers to the characteristic 
that in tensile deformation the stiffness becomes increasingly smaller compared with the stiffness at zero strain, while for compressive 
deformation the stiffness becomes increasingly larger.       

\emph{Nonlinear elasticity in the very small strain regime.} \quad 
We mention in passing that a qualitatively similar dependency of material stiffness on the sign of the strain has been made much earlier 
in the regime of extremely small strains ($10^{-6}$--$10^{-3}$). In Hartig's law \cite{hartig1893} from 1893 this dependency was expressed 
as $\frac{\mathrm{d}\sigma}{\mathrm{d}\eps}= E^0 +b\,\sigma$, where $E^0$ is the elasticity modulus at zero stress and $b<0$ is a dimensionless constant,\footnote{The negative curvature ($b<0$) was already suggested by Jacob Bernoulli in 1705 \cite{bernoulli1705veritable} (cf.\ \cite[p.\ 276]{benvenuto1991}): \quoteref{Homogeneous fibers of the same length and thickness, but loaded with different weights, neither lengthen nor shorten proportional to these weights; but the lengthening or the shortening caused by the small weight is less than the ratio that the first weight has to the second.}} 
cf.\ the book of Bell \cite{bell1973} and \cite{Man98} in the context of linear elasticity with initial stress.
Hartig also observed that the stress-stretch relation should have negative curvature\footnote{As Bell insists \cite[p.~155]{bell1973}, a purely linear elastic response to finite strain, corresponding to zero curvature of the stress-strain curve at the identity $\id$, is never exhibited by any physical material: \quoteref{The experiments of 280 years have demonstrated amply for every solid substance examined with sufficient care, that the \quoteesc{finite engineering} strain \quoteesc{$U-\id$} resulting from small applied stress is not a linear function thereof.}} in the vicinity of the identity, as shown in Figure \ref{fig:thirdOrderConstants}.

\emph{Crystalline elasticity on the nanoscale.} \quad  
Quite in contrast to the strictly \emph{stress}-based continuum constitutive modelling, atomistic theories are based on a concept 
of interatomic \emph{forces}. These forces are derived from potentials\footnote{For molecular dynamics (MD) simulations, 
a well-established level of sophistication is the modelling by potentials with environmental dependence (pair functionals 
like in the Embedded Atom Method (EAM) account for the energy cost to embed atomic nuclei into the electron gas of variable 
density) and angular dependence (like for Stillinger-Weber or Tersoff functionals).} $\mathcal{V}$ according to the potential relation 
$f_a = -\partial_{x_a}\mathcal{V}$, which endows the model with a variational structure. A further discussion of hybrid, atomistic-continuum coupling can be found in \cite{eidel2009variational}.
Thereby the discreteness of matter at the nanoscale and the nonlocality of atomic interactions 
are inherently captured. Here, atomistic stress is neither a constitutive agency nor does it enter a balance equation. Instead, 
it optionally can be calculated following the \emph{virial stress theorem} \cite[Chapter 8]{tadmor2011modeling} to illustrate the state of the system. 

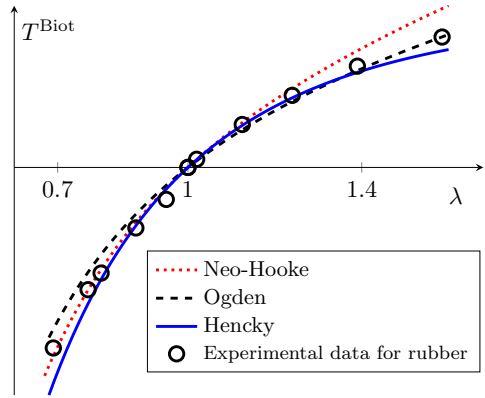
\begin{wrapfigure}{r}{0.5\textwidth}
	\centering
	\tikzsetnextfilename{thirdOrderConstants}
	\begin{tikzpicture}[scale=.91]
		\input{tikz/thirdOrderConstants.tex}
	\end{tikzpicture}
	\caption{\label{fig:thirdOrderConstants} The Biot stress $\Biot$ corresponding to uniaxial stretches by factor $\lambda$ of incompressible materials fitted to experimental measurements by Jones and Treloar \cite{jones1975}. The curvature in $\lambda=1$ suggests negative third order constants ($b<0$), which has also been postulated by Grioli \cite[eq.\ (32)]{grioli1966thermodynamic}.}
\end{wrapfigure}

With their analyses in \cite{Dluzewski00} and \cite{Dluzewski03}, D{\l}u{\.z}ewski and coworkers aim to link the atomistic world to 
the macroscopic world of continuum mechanics. They search for the \enquote{best} strain measure with a view towards crystalline elasticity 
on the nanoscale. The authors consider the deformation of a crystal structure and compare the atomistic and continuum approaches. 
Atomistic calculations are made using the Stillinger-Weber potential. The stress-strain behaviour of the best-known anisotropic hyperelastic
models are compared with the behaviour of the atomistic one in the uniaxial deformation test. The result is that the anisotropic energy
based on the Hencky strain energy $\frac{1}{2}\,\iprod{\C.\log U,\log U}$, where $\C$ is the anisotropic elasticity tensor from linear elasticity, gives the best fit to atomistic simulations. More in detail, this 
best fit manifests itself in the observation that for considerable compression (up to $\approx 20\%$) the material stiffness is larger 
than the reference stiffness at zero strain, and for considerable tension (up to $\approx 20\%$) it is smaller than the zero-strain 
stiffness, again in good agreement with the atomistic result. This is also corroborated by comparing tabulated experimentally determined 
third order elastic constants\footnote{Third order elastic constants are corrections to the elasticity tensor in order to improve the response curves beyond 
the infinitesimal neighbourhood of the identity. They exist as tabulated values for many materials. Their numerical values depend 
on the choice of strain measure used which needs to be corrected. D{\l}u{\.z}ewski \cite{Dluzewski00} shows that again the Hencky-strain energy
$\frac12\,\iprod{\C.\log U, \log U}$ provides the best overall approximation.} \cite{Dluzewski00}.

Elastic energy potentials based on logarithmic strain have also recently been motivated via molecular dynamics simulations \cite{henann2009fracture} by Henann and Anand \cite{henann2011large}.

\section{Applications and ongoing research}
\label{section:applicationsAndOngoingResearch}
\subsection{The exponentiated Hencky energy}
As indicated in Section \ref{sectionContains:energyAsStrainMeasure} and shown in Sections \ref{section:euclideanStrainMeasureInLinearElasticity} and \ref{section:riemannianStrainMeasureInNonlinearElasticity}, strain measures are closely connected to isotropic energy functions in nonlinear hyperelasticity: similarly to how the linear elastic energy may be obtained as the square of the Euclidean distance of $\grad u$ to $\son$, the nonlinear quadratic Hencky strain energy is the squared Riemannian distance of $\grad\varphi$ to $\SOn$. For the partial strain measures $\isomeas(F)=\norm{\dev_n\log\sqrt{F^TF}}$ and $\volmeas(F)=\abs{\tr(\log \sqrt{F^TF})}$ defined in Theorem \ref{theorem:separateMeasures}, the Hencky strain energy $\WH$ can be expressed as
\begin{equation}
	\WH(F) = \mu\,\isomeas^2(F) + \frac\kappa2\, \volmeas^2(F)\,.
\end{equation}
However, it is not at all obvious why this weighted squared sum should be viewed as the \enquote{canonical} energy associated with the geodesic strain measures:
while it is reasonable to view the elastic energy as a quantity depending on some strain \emph{measure} alone, the specific form of this dependence must not be determined by purely geometric deductions, but must take into account physical constraints as well as empirical observations.\footnote{G.W.~Leibniz, in a letter to Jacob Bernoulli \cite[p.~572]{leibniz1995letter}, stated as early as 1690 that \quoteref{the \quoteesc{constitutive} relation between extension and stretching force should be determined by experiment}, cf.\ \cite[p.~10]{bell1973}.}

For a large number of materials, the Hencky energy does indeed provide a very accurate model up to moderately large elastic deformations \cite{Anand79,Anand86}, i.e.\ up to stretches of about $40\%$, with only two constant material parameters which can be easily determined in the small strain range. %
For very large strains\footnote{The elastic range of numerous materials, including vulcanized rubber or skin and other soft tissues, lies well above stretches of $40\%$.}, however, the subquadratic growth of the Hencky energy in tension is no longer in agreement with empirical measurements.\footnote{While the behaviour of elasticity models for extremely large strains might not seem important due to physical restraints and intermingling plasticity effects outside a narrow range of perfect elasticity, it is nevertheless important to formulate an \emph{idealized} law of elasticity over the whole range of deformations; cf.\ Hencky \cite[p.~215]{Hencky1928} (as translated in \cite[p.2]{henckyTranslation}): \quoteref{It is not important that such an idealized elastic \quoteesc{behaviour} does not actually exist and our ideally elastic material must therefore remain an ideal. Like so many mathematical and geometric concepts, it is a useful ideal, because once its deducible properties are known it can be used as a comparative rule for assessing the actual elastic behaviour of physical bodies.}} In a series of articles \cite{agn_neff2015exponentiatedI, agn_neff2015exponentiatedII, agn_neff2014exponentiatedIII, agn_ghiba2015exponentiated}, Neff et al.\ have therefore introduced the \emph{exponentiated Hencky energy}
\begin{align}
	\WeH(F) \;&=\; \frac{\mu}{k}\,e^{k\,\isomeas^2(F)} + \frac{\kappa}{2\hat{k}}\,e^{\hat{k}\,\volmeas^2(F)} \;=\; \frac{\mu}{k}\,e^{k\,\|\dev_n\log {U}\|^2}+\frac{\kappa}{2\hat{k}}\,e^{\hat{k}\,[\tr(\log U)]^2}
\end{align}
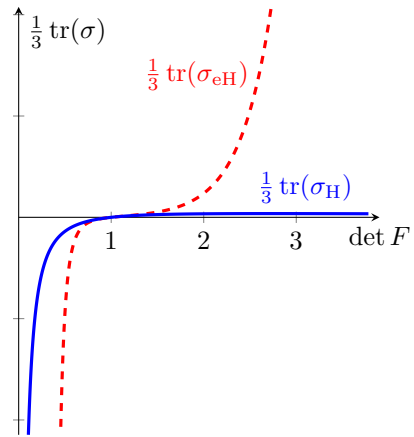
\begin{wrapfigure}{r}{0.49\textwidth}
	\centering
	\tikzsetnextfilename{expHenckyThreeDimensionalEoS}
	\begin{tikzpicture}
		\input{tikz/expHenckyThreeDimensionalEoS.tex}
	\end{tikzpicture}
	\caption{The \emph{equation of state} (EOS), i.e.\ the trace of the Cauchy stress corresponding to a purely volumetric deformation (cf.\ \cite{poirier1998logarithmic}), for the quadratic and the exponentiated Hencky model (with parameter $\khat=4$).}
\end{wrapfigure}
with additional dimensionless material parameters $k\geq\frac14$ and $\khat\geq\frac18$, which for all values of $k,\khat$ approximates $\WH$ for deformation gradients $F$ sufficiently close to the identity $\id$, but shows a vastly different behaviour for $\norm{F}\to\infty$, cf.\ Figure \ref{fig:oneDimensionalHenckyAndExpHenckyEnergyPlot}.

\begin{figure} %
	\centering
	\tikzsetnextfilename{oneDimensionalHenckyAndExpHenckyEnergyPlot}
	\begin{tikzpicture}
		\input{tikz/oneDimensionalHenckyAndExpHenckyEnergyPlot.tex}
	\end{tikzpicture}
	\caption{\label{fig:oneDimensionalHenckyAndExpHenckyEnergyPlot}The one-dimensional Hencky energy $\WH$ compared to the exponentiated Hencky energy $\WeH$ and the corresponding Cauchy stresses $\sigmaH$, $\sigmaeH$ for very large uniaxial stretches $\lambda$. Observe the non-convexity of $\WH$ and the non-invertibility of $\sigma_H$.}
\end{figure}
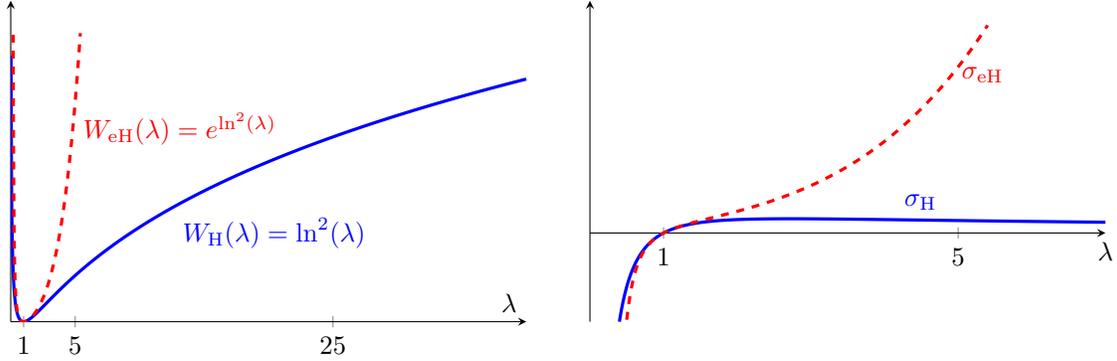

The exponentiated Hencky energy has many advantageous properties over the classical quadratic Hencky energy; for example, $\WeH$ is coercive on all Sobolev spaces $W^{1,p}$ for $1\leq p<\infty$, thus cavitation is excluded \cite{ball1982discontinuous,mueller1995cavitation}. In the planar case $n=2$, $\WeH$ is also polyconvex \cite{agn_neff2015exponentiatedII, agn_ghiba2015exponentiated} and thus Legendre-Hadamard-elliptic \cite{ball1976convexity}, whereas the classical Hencky energy is not even LH-elliptic (rank-one convex) outside a moderately large neighbourhood of $\id$ \cite{Bruhns01,Neff_Diss00} (see also \cite{hutchinson1981}, where the loss of ellipticity for energies of the form $\norm{\dev_3 \log U}^\beta$ with hardening index $0<\beta<1$ are investigated). Therefore, many results guaranteeing the existence of energy-minimizing  deformations for a variety of boundary value problems can be applied directly to $\WeH$ for $n=2$.

Furthermore, $\WeH$ satisfies a number of constitutive inequalities \cite{agn_neff2015exponentiatedI} such as the Baker-Ericksen inequality \cite{marsden1994foundations}, the pressure-compression inequality and the tension-extension inequality as well as Hill's inequality\footnote{Hill's inequality \cite{ogden1974rates} can be stated more generally as $\iprod{\ddtdot[\tau]-m\,[\tau\,D-D\,\tau],\, D} \geq 0$ in the hypoelastic formulation, where $\ddtdot$ is the Zaremba-Jaumann stress rate \eqref{eq:jaumannRate} and $\tau$ is the Kirchhoff stress tensor. For $m=0$, as {\v{S}}ilhav{\'y} explains, \quoteref{Hill's inequalities \quoteesc{\ldots} require the convexity of \quoteesc{the strain energy $W$} in \quoteesc{terms of the strain tensor $\log V$} \ldots This does not seem to contradict any theoretical or experimental evidence} \cite[p.~309]{silhavy1997mechanics}.} \cite{hill1970,ogden1970compressible,ogden1974rates}, which is equivalent to the convexity of the elastic energy with respect to the logarithmic strain tensor \cite{sidoroff1974}.

Moreover, for $\WeH$, the Cauchy-stress-stretch relation $V\mapsto \sigmaeH(V)$ is invertible (a property hitherto unknown for other hyperelastic formulations) and pure Cauchy shear stress corresponds to pure shear strain, as is the case in linear elasticity \cite{agn_neff2015exponentiatedI}. The physical meaning of Poisson's ratio \cite{poisson1829memoire,gercek2007poisson} $\nu=\frac{3\kappa-2\mu}{2(3\kappa+\mu)}$ is also similar to the linear case; for example, $\nu=\frac12$ directly corresponds to incompressibility of the material and $\nu=0$ implies that no lateral extension or contraction occurs in uniaxial tensions tests.

\subsection{Related geodesic distances}
The logarithmic distance measures obtained in Theorems \ref{theorem:mainResult} and \ref{theorem:separateMeasures} show a strong similarity to other geodesic distance measures on Lie groups. For example, consider the special orthogonal group $\SOn$ endowed with the canonical bi-invariant Riemannian metric\footnote{Note that $\muc\cdot\ghat$ is the restriction of our left-$\GLn$-invariant, right-$\On$-invariant metric $g$ (as defined in Section \ref{sectionContains:MetricDefinition}) to $\SOn$.}
\[
	\hat{g}_Q(X,Y) = \innerproduct{Q^T X, Q^T Y} = \innerproduct{X, Y}
\]
for $Q\in\SOn$ and $X,Y\in T_Q\SOn = Q\cdot\son$. Then the geodesic exponential at $\id\in\SOn$ is given by the matrix exponential on the Lie algebra $\son$, i.e.\ all geodesic curves are one-parameter groups of the form
\[
	\gammahat(t) = Q\cdot\exp(t\,A)
\]
with $Q\in\SOn$ and $A\in\son$ (cf.\ \cite{Moakher2002}). It is easy to show that the geodesic distance between $Q,R\in\SOn$ with respect to this metric is given by
\[
	\dgso(Q,R) = \norm{\log (Q^TR)}\,,
\]
where $\norm{\,.\,}$ is the Frobenius matrix norm and $\log\col\SOn\to\son$ denotes the principal matrix logarithm on $\SOn$, which is uniquely defined by the equality $\exp(\log Q)=Q$ and the requirement $\lambda_i(\log Q)\in (-\pi,\pi]$ for all $Q\in\SOn$ and all eigenvalues $\lambda_i(\log Q)$.

This result can be extended to the geodesic distance on the \emph{conformal special orthogonal group} $\CSOn$ consisting of all angle-preserving linear mappings: %
\[
	\CSOn \colonequals \{c\cdot Q \;|\; c>0\,, \; Q\in\SOn\}\,,
\]
where the bi-invariant metric $g_{\CSOn}$ is given by the canonical inner product:
\begin{equation}
	g_A^{\CSOn}(X,Y) = \innerproduct{A\inv X, A\inv Y}\,. \label{eq:CSOmetric}
\end{equation}
Then
\[
	\dgcso^2(c\cdot Q,d\cdot R) = \norm{\log (Q^TR)}^2 + \frac1n\left[\ln\left(\frac{c}{d}\right)\right]^2\,,
\]
where $\log$ again denotes the principal matrix logarithm on $\SOn$. Note that the punctured complex plane $\C\setminus\{0\}$ can be identified with $\CSO(2)$ via the mapping
\[
	z=a+i\, b\quad \mapsto \quad  Z\in \CSO(2) = \left\{ \matr{a&b\\-b&a} \,\bigg|\; a^2+b^2\neq0 \right\}\,.
\]

\subsection{Outlook}
While first applications of the exponentiated Hencky energy, which is based on the partial strain measures $\isomeas,\,\volmeas$ introduced here, show promising results, including an accurate modelling of so-called tire-derived material \cite{agn_montella2015exponentiated}, a more thorough fitting of the new parameter set to experimental data is necessary in order to assess the range of applicability of $\WeH$ towards elastic materials like vulcanized rubber. A different formulation in terms of the partial strain measures $\isomeas$ and $\volmeas$, i.e.\ an energy function of the form
\begin{equation}
\label{eq:measureFormulation}
	W(F) \;=\; \Psi(\isomeas(F), \volmeas(F))
	\;=\; \Psi(\norm{\dev_3\log U},\, \abs{\tr(\log U)})
\end{equation}
with $\Psi\col[0,\infty)^2\to[0,\infty)$, might even prove to be polyconvex in the three-dimensional case. The main open problem of finding a polyconvex (or rank-one convex) isochoric energy function\footnote{Ideally, the function $\widetilde{\Psi}$ should also satisfy additional requirements, such as monotonicity, convexity and exponential growth.} $F\mapsto \widetilde{\Psi}(\norm{\dev_3\log U})$ has also been considered by Sendova and Walton \cite{sendova2005strong}.
Note that while every isotropic elastic energy $W$ can be expressed as $W(F)=h(K_1,K_2,K_3)$ with \emph{Criscione's invariants}\footnote{The invariants $K_1$ and $K_2^2=\tr\big((\dev_3\log U)^2\big)$ as well as $\widetilde{K}_3 = \tr\big((\dev_3\log U)^3\big)$ had already been discussed exhaustively by H.\ Richter in a 1949 ZAMM article \cite[\S4]{richter1949verzerrung}, while $K_1$ and $K_2$ have also been considered by A.I.\ Lurie \cite[p.~189]{lurie2012nonlinear}. Criscione has shown that the invariants given in \eqref{eq:criscione} enjoy a favourable orthogonality condition which is useful when determining material parameters.} \cite{criscione2000invariant,criscione2002direct,diani2005combining,wilber2005baker}
\begin{equation}
	K_1 = \tr(\log U)\,, \qquad K_2 = \norm{\dev_3\log U} \qquad\text{ and }\qquad K_3=\det\left( \frac{\dev_3\log U}{\norm{\dev_3\log U}} \right) \,, \label{eq:criscione} %
\end{equation}
not every elastic energy has a representation of the form \eqref{eq:measureFormulation}; for example, \eqref{eq:measureFormulation} implies the \emph{tension-compression symmetry}\footnote{The tension-compression symmetry is often expressed as $\tau(V\inv)=-\tau(V)$, where $\tau(V)$ is the Kirchhoff stress tensor corresponding to the left Biot stretch $V$. This condition, which is the natural nonlinear counterpart of the equality $\sigma(-\eps)=-\sigma(\eps)$ in linear elasticity, is equivalent to the condition $W(F\inv)=W(F)$ for hyperelastic constitutive models.} $W(F) = W(F\inv)$, which is not necessarily satisfied by energy functions in general.\footnote{Truesdell and Noll \cite[p.~174]{truesdell65} argue that \quoteref{\ldots there is no foundation for the widespread belief that according to the theory of elasticity, pressure and tension have equal but opposite effects}. Examples for isotropic energy functions which do not satisfy this symmetry condition in general but only in the incompressible case can be found in \cite{henann2009large}. For an \emph{idealized} isotropic elastic material, however, the tension-compression symmetry is a \emph{plausible} requirement (with an obvious additive counterpart in linear elasticity), especially for incompressible bodies.} In terms of the \emph{Shield transformation}\footnote{Further properties of the Shield transformation can be found in \cite[p.288]{silhavy1997mechanics}; for example, it preserves the polyconvexity, quasiconvexity and rank-one convexity of the original energy.} \cite{shield1967inverse,carroll2005implications}
\[
	W^*(F) \colonequals \det F \cdot W(F\inv)\,,
\]
the tension-compression symmetry amounts to the requirement $\frac{1}{\det F}\,W^*(F) = W(F)$ or, for incompressible materials, $W^*(F)=W(F)$. Moreover, under the assumption of incompressibility, the symmetry can be immediately extended to arbitrary deformations $\varphi\col\Omega\to\varphi(\Omega)$ and $\varphi\inv\col\varphi(\Omega)\to\Omega$: if $\det\grad\varphi\equiv1$, we can apply the substitution rule to find
\begin{align*}
	\int_{\varphi(\Omega)} W(\grad(\varphi\inv)(x)) \,\dx &= \int_\Omega W(\grad(\varphi\inv)(\varphi(x))) \cdot \abs{\det\grad\varphi(x)}\,\dx\\
	&= \int_\Omega W(\grad\varphi(x)\inv)\,\dx = \int_\Omega W(\grad\varphi(x))\,\dx
\end{align*}
if $W(F\inv)=W(F)$ for all $F\in\SLn$, thus the total energies of the deformations $\varphi,\varphi\inv$ are equal, cf.\ Figure \ref{fig:tensionCompressionSymmetry}.

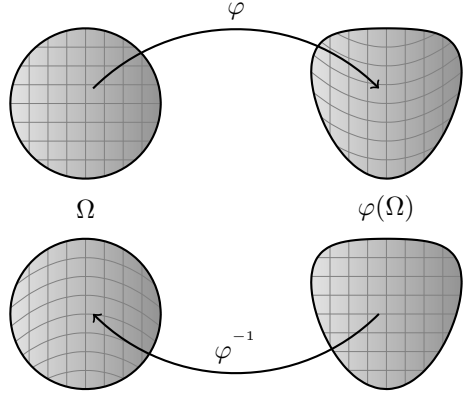
\begin{figure}
	\centering
	\tikzsetnextfilename{tensionCompressionSymmetry}
	\begin{tikzpicture}
		\input{tikz/tensionCompressionSymmetry.tex}
	\end{tikzpicture}
	\caption{\label{fig:tensionCompressionSymmetry}The tension-compression symmetry for incompressible materials: if $\det \grad\varphi \equiv 1$ and $W(F\inv)=W(F)$ for all $F\in\SLn$, then $\int_\Omega W(\grad\varphi(x)) \,\dx = \int_{\varphi(\Omega)} W(\grad(\varphi\inv)(x)) \,\dx$.}
\end{figure}

Since the function
\[
	F\mapsto e^{\norm{\dev_2\log U}^2} = e^{\dgisotwo^2 \left( \frac{F}{\det F^{\afrac12}},\; \SO(2) \right)}
\]
in planar elasticity is polyconvex \cite{agn_neff2015exponentiatedII, agn_ghiba2015exponentiated}, it stands to reason that a similar formulation in the three-dimensional case might prove to be polyconvex as well. A first step towards finding such an energy is to identify where the function $W$ with
\begin{equation}
	W(F) = e^{\norm{\dev_3\log U}^2} = e^{\dgisothree^2 \left( \frac{F}{\det F^{\afrac13}},\; \SO(3) \right)}\,,
\end{equation}
which is not rank-one convex \cite{agn_neff2015exponentiatedI}, loses its ellipticity properties. For that purpose, it may be useful to consider the \emph{quasiconvex hull} of $W$. There already are a number of promising results for similar energy functions; for example, the quasiconvex hull of the mapping
\[
	F\mapsto \disteuclid^2(F,\SO(2)) = \norm{U-\id}^2
\]
can be explicitly computed \cite{silhavy2001rank,dolzmann2012,dolzmann2013}, and the quasiconvex hull of the similar Saint-Venant-Kirchhoff energy $\WSVK(F) = \frac{\mu}{4}\,\norm{C-\id}^2 + \frac{\lambda}{8}\, [\tr(C-\id)]^2$ has been given by Le Dret and Raoult \cite{ledret1995quasiconvex}. For the mappings
\[
	F\mapsto\disteuclid^2(F,\SO(3)) \quad\text{ or }\quad F\mapsto \dg^2(F,\SO(n))
\]
with $n\geq2$, however, no explicit representation of the quasiconvex hull is yet known, although it has been shown that both expressions are not rank-one convex \cite{bertram2007rank}.

It might also be of interest to calculate the geodesic distance $\dg(A,B)$ for a larger class of matrices $A,B\in\GLpn$:\footnote{An improved understanding of the geometric structure of mechanical problems could, for example, help to develop new discretization methods \cite{sander2015finite,grohs2013optimal}.} although Theorem \ref{theorem:mainResult} allows us to explicitly compute the distance $\dg(\id, P)$ for $P\in\PSymn$ and local results are available for certain special cases \cite{agn_martin2014minimal}, it is an open question whether there is a general formula for the distance $\dgfull(Q,R)$ between arbitrary rotations $R,Q\in\SOn$ for all parameters $\mu,\muc,\kappa>0$. Since restricting our left-$\GLn$-invariant, right-$\On$-invariant metric on $\GLn$ to $\SOn$ yields a multiple of the canonical bi-$\SOn$-invariant metric on $\SOn$, we can compute
\[
	\dgfull^2(Q,R) = \muc\cdot\dgso^2(Q,R) = \muc\,\norm{\log (Q^TR)}^2
\]
if for all $Q,R\in\SOn$ a shortest geodesic in $\GLpn$ connecting $Q$ and $R$ is already contained within $\SOn$, cf.\ Figure \ref{fig:geodesicConvexity}. However, whether this is the case depends on the chosen parameters $\mu,\muc$; a general closed-form solution for $\dgfull$ on $\SOn$ is therefore not yet known \cite{agn_martin2016SO}.

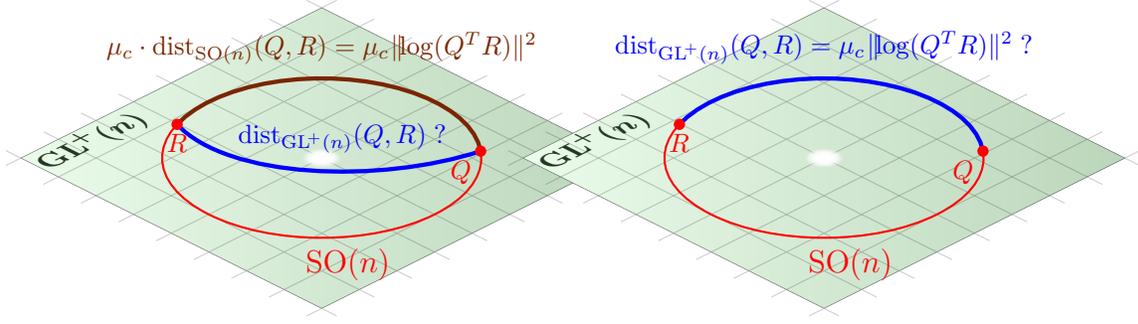
\begin{figure}[h]
	\centering
	\hspace*{-14mm}
	\tikzsetnextfilename{geodesicConvexity}
	\begin{tikzpicture}
		\input{tikz/geodesicConvexity.tex}
	\end{tikzpicture}
	\caption{\label{fig:geodesicConvexity}If $\SOn$ contains a length minimizing geodesic connecting $Q,R\in\SOn$ with respect to our left-$\GLn$-invariant, right-$\On$-invariant metric $g$ on $\GLn$, then the $\GLpn$-geodesic distance between $Q$ and $R$ is equal to the well-known $\SOn$-geodesic distance $\muc\,\norm{\log (Q^TR)}^2$.}%
\end{figure}

Moreover, it is not known whether our result can be generalized to \emph{anisotropic} Riemannian metrics, i.e.\ if the geodesic distance to $\SOn$ can be explicitly computed for a larger class of left-$\GLn$-invariant Riemannian metrics which are not necessarily right-$\On$-invariant. A result in this direction would have immediate impact on the modelling of finite strain anisotropic elasticity \cite{agn_balzani2006polyconvex,agn_schroder2005variational,agn_schroder2008anisotropic}. The difficulties with such an extension are twofold: one needs a representation formula for Riemannian metrics which are right-invariant under a given symmetry subgroup of $\On$, as well as an understanding of the corresponding geodesic curves.
\section{Conclusion}
\label{section:conclusion}

We have shown that the squared geodesic distance of the (finite) deformation gradient $F\in\GLpn$ to the special orthogonal group $\SOn$ is the quadratic isotropic Hencky strain energy:
\[
	\dg^2(F,\SOn) = \mu \, \norm{\dev_n \log U}^2+\frac{\kappa}{2}\,[\tr(\log U)]^2\,,
\]
if the general linear group is endowed with the left-$\GLn$-invariant, right-$\On$-invariant Riemannian metric $g_A(X,Y) = \isoprod{A\inv X, A\inv Y}\,$, where
\[
	\isoprod{X,Y} = \mu\,\iprod{\dev_n\sym X, \dev_n\sym Y} + \muc\,\iprod{\skew X, \skew Y} + \tfrac{\kappa}{2}\tr(X) \tr(Y)
\]
with $\iprod{X,Y} = \tr(X^TY)$. Furthermore, the (partial) logarithmic strain measures
\[
	\isomeas = \norm{\dev_n \log U} = \norm{\dev_n \log \sqrt{F^TF}} \quad\text{ and }\quad \volmeas = \abs{\tr(\log U)} = \abs{\tr(\log\sqrt{F^TF})}
\]
have been characterized as the geodesic distance of $F$ to the special orthogonal group $\SOn$ and the identity tensor $\id$, respectively:
\begin{alignat*}{2}
	\isomeas &= \norm{\dev_n \log U}& &= \dgiso \left( \frac{F}{\det F^{\afrac1n}},\; \SOn \right)\,,\\
	\volmeas &= \;\;\abs{\tr(\log U)}& &= \sqrt{n}\cdot \dgvol \left( (\det F)^{\afrac1n}\cdot\id,\; \id \right)\,,
\end{alignat*}
where the geodesic distances on $\SLn$ and $\R^+\cdot\id$ are induced by the canonical left invariant metric $\bar{g}_A(X,Y) = \iprod{A\inv X, A\inv Y}$.

We thereby show that the two quantities $\isomeas=\norm{\dev_n \log U}$ and $\volmeas=\abs{\tr(\log U)}$ are purely geometric properties of the deformation gradient $F$, similar to the invariants $\norm{\dev_n \eps}$ and $\abs{\tr(\eps)}$ of the infinitesimal strain tensor $\eps$ in the linearized setting.

While there have been prior attempts to deductively motivate the use of logarithmic strain in nonlinear elasticity theory, these attempts have usually focussed on
the logarithmic \emph{Hencky strain tensor} $E_0=\log U$ (or $\Ehat_0=\log V$) and its status as the \enquote{natural} material (or spatial) strain tensor in isotropic elasticity.
We discussed, for example, a well-known characterization of $\log V$ in the hypoelastic context: if the strain rate $\ddtsquare$ is objective \emph{as well as corotational}, and if
\[
	\ddtsquare[\Ehat] = D \colonequals \sym(\dot{F}F\inv)
\]
for some strain tensor $\Ehat$, then $\ddtsquare=\ddtlog$ must be the logarithmic rate and $\Ehat=\Ehat_0=\log V$ must be the spatial Hencky strain tensor.

However, as discussed in Section \ref{section:actualIntroduction}, all \emph{strain tensors} are interchangeable:
the choice of a specific strain tensor in which a constitutive law is to be expressed is not a restriction on the available constitutive relations.
Such an approach can therefore not be applied to deduce necessary conditions or a priori properties of constitutive laws.

Our deductive approach, on the other hand, directly motivates the use of the \emph{strain measures} $\isomeas$ and $\volmeas$ from purely differential geometric observations. As we have indicated, the requirement that a constitutive law depends only on $\isomeas$ and $\volmeas$ has direct implications; for example, the tension-compression symmetry $W(F)=W(F\inv)$ is satisfied by every hyperelastic potential $W$ which can be expressed in terms of $\isomeas$ and $\volmeas$ alone.

Moreover, as demonstrated in Section \ref{section:alternativeMotivations}, similar approaches oftentimes \emph{presuppose} the role of the positive definite factor $U=\sqrt{F^TF}$ as the sole measure of the deformation, whereas this independence from the orthogonal polar factor is obtained \emph{deductively} in our approach %
(cf.\ Table \ref{table:summary}). %

\def\arraystretch{2}
\setlength\tabcolsep{3.5pt}
\begin{table}[h]
\centering
{\small
\begin{tabular}{c|c|c}%
& Measure of deformation deduced & Measure of deformation postulated
\\[.75em]
\hline \rule{0pt}{6ex}
linear
& $\begin{aligned} &\fulldisteuclid^2(\grad u, \son)\\&\quad= \mu\,\norm{\dev_n\sym\grad u}^2 + \frac{\kappa}{2}\, [\tr(\sym\grad u)]^2 \end{aligned}$
& $\begin{aligned} &\fulldisteuclidSym^2(\eps,0)\\&\quad= \mu\,\norm{\dev_n\eps}^2 + \frac{\kappa}{2}\, [\tr(\eps)]^2 \end{aligned}$
\\[.75em]
\hline \rule{0pt}{5ex}
$\begin{aligned} &\text{geometrically}\\&\;\;\;\text{nonlinear}\end{aligned}$
& $\begin{aligned} &\disteuclid^2(F, \SOn) = \mu\,\norm{\sqrt{F^TF}-\id}^2 \end{aligned}$
& $\begin{aligned} &\dist_{\mathrm{Euclid},\Symn}^2(U,\id) = \mu\,\norm{U-\id}^2 \end{aligned}$
\\[.75em]
\hline \rule[-5ex]{0pt}{12ex}
$\begin{aligned} &\text{geometrically}\\&\;\;\;\text{nonlinear}\\&\;\,\,\text{(weighted)}\tablefootnote{Observe that $\norm{\dev_n(U-\id)}^2$ does not measure the isochoric (distortional) part $\frac{F}{(\det F)^{\afrac1n}}$ of $F$.}\end{aligned}$
& not well defined
& $\begin{aligned} &\fulldisteuclidSym^2(U,\id)\\&\quad=  \mu\,\norm{\dev_n(U-\id)}^2 + \frac{\kappa}{2}\,[\tr(U-\id)]^2 \end{aligned}$
\\[.75em]
\hline \rule[-5ex]{0pt}{12ex}
geodesic
& $\begin{aligned} &\fulldistgeod^2(F, \SOn)\\&\quad= \mu\,\norm{\dev_n\log(\sqrt{F^TF})}^2 + \frac{\kappa}{2}\,[\tr(\log \sqrt{F^TF})]^2 \end{aligned}$
& $\begin{aligned} &\fulldistgeodPsym^2(U, \id)\\&\quad= \mu\,\norm{\dev_n\log U}^2 + \frac{\kappa}{2}\,[\tr(\log U)]^2 \end{aligned}$
\\[.75em]
\hline \rule{0pt}{9ex}
log-Euclidean
& not well defined
& $\begin{aligned} &\dist_{\text{\rm{Log-Euclid}},\mu,\kappa}^2(U, \id)\\&\quad= \fulldisteuclidSym^2(\log U,0)\\&\quad= \mu\,\norm{\dev_n \log U}^2 + \frac{\kappa}{2}\,[\tr(\log U)]^2 \end{aligned}$
\\[.75em]
\end{tabular}
}
\caption{\label{table:summary}Different approaches towards the motivation of different strain tensors and strain measures.}
\end{table}

Note also that the specific distance measure $\dg$ on $\GLpn$ used here is not chosen arbitrarily: the requirements of left-$\GLn$-invariance and right-$\On$-invariance, which have been motivated by mechanical considerations, uniquely determine $g$ up to the three parameters $\mu,\muc,\kappa>0$. This uniqueness property further emphasizes the generality of our results, which yet again strongly suggest that Hencky's constitutive law should be considered the idealized nonlinear model of elasticity for very small strains outside the infinitesimal range.

\section*{Acknowledgements}
The second author acknowledges support by the Deutsche Forschungsgemeinschaft (DFG) through a Heisenberg fellowship under grant EI 453/2-1.

We are grateful to Prof.\ Alexander Mielke (Weierstra\ss-Institut, Berlin) for pertinent discussions on geodesics in $\GL(n)$; the first parametrization of geodesic curves on $\SLn$ known to us is due to him \cite{Mielke2002}. We also thank Prof. Robert Bryant (Duke University) for his helpful remarks regarding geodesics on Lie groups and invariances of inner products on $\gln$, as well as a number of friends who helped us with the draft.

We also thank Dr.~Andreas Fischle (Technische Universität Dresden) who, during long discussions on continuum mechanics and differential geometry, inspired many of the ideas laid out in this paper.

The first author had the great honour of presenting the main ideas of this paper to Richard Toupin on the occasion of the Canadian Conference on Nonlinear Solid Mechanics 2013 in the mini-symposium organized by Francesco dell'Isola and David J.\ Steigmann, which was dedicated to Toupin.

\section*{Conflict of Interest}
The authors declare that they have no conflict of interest.

\nocite{agn_neff2016henckylatin}
{\footnotesize
\printbibliography
}

\newpage
\begin{appendix}
\section{Appendix}
\subsection{Notation}
{ %
\vspace*{.75em}
\begin{itemize}[leftmargin=.75em]
\setlength{\itemsep}{.4032em}
	\item $\R$ is the set of \emph{real numbers},
	\item $\R^+=(0,\infty)$ is the set of \emph{positive real numbers},
	\item $\R^n$ is the set of real \emph{column vectors} of length $n$,
	\item $\R^{n\times m}$ is the set of real $n\times m$-\emph{matrices},
	\item $\id$ is the \emph{identity tensor}; %
	\item $X^T$ is the \emph{transpose} of a matrix $X\in\Rnm$,
	\item $\tr(X) = \sum_{i=1}^n X_{i,i}$ is the \emph{trace} of $X\in\Rnn$,
	\item $\Cof X$ is the \emph{cofactor} of $X\in\Rnn$,
	\item $\innerproduct{X,Y}=\tr(X^TY)=\sum_{i,j=1}^n X_{i,j}Y_{i,j}$ is the \emph{canonical inner product} on $\Rnn$,
	\item $\norm{X}=\sqrt{\iprod{X,X}}$ is the \emph{Frobenius matrix norm} on $\Rnn$,
	\item $\sym X = \half(X+X^T)$ is the \emph{symmetric part} of $X\in\Rnn$,
	\item $\skew X = \half(X-X^T)$ is the \emph{skew-symmetric part} of $X\in\Rnn$,
	\item $\dev_n X = X-\frac1n\tr(X)\cdot\id$ is the $n$-dimensional \emph{deviator} of $X\in\Rnn$,
	\item $\isoprod{X,Y}=\mu\,\iprod{\dev_n\sym X, \dev_n\sym Y} + \muc\,\iprod{\skew X, \skew Y} + \tfrac{\kappa}{2}\tr(X) \tr(Y)$ is the \emph{weighted inner product} on $\Rnn$,
	\item $\isonorm{X}=\sqrt{\isoprod{X,X}}$ is the \emph{weighted Frobenius norm} on $\Rnn$,
	\item $\GLn = \{A\in\Rnn \setvert \det A \neq 0\}$ is the \emph{general linear group} of all invertible $A\in\Rnn$,
	\item $\GLpn = \{A\in\Rnn \setvert \det A > 0\}$ is the \emph{identity component} of $\GLn$,
	\item $\SLn = \{A\in\Rnn \setvert \det A = 1\}$ is the \emph{special linear group} of all $A\in\GLn$ with $\det A = 1$,
	\item $\On$ is the \emph{orthogonal group} of all $Q\in\Rnn$ with $Q^TQ=\id$,
	\item $\SOn$ is the \emph{special orthogonal group} of all $Q\in\On$ with $\det Q = 1$,
	\item $\Symn$ is the set of \emph{symmetric}, real $n\times n$-matrices, i.e.\ $S^T=S$ for all $S\in\Symn$,
	\item $\PSymn$ is the set of \emph{positive definite}, symmetric, real $n\times n$-matrices, i.e.\ $x^TPx > 0$ for all $P\in\PSymn,\: 0\neq x \in\R^n$,
	\item $\gln=\Rnn$ is the Lie algebra of all real $n\times n$-matrices,
	\item $\son = \{W\in\Rnn \setvert W^T=-W\}$ is the Lie algebra of \emph{skew symmetric}, real $n\times n$-matrices,
	\item $\sln = \{X\in\Rnn \setvert \tr(X) = 0\}$ is the Lie algebra of \emph{trace free}, real $n\times n$-matrices, i.e.\ $\tr(X)=0$ for all $X\in\sln$,
	\item $\Omega\subset\R^n$ is the \emph{reference configuration} of an elastic body,
	\item $\grad \varphi=D\varphi$ is the \emph{first derivative} of a differentiable function $\varphi\col\Omega\subset\R^n\to\R^n$, often called the \emph{deformation gradient},
	\item $\curl v$ denotes the curl of a vector valued function $v\col\R^3\to\R^3$,
	\item $\Curl p$ denotes the curl of a matrix valued function $p\col\R^3\to\R^{3\times3}$, taken row-wise,
	\item $\varphi\col\Omega\to\R^n$ is a continuously differentiable \emph{deformation} with $\grad\varphi(x)\in\GLpn$ for all $x\in\Omega$,
	\item $F=\grad\varphi\in\GLpn$ is the \emph{deformation gradient},
	\item $U=\sqrt{F^TF}\in\PSymn$ is the \emph{right Biot-stretch tensor},
	\item $V=\sqrt{FF^T}\in\PSymn$ is the \emph{left Biot-stretch tensor},
	\item $B=FF^T=V^2$ is the \emph{Finger tensor},
	\item $C=F^TF=U^2$ is the \emph{right Cauchy-Green deformation tensor},
	\item $F=R\,U=V\,R$ is the \emph{polar decomposition} of $F$ with $R=\polar(F)\in\SOn$,
	\item $E_0=\log U$ is the \emph{material Hencky strain tensor},
	\item $\Ehat_0=\log V$ is the \emph{spatial Hencky strain tensor},
	\item $S_1 = D_F W(F)$ is the \emph{first Piola-Kirchhoff stress} corresponding to an elastic energy $W=W(F)$,
	\item $S_2 = F\inv\,S_1 = 2\,D_C W(C)$ is the \emph{second Piola-Kirchhoff stress} corresponding to an elastic energy $W=W(C)$ (Doyle-Ericksen formula),
	\item $\tau = S_1 \,F^T = D_{\log V} W(\log V)$ \cite[p.~116]{lurie2012nonlinear} is the \emph{Kirchhoff stress tensor},
	\item $\sigma = \frac{1}{\det F}\,\tau$ is the \emph{Cauchy stress tensor},
	\item $\Biot = U \,S_2 = D_U W(U)$ is the \emph{Biot stress tensor} corresponding to an elastic energy $W=W(U)$,
	\item $L=\dot{F} F\inv$ is the \emph{spatial velocity gradient},
	\item $D=\sym L$ is the \emph{rate of stretching} or \emph{spatial strain rate tensor},
	\item $W=\skew L$ is the \emph{spatial continuum spin}.
\end{itemize}
\renewcommand*\labelitemi{\textbullet}
}
\newpage
\subsection{Linear stress-strain relations in nonlinear elasticity}
\label{section:linearRelations}
Many constitutive laws commonly used in applications are expressed in terms of linear relations between certain strains and stresses, including Hill's family of generalized linear elasticity laws (cf.\ Section \ref{section:hypoelasticity}) of the form
\begin{equation}
	T_r = 2\,\mu\, E_r + \lambda\,\tr(E_r)\cdot\id
\end{equation}
with work-conjugate pairs $(T_r,E_r)$ based on the Lagrangian strain measures given in \eqref{eq:sethHillFamily}. A widely known example of such a constitutive law is the hyperelastic \emph{Saint-Venant-Kirchhoff model}
\[
	S_2 = 2\,\mu\, E_1 + \lambda\,\tr(E_1)\, \id = \mu\,(C-\id) + \frac\lambda2\, \tr(C-\id)\cdot\id
\]
for $r=1$ and $T_1=S_2$, where $S_2$ denotes the second Piola-Kirchhoff stress tensor.
Similarly, a number of elasticity laws can be written in the form
\[
	\That_r = 2\,\mu\, \Ehat_r + \lambda\,\tr(\Ehat_r)\cdot\id
\]
with a spatial strain tensor $\Ehat_r$ and a corresponding stress tensor $\That_r$. Examples include the \emph{Neo-Hooke type model}
\[
	\sigma = 2\,\mu\, \Ehat_1 + \lambda\,\tr(\Ehat_1)\, \id = \mu\,(B-\id) + \frac\lambda2\, \tr(B-\id)\cdot\id
\]
for $r=1$, where $T_1=\sigma$ is the Cauchy stress tensor, the \emph{Almansi-Signorini model}
\[
	\sigma = 2\,\mu\, \Ehat_{-1} + \lambda\,\tr(\Ehat_{-1})\, \id = \mu\,(\id-B\inv) + \frac\lambda2\, \tr(\id-B\inv)\cdot\id
\]
for $r=-1$ and $T_{-1}=\sigma$, as well as the hyperelastic Hencky model
\[
	\tau = 2\,\mu\,\log V + \lambda\,\tr(\log V)\cdot \id
\]
for $r=0$ and $\That_0=\tau$. A thorough comparison of these four constitutive laws can be found in \cite{batra2001comparison}.

Another example of a postulated linear stress-strain relation is the model
\[
	\Biot = 2\,\mu\,\log U + \lambda\,\tr(\log U)\cdot \id\,,
\]
where $\Biot$ denotes the \emph{Biot stress tensor}, which measures the \quoteref{stress per unit initial area before deformation} \cite{biot1939}. This constitutive relation was first given in an 1893 article by the geologist G.\,F.\ Becker \cite{becker1893,neff2014becker}, who deduced it from a law of superposition in an approach similar to that of H.\ Hencky. The same constitutive law was considered by Carroll \cite{carroll2009} as an example to emphasize the necessity of a hyperelastic formulation in order to ensure physical plausibility in the description of elastic behaviour. Note that of the constitutive relations listed in this section, only the Hencky model and the Saint-Venant-Kirchhoff model are indeed hyperelastic (cf.\ \cite[Chapter 7.4]{bertram2008elasticity}).
\subsection{Tensors and tangent spaces}
\label{section:tensorsAndTangentSpaces}
In the more general setting of differential geometry, the linear mappings $F,U,C,V,B$ and $R$ as well as various stresses at a single point $x$ in an elastic body $\Omega$ are defined as mappings between different \emph{tangent spaces}: for a point $x\in\Omega$ and a deformation $\varphi$, we must then distinguish between the two tangent spaces $T_x\Omega$ and $T_{\varphi(x)}\varphi(\Omega)$. The domains and codomains of various linear mappings are listed below and indicated in Figure \ref{fig:tensorsAndTangentSpaces}. Note that we do not distinguish between tangent and cotangent vector spaces (cf.\ \cite{Federico2015}).
\begin{alignat*}{2}
	F,R&:\quad& T_{x}\Omega \,&\to\, T_{\varphi(x)}\varphi(\Omega)\,,\\
	U,C&:\quad& T_{x}\Omega \,&\to\, T_{x}\Omega\,,\\
	V,B&:\quad& T_{\varphi(x)}\varphi(\Omega) \,&\to\, T_{\varphi(x)}\varphi(\Omega)\,.
\end{alignat*}
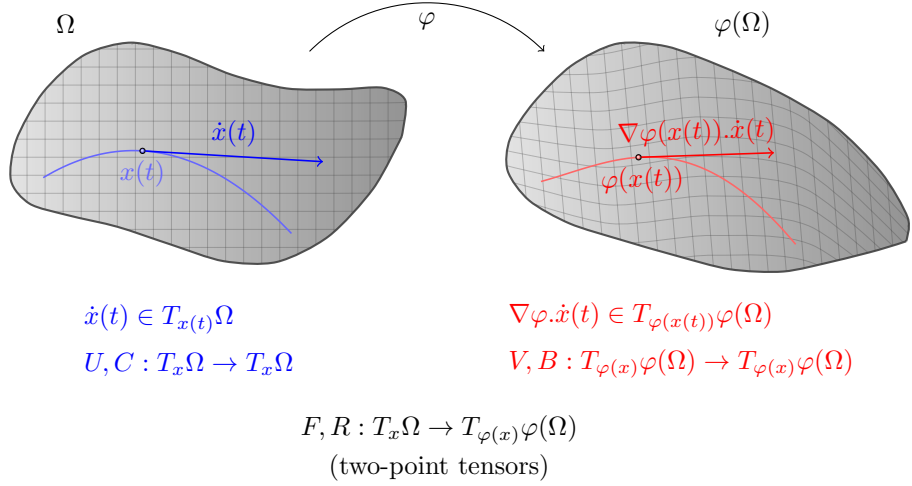
\begin{figure}[h]
	\centering
	\tikzsetnextfilename{tensorsAndTangentSpaces}
	\begin{tikzpicture}[scale=1.2]
		\input{tikz/tensorsAndTangentSpaces.tex}
	\end{tikzpicture}
	\caption{\label{fig:tensorsAndTangentSpaces}Various linear mappings between the tangent spaces $T_{x}\Omega$ and $T_{\varphi(x)}\varphi(\Omega)$.}
\end{figure}

The right Cauchy-Green tensor $C=F^TF$, in particular, is often interpreted as a \emph{Riemannian metric} on $\Omega$; Epstein \cite[p.~113]{epstein2010geometrical} explains that \quoteref{the right Cauchy-Green tensor is precisely the pull-back of the spatial metric to the body manifold}, cf.\ \cite{marsden1994foundations}. If $\Omega$ and $\varphi(\Omega)$ are embedded in the Euclidean space $\R^n$, this connection can immediately be seen: while the length of a curve $x\col[0,1]\to\Omega$ is given by $\int_0^1 \sqrt{\iprod{\xdot,\xdot}}\,\dt$, where $\iprod{\cdot,\cdot}$ is the canonical inner product on $\R^n$, the length of the deformed curve $\varphi\circ x$ is given by (cf.\ Figure \ref{fig:tensorsAndTangentSpaces})
\begin{align*}
	\int_0^1 \sqrt{\iprod{\smallddt(\varphi\circ x),\;\smallddt(\varphi\circ x)}} \,\dt = \int_0^1 \sqrt{\iprod{F(x)\,\xdot,\,F(x)\,\xdot}} \,\dt = \int_0^1 \sqrt{\iprod{C(x)\,\xdot,\,\xdot}} \,\dt\,.
\end{align*}
The quadratic form $g_x(v,v) = \iprod{C(x)\, v,\,v}$ at $x\in\Omega$ therefore measures the length of the deformed line element $Fv$ at $\varphi(x)\in\varphi(\Omega)$. Thus locally,
\[
	\disteuclidPhiOmega(\varphi(x),\varphi(y)) = \dgOmega(x,y)\,,
\]
where $\disteuclidPhiOmega(\varphi(x),\varphi(y)) = \norm{\varphi(x)-\varphi(y)}$ is the Euclidean distance between $\varphi(x),\varphi(y)\in\varphi(\Omega)$ and $\dgOmega(x,y)$ denotes the geodesic distance between $x,y\in\Omega$ with respect to the Riemannian metric $g_x(v,w) = \iprod{C(x)\, v,\,w}$.

Moreover, this interpretation characterizes the Green-Lagrangian strain tensor $E_1 = \frac12(C-\id)$ as a measure of \emph{change in length}: the difference between the squared length of a line element $v\in T_x\Omega$ in the reference configuration and the squared length of the deformed line element $F(x)\,v\in T_{\varphi(x)}\varphi(\Omega)$ is given by
\[
	\norm{F(x)\,v}^2 - \norm{v}^2 = \iprod{C(x)\, v,\,v} - \iprod{v,v} = \iprod{(C(x)-\id)\, v,\,v} = 2\,\iprod{E_1(x)\, v,\,v}\,,
\]
where $\norm{\,.\,}$ denotes the Euclidean norm on $\R^n$.
Note that for $F(x)=\id+\grad u(x)$ with the displacement gradient $\grad u(x)$, the expression $\norm{F(x)\,v}^2$ can be linearized to
\begin{align*}
	\norm{F(x)\,v}^2 &= \norm{(\id+\grad u(x))\,v}^2 = \iprod{(\id+\grad u(x))\,v, \, (\id+\grad u(x))\,v}\\
	&= \iprod{v,v} + 2\,\iprod{\grad u(x)\,v,\, v} + \iprod{\grad u(x)\,v,\, \grad u(x)\,v}\\
	&= \norm{v}^2 + 2\,\iprod{\sym\grad u(x)\,v,\, v} + \norm{\grad u(x)\,v}^2\\
	&= \norm{v}^2 + 2\,\iprod{\sym\grad u(x)\,v,\, v} + \hot\,,
\end{align*}
where $\hot$ denotes higher order terms with respect to $\grad u(x)$. Thus
\[
	\norm{F(x)\,v}^2 - \norm{v}^2 = 2\,\iprod{\eps(x)\,v,\, v} + \hot\,,
\]
where $\eps=\sym\grad u$ is the linear strain tensor.
\subsection{Additional computations}
\label{appendix:cofactorComputation}
Let $\Cof F = (\det F) \cdot F^{-T}$ denote the cofactor of $F\in\GLpn$. Then the geodesic distance of $\Cof F$ to $\SOn$ with respect to the Riemannian metric $g$ introduced in \eqref{eq:invariantRiemannianMetricDefinition} can be computed directly by applying Theorem \ref{theorem:mainResult}:
\begin{align*}
	&\dg^2(\Cof F, \SOn)\\
	&\hspace{1.96em}= \mu\,\norm{\dev_n \log \sqrt{(\Cof F)^T \Cof F}}^2 + \frac\kappa2\,\Big[\tr\Big(\log \sqrt{(\Cof F)^T \Cof F}\Big)\Big]^2\\
	&\hspace{1.96em}= \mu\,\norm{\dev_n \log \sqrt{(\det F)^2\cdot F\inv F^{-T}}}^2 + \frac\kappa2\,\Big[\tr\Big(\log \sqrt{(\det F)^2\cdot F\inv F^{-T}}\Big)\Big]^2\\
	&\hspace{1.96em}= \mu\,\norm{\dev_n \log \sqrt{F\inv F^{-T}}}^2 + \frac\kappa2\,\Big[\tr\Big(\log \big((\det F)\cdot\id\big) + \log \sqrt{F\inv F^{-T}}\Big)\Big]^2\\
	&\hspace{1.96em}= \mu\,\norm{\dev_n \log (U\inv)}^2 + \frac\kappa2\,[\tr\big((\ln\det F)\cdot\id + \log(U\inv)\big)]^2\\
	&\hspace{1.96em}= \mu\,\norm{-\dev_n \log U}^2 + \frac\kappa2\,[n\cdot(\ln\det U) - \tr(\log U)]^2\\
	&\hspace{1.96em}= \mu\,\norm{\dev_n \log U}^2 + \frac{\kappa\,(n-1)^2}{2}\,[\tr(\log U)]^2\,.
\end{align*}
\subsection[The principal matrix logarithm on $\PSymn$ and the matrix exponential]{\boldmath The principal matrix logarithm on $\PSymn$ and the matrix exponential}
The following lemma states some basic computational rules for the matrix exponential $\exp\col\Rnn\to\GLpn$ and the \emph{principal matrix logarithm} $\log\col\PSymn\to\Symn$ involving the trace operator $\tr$ and the deviatoric part $\dev_n X = X - \frac{\tr(X)}{n}\cdot\id$ of a matrix $X\in\Rnn$.
\begin{lemma}
\label{lemma:logRules}
Let $X\in\Rnn$, $P\in\PSymn$ and $c>0$. Then
\begin{alignat*}{2}
	&\text{i)}& \det(\exp(X)) &= e^{\tr(X)}\,,\\
	&\text{ii)}& \exp(\dev_n X) &= e^{-\frac{\tr(X)}{n}}\cdot \exp(X)\,,\\
	&\text{iii)}& \log(c\cdot \id) &= \ln(c)\cdot\log(\id)\,,\\
	&\text{iv)}&\quad \log((\det P)^{-\afrac1n}\cdot P) &= \log P - \tfrac{\ln(\det P)}{n}\cdot\id = \dev_n \log P\,.
\end{alignat*}
\end{lemma}
\begin{proof}
Equality i) is well known (see e.g.\ \cite{Bernstein2009}). Equality $iii)$ follows directly from the fact that $\exp\col\Symn\to\PSymn$ is bijective and that $\exp(\ln(c)\cdot\id) = e^{\ln(c)}\cdot\id=c\cdot \id$. Since $AB=BA$ implies $\exp(AB)=\exp(A)\exp(B)$, we find
\[
	\exp(\dev_n X) = \exp(X-\frac{\tr(X)}{n}\cdot\id) = \exp(X)\cdot\exp(-\frac{\tr(X)}{n}\cdot\id) = \exp(X) \cdot e^{-\frac{\tr(X)}{n}}\cdot\id\,,
\]
showing $ii)$. For iv), note that
\[
	\tr(\log P) = \ln(\det P) \quad\Longrightarrow\quad \log P - \tfrac{\ln(\det P)}{n}\cdot\id = \dev_n \log P\,,
\]
and
\begin{align*}
	\exp(\dev_n \log P)	&= e^{-\frac{\tr(\log P)}{n}}\cdot \exp(\log P)\\
	&= \left(e^{\ln(\det P)}\right)^{-\afrac1n}\cdot P = (\det P)^{-\afrac1n}\cdot P\,.
\end{align*}
according to ii). Then the injectivity of the matrix exponential on $\Symn$ shows iv).
\end{proof}
\subsection{A short biography of Heinrich Hencky}
\begin{wrapfigure}{r}{0.224\textwidth}
	\vspace{6em}
		\vspace*{-.7em}
		\hspace*{-2.31em}
		\includegraphics[width=3.85cm]{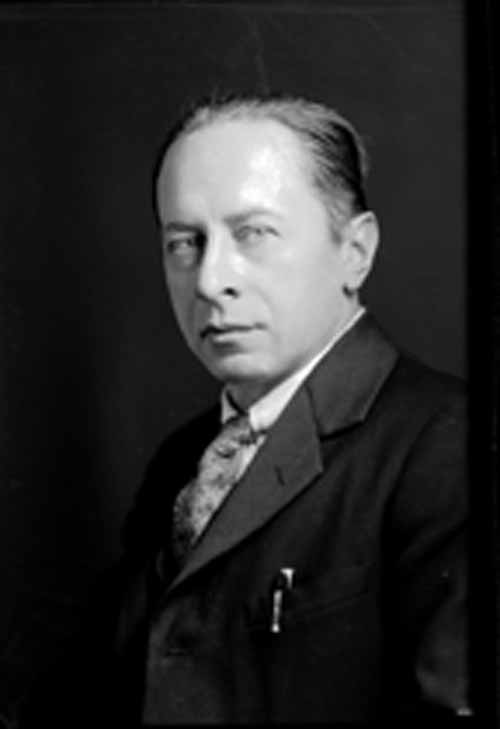}
		\vspace*{-.49em}
		\captionsetup{labelformat=empty,justification=centering}
		\caption{\hspace*{-3em}Hencky at MIT, age 45 \cite{mitchell2009hencky}}
	\vspace{-42em}
\end{wrapfigure}
Biographical information on Heinrich Hencky, as laid out in \cite{tanner2003heinrich,bruhns2014history,henckyObituary}:
\begin{itemize}
	\item November 2, 1885: Hencky is born in Ansbach, Franken, Germany
	\item 1904: Hencky finishes high school in Speyer
	\item 1904--1908: Technische Hochschule M\"unchen
	\item 1909: Military service with the 3rd Pioneer Battalion in M\"unchen
	\item 1912--1913: Ph.D studies at Technische Hochschule Darmstadt
	\item 1910--1912: Work on the Alsatian Railways
	\item 1913--1915: Work for a railway company in Kharkov, Ukraine
	\item 1915--1918: Internment in Kharkov, Ukraine
	\item 1919--1920: Habilitation at Technische Hochschule Darmstadt
	\item 1920--1921: Technische Hochschule Dresden
	\item 1922--1929: Technical University of Delft
	\item 1930--1932: Massachusetts Institute of Technology (MIT)
	\item 1933--1936: Potato farming in New Hampshire %
	\item 1936--1938: Academic work in the Soviet Union, first at Kharkov Chemical Technical Institute, then at the Mechanics Institute of Moscow University
	\item 1938--1950: MAN (Maschinenfabrik Augsburg-N\"urnberg) in Mainz
	\item July 6, 1951: Hencky dies in an avalanche at age 65 during mountain climbing\\
\end{itemize}
Hencky received his diploma in civil engineering from TH M\"unchen in 1908 and his Ph.D from TH Darmstadt in 1913. The title of his thesis was \enquote{\"Uber den Spannungszustand in rechteckigen, ebenen Platten bei gleichm\"a{\ss}ig verteilter und bei konzentrierter Belastung} (\enquote{On the stress state in rectangular flat plates under uniformly distributed and concentrated loading}). In 1915, the main results of his thesis were also published in the Zeitschrift f{\"u}r angewandte Mathematik und Physik \cite{henky1915spannungszustand}.

After working on plasticity theory and small-deformation elasticity, he began his work on finite elastic deformations in 1928. In 1929 he introduced the logarithmic strain $\scaleelog=\log\big(\frac{\text{final length}}{\text{original length}}\big)$ in a tensorial setting \cite{Hencky1928} and applied it to the description of the elastic behavior of vulcanized rubber \cite{hencky1933elastic}.

Today, Hencky is mostly known for  his contributions to plasticity theory: the article \enquote{\"{U}ber einige statisch bestimmte F{\"a}lle des Gleichgewichts in plastischen K\"{o}rpern} \cite{hencky1923gleichgewicht} (\enquote{On statically determined cases of equilibrium in plastic bodies}), published in 1923, is considered his most famous work \cite{tanner2003heinrich}.
\end{appendix}
\end{document}

%% file: tikz/tikzGlobal.tex
\usetikzlibrary{arrows}
\usetikzlibrary{decorations.markings}
\usepackage{pgfplots}
\usetikzlibrary{calc}
\definecolor{ballColor}{rgb}{.924,.987,.924}
\newcommand{\GLBallOpacity}{.7}
\newcommand{\colorGLprime}{white}
\definecolor{planeColorLeft}{rgb}{.924,.987,.924}
\definecolor{planeColorRight}{rgb}{.735,.84,.735}
\definecolor{tangentPlaneColor}{rgb}{.924,.987,.924}
\definecolor{nametagColorGL}{rgb}{.14,.21,.14}
\newcommand{\colorFlatso}{orange}
\newcommand{\colorFlatskew}{red}
\newcommand{\colorFlatgradu}{black}
\newcommand{\colorFlatsym}{violet}
\newcommand{\colorFlatSO}{red}
\newcommand{\colorFlatSOconvexity}{red}
\newcommand{\colorConvexitySOgeodesic}{Brown}
\newcommand{\colorConvexityGLgeodesic}{blue}
\newcommand{\colorFlatId}{black}
\newcommand{\colorFlatPolar}{red}
\newcommand{\colorFlatF}{black}
\newcommand{\colorFlatU}{blue}
\newcommand{\colorNametagGL}{nametagColorGL}
\newcommand{\colorRiemannSO}{red}
\newcommand{\colorRiemannId}{black}
\newcommand{\colorRiemannPolar}{red}

\newcommand{\colorRiemannF}{blue}
\newcommand{\colorRiemannso}{orange}
\newcommand{\colorRiemanngradu}{violet}
\newcommand{\colorRiemannsym}{violet}
\newcommand{\colorRiemannX}{red}
\newcommand{\colorRiemannY}{blue}
\newcommand{\colorRiemannTangent}{white}
\newcommand{\colorEuler}{red}
\newcommand{\colorLagrange}{blue}
\newcommand{\colorUndeformedLeft}{blue!35!white}
\newcommand{\colorUndeformedRight}{blue}
\newcommand{\colorBdeformedLeft}{red!70!white}
\newcommand{\colorBdeformedRight}{red}
\input{tikz/ball_lib.tex}
\input{tikz/deformationLibrary.tex}
\newcommand{\deformationDefaultOpacity}{.42}
\newcommand{\deformationColorOpacity}{.49}
\OmegaSetDefaults
\input{tikz/tikzFormulae.tex}
\tikzset{graphStyle/.style={smooth, color=black, very thick}}
\tikzset{markStyle/.style={only marks, color=black, mark=o, mark size=1.05mm}}
\pgfplotsset{legendStyle/.style={at={(1.1,.55)}, anchor=south west, nodes=right}}
\pgfplotsset{axisStyle/.style={axis x line=middle, axis y line=middle, ylabel style={above right}, xlabel style={right}, legend style={legendStyle}}}

%% file: tikz/ball_lib.tex
\makeatletter
\pgfdeclareradialshading[tikz@ball]{ball}{\pgfqpoint{-10bp}{10bp}}{%
 color(0bp)=(tikz@ball!10!white);
 color(9bp)=(tikz@ball!50!white);
 color(18bp)=(tikz@ball!77!black);
 color(25bp)=(tikz@ball!60.2!black);
 color(50bp)=(black)}
\makeatother

\newcommand\pgfmathsinandcos[3]{
	\pgfmathsetmacro#1{sin(#3)}
	\pgfmathsetmacro#2{cos(#3)}
}
\newcommand\LongitudePlane[3][current plane]{
  \pgfmathsinandcos\sinEl\cosEl{#2} 
  \pgfmathsinandcos\sint\cost{#3}
  \tikzset{#1/.estyle={cm={\cost,\sint*\sinEl,0,\cosEl,(0,0)}}}
}
\newcommand\LatitudePlane[3][current plane]{
	\pgfmathsinandcos\sinEl\cosEl{#2}
	\pgfmathsinandcos\sint\cost{#3}
	\pgfmathsetmacro\yshift{\cosEl*\sint}
	\tikzset{#1/.estyle={cm={\cost,0,0,\cost*\sinEl,(0,\yshift)}}} 
}
\newcommand\DrawLongitudeCircle[2][1]{
	\LongitudePlane{\angEl}{#2}
	\tikzset{current plane/.prefix style={scale=#1}}
	\pgfmathsetmacro\angVis{atan(sin(#2)*cos(\angEl)/sin(\angEl))} 
	\draw[current plane, color=black, opacity=0.2] (\angVis:1) arc (\angVis:\angVis+180:1);
}
\newcommand\DrawLatitudeCircle[2][1]{
	\LatitudePlane{\angEl}{#2}
	\tikzset{current plane/.prefix style={scale=#1}}
	\pgfmathsetmacro\sinVis{sin(#2)/cos(#2)*sin(\angEl)/cos(\angEl)}
	\pgfmathsetmacro\angVis{asin(min(1,max(\sinVis,-1)))}
	\draw[current plane, color=black, opacity=0.2] (\angVis:1) arc (\angVis:-\angVis-180:1);
}
\newcommand\DrawSO[2][1]{
	\LongitudePlane{\angEl}{#2}
	\tikzset{current plane/.prefix style={scale=#1}}
	\pgfmathsetmacro\angVis{atan(sin(#2)*cos(\angEl)/sin(\angEl))} 
	\draw[current plane,color=\colorRiemannSO,opacity=1, thick] (\angVis:1) arc (\angVis:\angVis+180:1);
}
\newcommand\DrawGeodesic[4][1]{
	\LatitudePlane{\angEl}{#2}
	\tikzset{current plane/.prefix style={scale=#1}}
	\pgfmathsetmacro\sinVis{sin(#2)/cos(#2)*sin(\angEl)/cos(\angEl)}
	\pgfmathsetmacro\angStart{#3}
	\pgfmathsetmacro\angEnd{#4}
	\draw[current plane,color=\colorRiemannF,opacity=1] (\angStart:1) arc (\angStart:-\angEnd:1);
}

%% file: tikz/deformationLibrary.tex
	\newcommand\OmegaSetDefaults{
		\OmegaSetGridSize{.1}{.1}
		\OmegaSetOutlineSampleCount{25}
		\OmegaSetGridSampleCount{5}
		\OmegaSetOutlineToUnitCircle
		\OmegaSetDeformationToId
		\OmegaSetOutlineStyle{thick, color=black}
		\OmegaSetShadingStyle{left color = lightgray, right color = black, opacity = \deformationDefaultOpacity}
		\OmegaSetGridStyle{help lines}
		\OmegaSetSmoothOutline
	}
	\newcommand{\OmegaSetOutlineStyle}[1]{
		\tikzset{outlinestyle/.style={#1}}
	}
	\newcommand{\OmegaSetShadingStyle}[1]{
		\tikzset{shadingstyle/.style={#1}}
	}
	\newcommand{\OmegaSetGridStyle}[1]{
		\tikzset{gridstyle/.style={#1}}
	}
	\newcommand{\OmegaSetSharpOutline}{
		\tikzset{outlineplotstyle/.style={sharp plot}}
	}
	\newcommand{\OmegaSetSmoothOutline}{
		\tikzset{outlineplotstyle/.style={smooth}}
	}
	\newcommand{\OmegaSetOutline}[3]{
		\pgfmathdeclarefunction*{omegagammax}{1}{\pgfmathparse{#1(##1)}}
		\pgfmathdeclarefunction*{omegagammay}{1}{\pgfmathparse{#2(##1)}}
		\pgfmathsetmacro{\OmegaParameterlength}{#3}
	}
	\newcommand{\OmegaSetDeformation}[2]{
		\pgfmathdeclarefunction*{omegaphix}{2}{\pgfmathparse{#1(##1, ##2)}}
		\pgfmathdeclarefunction*{omegaphiy}{2}{\pgfmathparse{#2(##1, ##2)}}
	}
	\newcommand{\OmegaSetGridSize}[2]{
		\pgfmathsetmacro{\OmegaGridx}{#1}
		\pgfmathsetmacro{\OmegaGridy}{#2}
	}
	\newcommand{\OmegaSetOutlineSampleCount}[1]{
		\pgfmathsetmacro{\OmegaOutlineSampleCount}{#1}
	}
	\newcommand{\OmegaSetGridSampleCount}[1]{
		\pgfmathsetmacro{\OmegaGridSampleCount}{#1}
	}
	\newcommand{\OmegaSetOutlineToCircle}[3]{
		\pgfmathdeclarefunction*{circx}{1}{\pgfmathparse{#1 + #3*cos(##1)}}
		\pgfmathdeclarefunction*{circy}{1}{\pgfmathparse{#2 + #3*sin(##1)}}
		\OmegaSetOutline{circx}{circy}{360}
	}
	\newcommand{\OmegaSetOutlineToUnitCircle}{
		\OmegaSetOutlineToCircle{0}{0}{1}
	}
	\newcommand{\OmegaSetOutlineToEllipse}[5]{
		\pgfmathdeclarefunction*{ellx}{1}{\pgfmathparse{#1 + #3*cos(##1)*cos(#5) - #4*sin(##1)*sin(#5)}}
		\pgfmathdeclarefunction*{elly}{1}{\pgfmathparse{#2 + #3*cos(##1)*sin(#5) + #4*sin(##1)*cos(#5)}}
		\OmegaSetOutline{ellx}{elly}{360}
	}
	\newcommand{\OmegaSetOutlineToRectangle}[4]{
		\pgfmathsetmacro{\rectWidth}{#3-#1}
		\pgfmathsetmacro{\rectHeight}{#4-#2}
		\pgfmathdeclarefunction*{rectx}{1}{\pgfmathparse{%
			##1<=\rectWidth ? (#1+##1) : (##1<=\rectWidth+\rectHeight ? (#1+\rectWidth) : (##1<=(2*\rectWidth)+\rectHeight ? (#1+\rectWidth-(##1-(\rectWidth+\rectHeight))) : (#1))}%
		}
		\pgfmathdeclarefunction*{recty}{1}{\pgfmathparse{%
			##1<=\rectWidth ? (#2) : (##1<=\rectWidth+\rectHeight ? (#2+##1-\rectWidth) : (##1<=(2*\rectWidth)+\rectHeight ? (#2+\rectHeight) : (#2+\rectHeight-(##1-(2*\rectWidth+\rectHeight)))))}%
		}
		\OmegaSetOutline{rectx}{recty}{2*\rectWidth+2*\rectHeight}
	}
	\newcommand{\OmegaSetDeformationToId}{
		\pgfmathdeclarefunction*{idx}{2}{\pgfmathparse{##1}}
		\pgfmathdeclarefunction*{idy}{2}{\pgfmathparse{##2}}
		\OmegaSetDeformation{idx}{idy}
	}
	
	\newcommand{\OmegaDraw}{
		\OmegaDrawCustomDeformation{omegaphix}{omegaphiy}
	}
	\newcommand{\OmegaDrawCustomDeformation}[2]{
		\OmegaDrawFull{\OmegaGridx}{\OmegaGridy}{\OmegaParameterlength}{omegagammax}{omegagammay}{#1}{#2}{\OmegaOutlineSampleCount}{\OmegaGridSampleCount}
	}
	\newcommand{\OmegaDrawFull}[9]{	%
		\begin{scope}
			\def\conversionFromBoundingBoxUnit{0.03514598035}
			\tikzset{borderdomainstyle/.style={domain=0:#3, variable=\t, samples=#8}}
			\tikzset{griddomainstyle/.style={smooth, variable=\t, samples=#9}}
			\newdimen\leftBoundWrongUnit
			\newdimen\rightBoundWrongUnit
			\newdimen\lowerBoundWrongUnit
			\newdimen\upperBoundWrongUnit

			\shade[shadingstyle]
				plot[borderdomainstyle, outlineplotstyle] ({#6(#4(\t), #5(\t))}, {#7(#4(\t), #5(\t))});

			\begin{scope}
				\coordinate (oldLowerLeftBound) at (current bounding box.south west);
				\coordinate (oldUpperRightBound) at (current bounding box.north east);

				\pgfresetboundingbox
				\path[use as bounding box]	%
					plot[borderdomainstyle, outlineplotstyle] ({#4(\t)}, {#5(\t))});
				\coordinate (lowerLeftBound) at (current bounding box.south west);
				\coordinate (upperRightBound) at (current bounding box.north east);
			\end{scope}
			\pgfextractx{\leftBoundWrongUnit}{\pgfpointanchor{lowerLeftBound}{center}}
			\pgfextractx{\rightBoundWrongUnit}{\pgfpointanchor{upperRightBound}{center}}
			\pgfextracty{\lowerBoundWrongUnit}{\pgfpointanchor{lowerLeftBound}{center}}
			\pgfextracty{\upperBoundWrongUnit}{\pgfpointanchor{upperRightBound}{center}}
			\pgfmathsetmacro{\leftBound}{\conversionFromBoundingBoxUnit*\leftBoundWrongUnit}
			\pgfmathsetmacro{\rightBound}{\conversionFromBoundingBoxUnit*\rightBoundWrongUnit}
			\pgfmathsetmacro{\lowerBound}{\conversionFromBoundingBoxUnit*\lowerBoundWrongUnit}
			\pgfmathsetmacro{\upperBound}{\conversionFromBoundingBoxUnit*\upperBoundWrongUnit}

			\begin{scope}
				\clip
					plot[borderdomainstyle, outlineplotstyle] ({#6(#4(\t), #5(\t))}, {#7(#4(\t), #5(\t))});

				\pgfmathsetmacro{\horizontalGridCount}{(\upperBound-\lowerBound)/#2}
				\foreach \b in {1,2,...,{\horizontalGridCount}}{
					\draw[gridstyle]
						plot[griddomainstyle, domain=\leftBound:\rightBound] ({#6(\t,{\lowerBound+(#2*\b)})},{#7(\t,{\lowerBound+(#2*\b})});
				}

				\pgfmathsetmacro{\verticalGridCount}{(\rightBound-\leftBound)/#1}
				\foreach \b in {1,2,...,{\verticalGridCount}}{
					\draw[gridstyle]
						plot[griddomainstyle, domain=\lowerBound:\upperBound] ({#6({\leftBound+(#1*\b)},\t)},{#7({\leftBound+(#1*\b)},\t)});
				}
			\end{scope}

			\draw[outlinestyle]
					plot[borderdomainstyle, outlineplotstyle] ({#6(#4(\t), #5(\t))}, {#7(#4(\t), #5(\t))});

			\path[use as bounding box] (oldLowerLeftBound) rectangle (oldUpperRightBound);
		\end{scope}
	}

%% file: tikz/tikzFormulae.tex
\def\tablepath{./tikz/tables}
\pgfplotstableread{\tablepath/treloar_1944_simple_elongation.txt}{\simpleElongationTable}
\pgfplotstableread{\tablepath/treloar_1944_2-dimensional.txt}{\equibiaxialAndCompressionTable}

\pgfmathdeclarefunction*{lnsquaredsum}{3}{\pgfmathparse{(ln(#1))^2+(ln(#2))^2+(ln(#3))^2}}
\pgfmathdeclarefunction*{shearlam}{1}{\pgfmathparse{.5*(sqrt(#1^2+4)+#1)}}%

\pgfmathdeclarefunction*{Imbert_t_uniaxial}{1}{\pgfmathparse{3*\mucon*ln(#1)}}
\pgfmathdeclarefunction*{Imbert_sigma_uniaxial}{1}{\pgfmathparse{#1*Imbert_t_uniaxial(#1)}}

\def\acon{4}
\def\bcon{2}
\pgfmathdeclarefunction*{Landel_t_uniaxial}{1}{\pgfmathparse{\acon*(1-#1^(-1.5)+\bcon*((1/#1)-(1/#1^2))}}
\pgfmathdeclarefunction*{Landel_sigma_uniaxial}{1}{\pgfmathparse{#1*Landel_t_uniaxial(#1)}}

\def\ogdenmuone{0.000012069}
\def\ogdenmutwo{3.7729}
\def\ogdenmuthree{-0.052171}
\def\ogdenalphaone{8.3952}
\def\ogdenalphatwo{1.8821}
\def\ogdenalphathree{-2.2453}
\pgfmathdeclarefunction*{ogden_t_uniaxial}{1}{\pgfmathparse{\ogdenmuone*(#1^(\ogdenalphaone-1)-#1^(-1-(\ogdenalphaone/2)))+\ogdenmutwo*(#1^(\ogdenalphatwo-1)-#1^(-1-(\ogdenalphatwo/2)))+\ogdenmuthree*(#1^(\ogdenalphathree-1)-#1^(-1-(\ogdenalphathree/2)))}}
\pgfmathdeclarefunction*{ogden_t_equibiaxial}{1}{\pgfmathparse{\ogdenmuone*(#1^(\ogdenalphaone-1)-#1^(-1-(\ogdenalphaone*2)))+\ogdenmutwo*(#1^(\ogdenalphatwo-1)-#1^(-1-(\ogdenalphatwo*2)))+\ogdenmuthree*(#1^(\ogdenalphathree-1)-#1^(-1-(\ogdenalphathree*2)))}}
\pgfmathdeclarefunction*{ogden_t_pureshear}{1}{\pgfmathparse{ogden_sigmaone_pureshear(#1)/#1}}
\pgfmathdeclarefunction*{ogden_sigma_uniaxial}{1}{\pgfmathparse{#1*ogden_t_uniaxial(#1)}}
\pgfmathdeclarefunction*{ogden_sigma_equibiaxial}{1}{\pgfmathparse{#1*ogden_t_equibiaxial(#1)}}
\pgfmathdeclarefunction*{ogden_sigmaone_pureshear}{1}{\pgfmathparse{#1*(\ogdenmuone*(#1^(\ogdenalphaone-1)-#1^(-1-(\ogdenalphaone)))+\ogdenmutwo*(#1^(\ogdenalphatwo-1)-#1^(-1-(\ogdenalphatwo)))+\ogdenmuthree*(#1^(\ogdenalphathree-1)-#1^(-1-(\ogdenalphathree))))}}
\pgfmathdeclarefunction*{ogden_sigmatwo_pureshear}{1}{\pgfmathparse{ogden_sigmaone_pureshear(#1)-ogden_sigmadifference_pureshear(#1)}}
\pgfmathdeclarefunction*{ogden_sigma_simpleshear}{1}{\pgfmathparse{ogden_sigmaone_pureshear(shearlam(#1))/(shearlam(#1)+(1/shearlam(#1)))}}

\pgfmathdeclarefunction*{Neohooke_t_uniaxial}{1}{\pgfmathparse{\mucon*(#1-#1^(-2))}}
\pgfmathdeclarefunction*{Neohooke_t_equibiaxial}{1}{\pgfmathparse{\mucon*(#1-#1^(-5))}}
\pgfmathdeclarefunction*{Neohooke_t_pureshear}{1}{\pgfmathparse{Neohooke_sigmaone_pureshear(#1)/#1}}
\pgfmathdeclarefunction*{Neohooke_sigma_uniaxial}{1}{\pgfmathparse{#1*Neohooke_t_uniaxial(#1)}}
\pgfmathdeclarefunction*{Neohooke_sigma_equibiaxial}{1}{\pgfmathparse{#1*Neohooke_t_equibiaxial(#1)}}
\pgfmathdeclarefunction*{Neohooke_sigmaone_pureshear}{1}{\pgfmathparse{#1*(\mucon*(#1-#1^(-3))}}
\pgfmathdeclarefunction*{Neohooke_sigmatwo_pureshear}{1}{\pgfmathparse{Neohooke_sigmaone_pureshear(#1)-Neohooke_sigmadifference_pureshear(#1)}}
\pgfmathdeclarefunction*{Neohooke_sigma_simpleshear}{1}{\pgfmathparse{Neohooke_sigmaone_pureshear(shearlam(#1))/(shearlam(#1)+(1/shearlam(#1)))}}

\pgfmathdeclarefunction*{SVK_t_uniaxial}{1}{\pgfmathparse{(.5*\lambdacon+\mucon)*(#1^3)+(1.5*\lambdacon+\mucon)*(#1^(-2)-#1)-(\lambdacon+\mucon)*(#1^(-3))+.5*\lambdacon}}
\pgfmathdeclarefunction*{SVK_sigma_uniaxial}{1}{\pgfmathparse{#1*SVK_t_uniaxial(#1)}}

\pgfmathdeclarefunction*{expHencky_energy}{3}{\pgfmathparse{ (\mucon/(1+\gammacon))*((1/\betaonecon)*(exp(\betaonecon*lnsquaredsum(#1,#2,#3))-1)+(\gammacon/\betatwocon)*(exp(\betatwocon*lnsquaredsum(#1,#2,#3))-1))}}
\pgfmathdeclarefunction*{expHencky_energy_uniaxial}{1}{\pgfmathparse{expHencky_energy(#1,(1/(sqrt(#1))),(1/(sqrt(#1))))}}
\pgfmathdeclarefunction*{expHencky_energy_equibiaxial}{1}{\pgfmathparse{expHencky_energy(#1,#1,(1/(#1^2)))}}
\pgfmathdeclarefunction*{expHencky_t_uniaxial}{1}{\pgfmathparse{(\mucon*3*ln(#1)/((1+\gammacon)*#1))*(exp(1.5*\betaonecon*(ln(#1))^2)+\gammacon*exp(1.5*\betatwocon*(ln(#1))^2))}}
\pgfmathdeclarefunction*{expHencky_t_equibiaxial}{1}{\pgfmathparse{(\mucon*6*ln(#1)/((1+\gammacon)*#1))*(exp(6*\betaonecon*(ln(#1))^2)+\gammacon*exp(6*\betatwocon*(ln(#1))^2))}}
\pgfmathdeclarefunction*{expHencky_t_pureshear}{1}{\pgfmathparse{expHencky_sigmaone_pureshear(#1)/#1}}
\pgfmathdeclarefunction*{expHencky_sigma_uniaxial}{1}{\pgfmathparse{#1*expHencky_t_equibiaxial(#1)}}
\pgfmathdeclarefunction*{expHencky_sigma_equibiaxial}{1}{\pgfmathparse{#1*expHencky_t_equibiaxial(#1)}}
\pgfmathdeclarefunction*{expHencky_sigmaone_pureshear}{1}{\pgfmathparse{(4*\mucon/(1+\gammacon))*ln(#1)*(exp(2*\betaonecon*(ln(#1))^2)+\gammacon*exp(2*\betatwocon*(ln(#1))^2))}}
\pgfmathdeclarefunction*{expHencky_sigmatwo_pureshear}{1}{\pgfmathparse{expHencky_sigmaone_pureshear(#1)-expHencky_sigmadifference_pureshear(#1)}}
\pgfmathdeclarefunction*{expHencky_sigmadifference_pureshear}{1}{\pgfmathparse{(2*\mucon/(1+\gammacon))*ln(#1)*(exp(2*\betaonecon*(ln(#1))^2)+\gammacon*exp(2*\betatwocon*(ln(#1))^2))}}	%
\pgfmathdeclarefunction*{expHencky_sigma_simpleshear}{1}{\pgfmathparse{expHencky_sigmaone_pureshear(shearlam(#1))/(shearlam(#1)+(1/shearlam(#1)))}}

\pgfmathdeclarefunction*{Neohooke_t_uniaxial}{1}{\pgfmathparse{\mucon*(#1-#1^(-2))}}
\pgfmathdeclarefunction*{Neohooke_t_equibiaxial}{1}{\pgfmathparse{\mucon*(#1-#1^(-5))}}
\pgfmathdeclarefunction*{Neohooke_t_pureshear}{1}{\pgfmathparse{Neohooke_sigmaone_pureshear(#1)/#1}}
\pgfmathdeclarefunction*{Neohooke_sigma_uniaxial}{1}{\pgfmathparse{#1*Neohooke_t_uniaxial(#1)}}
\pgfmathdeclarefunction*{Neohooke_sigma_equibiaxial}{1}{\pgfmathparse{#1*Neohooke_t_equibiaxial(#1)}}
\pgfmathdeclarefunction*{Neohooke_sigmaone_pureshear}{1}{\pgfmathparse{#1*(\mucon*(#1-#1^(-3))}}
\pgfmathdeclarefunction*{Neohooke_sigmatwo_pureshear}{1}{\pgfmathparse{Neohooke_sigmaone_pureshear(#1)-Neohooke_sigmadifference_pureshear(#1)}}
\pgfmathdeclarefunction*{Neohooke_sigma_simpleshear}{1}{\pgfmathparse{Neohooke_sigmaone_pureshear(shearlam(#1))/(shearlam(#1)+(1/shearlam(#1)))}}

\pgfmathdeclarefunction*{Hencky_energy}{3}{\pgfmathparse{\mucon*lnsquaredsum(#1,#2,#3)}}
\pgfmathdeclarefunction*{Hencky_energy_uniaxial}{1}{\pgfmathparse{Hencky_energy(#1,(1/(sqrt(#1))),(1/(sqrt(#1))))}}
\pgfmathdeclarefunction*{Hencky_energy_equibiaxial}{1}{\pgfmathparse{Hencky_energy(#1,#1,(1/(#1^2)))}}
\pgfmathdeclarefunction*{Hencky_t_uniaxial}{1}{\pgfmathparse{3*\mucon*ln(#1)/#1}}
\pgfmathdeclarefunction*{Hencky_t_equibiaxial}{1}{\pgfmathparse{6*\mucon*ln(#1)/#1}}
\pgfmathdeclarefunction*{Hencky_t_pureshear}{1}{\pgfmathparse{Hencky_sigmaone_pureshear(#1)/#1}}
\pgfmathdeclarefunction*{Hencky_sigma_uniaxial}{1}{\pgfmathparse{3*\mucon*ln(#1)}}
\pgfmathdeclarefunction*{Hencky_sigma_equibiaxial}{1}{\pgfmathparse{6*\mucon*ln(#1)}}
\pgfmathdeclarefunction*{Hencky_sigmaone_pureshear}{1}{\pgfmathparse{4*\mucon*ln(#1)}}
\pgfmathdeclarefunction*{Hencky_sigmatwo_pureshear}{1}{\pgfmathparse{Hencky_sigmaone_pureshear(#1)-Hencky_sigmadifference_pureshear(#1)}}
\pgfmathdeclarefunction*{Hencky_sigmadifference_pureshear}{1}{\pgfmathparse{2*\mucon*ln(#1)}}	%
\pgfmathdeclarefunction*{Hencky_sigma_simpleshear}{1}{\pgfmathparse{Hencky_sigmaone_pureshear(shearlam(#1))/(shearlam(#1)+(1/shearlam(#1)))}}

\newcommand{\setParametersTreloar}{
	\pgfmathsetmacro\mucon{9.81*.435}
	\pgfmathsetmacro\betaonecon{.334}
	\pgfmathsetmacro\betatwocon{1.22}
	\pgfmathsetmacro\gammacon{.005}
	\pgfmathsetmacro\betasinhonecon{.63}
	\pgfmathsetmacro\betasinhtwocon{.52}
	\pgfmathsetmacro\gammasinhcon{.000}
	\pgfmathsetmacro\mrmuone{\mucon/2}
	\pgfmathsetmacro\mrmutwo{0}
}

\setParametersTreloar
\pgfmathsetmacro\lambdacon{0.5}

%% file: tikz/strainPlot.tex
\draw[->,thick] (-0.5,0) -- (4.1,0) node[below=3mm, left=-2mm] {\large$\lambda$};
\draw[->,thick] (0,-1.65) -- (0,1.75) node[above] {\Large$\scalee$};
\draw[thick,orange,smooth,samples=100,domain=0.05:2.24] plot(\x,{(1/2)*(\x*\x -1)});
\draw (2,1.6) node[orange,right] {$\scalee_1$};
\draw[thick,blue,smooth,samples=100,domain=0.075:3] plot(\x,{\x-1});
\draw (3,1.96) node[blue,right] {$\scalee_{\afrac12}$};
\draw[thick,green,smooth,samples=100,domain=0.483:3.9] plot(\x,{(1/2)*(1- 1/(\x*\x))});
\draw (2.2,0.175) node[green,right] {$\scalee_{-1}$}; 
\draw[thick,red,smooth,samples=100,domain=0.2:3.9] plot(\x,{ln(\x)});
\draw (2.2,0.68) node[red,right] {$\scalee_{0}$};
\draw[thick,violet,smooth,samples=100,domain=0.28:3.5] plot(\x,{.5*(\x-(1/\x))});
\draw (3.5,1.54) node[violet,right] {$\scaleetilde_{\afrac12}$};

%% file: tikz/euclideanDistanceToSkew.tex
\begin{scope}[yslant=0.5,xslant=-1,xshift=100]
	\fill[left color=planeColorLeft,right color=planeColorRight] (0,1) rectangle (5,6);
	\draw[color=black, opacity=.2,step=.5] (-.2,.8) grid (5.2,6.2);
	\draw[color=\colorFlatso, very thick] (-.2,3) -- (5.2,3) node[pos=1, above right=-4.2] {\large $\son$};
	\node (zero) at (3,3) {};
	\node (gradu) at (1,5) {};
	\node (skewgradu) at (1,3) {};
\end{scope}
\begin{scope}
	\node[below right,yslant=0.5,xslant=-1,color=\colorNametagGL] at (.7,6.3) {{\huge\boldmath $\Rnn$}};
	\fill[color=\colorFlatgradu] (gradu) circle (.075);
	\fill[color=\colorFlatso] (zero) circle (.075);
	\node[below right, color=\colorFlatso] at (zero) {$0$};
	\draw[color=\colorFlatsym, <-, very thick] (gradu) -- (skewgradu) node[below left=-3.9, pos=.42, align=center] {$\varepsilon=\sym\grad u$};
	\draw[color=\colorFlatskew,->, very thick] (zero) -- (skewgradu) node[below right, pos=.75] {$\skew \grad u$};
	\draw[color=\colorFlatgradu,->, very thick] (zero) -- (gradu) node[above right, pos=.84] {$\grad u$};
	\fill[color=\colorFlatskew] (skewgradu) circle (.075);
\end{scope}

%% file: tikz/flatGLn.tex
\begin{scope}[yslant=0.5,xslant=-1,xshift=100]
	\fill[left color=planeColorLeft,right color=planeColorRight] (0,0) rectangle (5,6);
	\draw[color=black, opacity=.2,step=.5] (-.2,-.2) grid (5.2,6.2);
	\draw[color=\colorFlatSO, very thick] (3,3) circle (1);
	\foreach \x in {0,.01,...,.2}
	\fill[color=white, very thick, opacity={1-5*\x}] (3,3) circle (\x);
	\node (F) at ({3+2.75*sin(-35)},{3+2.75*cos(-35)}) {};
	\node (U) at ({3+2.75*sin(-105)},{3+2.75*cos(-105)}) {};
	\node (R) at ({3+sin(-35)},{3+cos(-35)}) {};
	\node (Q) at ({3+sin(20)},{3+cos(20)}) {};
	\node (id) at ({3+sin(-105)},{3+cos(-105)}) {};
	\draw [\colorFlatF,thick,domain=-105:-35,dotted,opacity=.7] plot ({3+2.75*sin(\x)}, {3+2.75*cos(\x)});
\end{scope}
\begin{scope}
	\node[below right,yslant=0.5,xslant=-1,color=\colorNametagGL] at (-.4,5.8) {{\huge\boldmath $\GLpn$}};
	\node[color=\colorFlatSO,below] (SOn) at (4.9,4.2) {\large $\SOn$};
	\fill[color=\colorFlatF] (F) circle (.075);
	\fill[color=\colorFlatU, opacity=1] (U) circle (.075);
	\fill[color=\colorFlatPolar] (R) circle (.075);
	\fill[color=\colorFlatId] (id) circle (.075);
	\draw[color=\colorFlatF, thick,<-] (F) -- (R) node[pos=.4] (formulaEndR) {} node[pos=.5,below] {$F-R$};
	\draw[color=\colorFlatU, thick,<-,opacity=1] (U) -- (id) node[pos=.4] (formulaEndR) {} node[pos=.5,below right] {$U-\id \,=\, R^TF-\id$};
	\node[color=\colorFlatPolar, right] at (R) {$R$};
	\node[color=\colorFlatId, below right] at (id) {$\id$};
	\node[color=\colorFlatF, above left=-3] at (F) {$F=R\,U$};
	\node[color=\colorFlatU, below right,opacity=1] at (U) {$U$};
\end{scope}

%% file: tikz/euclideanDistanceToSOn.tex
\begin{scope}
	\def\pointNodeSize{4}
	\tikzset{pointNodeStyle/.style={inner sep = 0, minimum size=\pointNodeSize, draw, circle, very thin, shade, shading=ball}}
	\tikzset{textNodeStyle/.style={font = {}}}
	\coordinate (v) at (7,-6);
	\coordinate (w) at (3,1);
	\def\lambda{.5} 
	\def\nu{1.5}
	\coordinate (unity) at (0,1.785);
	\def\R{2.5}
	\def\angEl{44.5}

	\filldraw[color=\colorGLprime,ball color=ballColor,opacity=\GLBallOpacity] (0,0) circle (\R);
	\foreach \t in {-80,-70,...,80} { \DrawLatitudeCircle[\R]{\t}}
	\foreach \t in {-5,-25,...,-175} { \DrawLongitudeCircle[\R]{\t+10}}

	\DrawSO[\R]{25}

	\coordinate (F) at (1.4,-1.1){};
	\coordinate (R) at (-2.23,-.42){};
	\node[pointNodeStyle, ball color = \colorRiemannF] (FNode) at (1.4,-1.1){};
	\node[pointNodeStyle, ball color = \colorRiemannSO] (RNode) at (-2.23,-.42){};

	\node[pointNodeStyle, ball color = \colorRiemannId, minimum size = {1.5*\pointNodeSize}] (unityNode) at (unity){};

	\node[textNodeStyle, color=\colorRiemannSO] (SOText) at (-.5,1){$\mathrm{SO}(n)$};
	\node[textNodeStyle, above left, color=\colorRiemannId] (unityText) at (unityNode){$\id$};
	\node[textNodeStyle, color=\colorRiemannSO, right] (RText) at ($(R)+(0,.2)$){$R=\mathrm{polar}(F)$};

	\def\deltaT{0.05}
	\def\minOpacity{.2}
	\def\opacityExponent{2}
	\def\maxOpacity{.7}
	\pgfmathsetmacro\opacityFactor{(\maxOpacity-\minOpacity)/pow(.5,\opacityExponent)}
	\foreach \t in {0,\deltaT,...,1}
		\draw[dashed, color=black, opacity={\minOpacity+\opacityFactor*pow(abs((\t+\deltaT/2)-.5),\opacityExponent)}] (${1-\t}*(F)+\t*(R)$) -- (${1-(\t+\deltaT)}*(F)+{\t+\deltaT}*(R)$);

	\node[textNodeStyle, color=\colorNametagGL] (GLText) at (-3.3,.35){$\mathrm{GL^+}(n)$};
	\node[textNodeStyle, right, color=\colorRiemannF] (FText) at (F){$F$};

	\tikzset{arrowstyle/.style={<-, dotted, thick}}
	\tikzset{formulastyle/.style={align=left, font={\footnotesize}}}
	
	\draw[arrowstyle, color=darkgray] (.1,-.7) .. controls (1,-.2) and (3,-.2) .. (3.5,-.2) node[formulastyle, color=darkgray, right]{\\[4mm]$\mathrm{dist}_{\mathrm{euclid}}^2(F, \mathrm{SO}(n))$\\[1mm] $\hspace*{4mm}=\norm{U-\id}^2 = \norm{\sqrt{F^T\!F}-\id}^2$};
\end{scope}

%% file: tikz/euclideanAndRiemannianDistanceOnGLn.tex
\begin{scope}[]
	\def\pointNodeSize{4}
	\tikzset{pointNodeStyle/.style={inner sep = 0, minimum size=\pointNodeSize, draw, circle, very thin, shade, shading=ball}}
	\tikzset{textNodeStyle/.style={font = {}}}
	\coordinate (v) at (7,-6);
	\coordinate (w) at (3,1);
	\def\lambda{.5} 
	\def\nu{1.5}
	\coordinate (unity) at (0,1.785);
	\def\R{2.5}
	\def\angEl{44.5}

	\path[use as bounding box] (-4.55,-1.12) rectangle (7.14,.77);

	\filldraw[color=\colorGLprime,ball color=ballColor,opacity=\GLBallOpacity] (0,0) circle (\R);
	\foreach \t in {-80,-70,...,80} { \DrawLatitudeCircle[\R]{\t}}
	\foreach \t in {-5,-25,...,-175} { \DrawLongitudeCircle[\R]{\t+10}}
	\coordinate (F) at (1.75,0){};
	\coordinate (P) at (-2,.28){};
	\node[pointNodeStyle, ball color = \colorRiemannF] (FNode) at (F){};
	\node[pointNodeStyle, ball color = \colorRiemannF] (PNode) at (P){};
	\DrawGeodesic[\R]{30}{-35}{156}
	
	\node[textNodeStyle, color=\colorRiemannF, above right] (RText) at (P){$A$};

	\def\deltaT{0.05}
	\def\minOpacity{.2}
	\def\opacityExponent{2}
	\def\maxOpacity{.7}
	\pgfmathsetmacro\opacityFactor{(\maxOpacity-\minOpacity)/pow(.5,\opacityExponent)}
	\foreach \t in {0,\deltaT,...,1}
		\draw[dashed, color=black, opacity={\minOpacity+\opacityFactor*pow(abs((\t+\deltaT/2)-.5),\opacityExponent)}] (${1-\t}*(F)+\t*(P)$) -- (${1-(\t+\deltaT)}*(F)+{\t+\deltaT}*(P)$);

	\node[textNodeStyle, color=\colorNametagGL] (GLText) at (-3.3,.35){$\mathrm{GL^+}(n)$};
	\node[textNodeStyle, above, color=\colorRiemannF] (FText) at (F){$B$};

	\tikzset{arrowstyle/.style={<-, dotted, thick}}
	\tikzset{formulastyle/.style={align=left, font={\footnotesize}}}
	
	\draw[arrowstyle, color=darkgray] (0,.4) .. controls (1,.7) and (2,.7) .. (3,.5) node[formulastyle, color=darkgray, right]{$\mathrm{dist}_{\mathrm{euclid}}^2(A,B)=\norm{A-B}^2$};
	\draw[arrowstyle, color=\colorRiemannF] (1.3,-.5) .. controls (1.5,-.6) and (2.5,-.6) .. (3,-.5) node[formulastyle, color=\colorRiemannF, right]{$\mathrm{dist}_{\mathrm{geod}}^2(A, B)$};
\end{scope}

%% file: tikz/flatGLnEuclideanGeodesics.tex
\begin{scope}[yslant=0.5,xslant=-1,xshift=100]
	\fill[left color=planeColorLeft,right color=planeColorRight] (1,1) rectangle (5,5);
	\draw[color=black, opacity=.2,step=.5] (.8,.8) grid (5.2,5.2);
	\foreach \x in {0,.01,...,.2}
	\fill[color=white, very thick, opacity={1-5*\x}] (3,3) circle (2*\x);
	\node (A) at ({3+2*sin(-40)},{3+2*cos(-40)}) {};
	\node (B) at ({3+2*sin(-220)},{3+2*cos(-220)}) {};
	\node (C) at ({3+1.785*sin(-135)},{3+1.785*cos(-135)}) {};
	\node (center) at (3,3) {};
\end{scope}
\begin{scope}
	\node[below right,yslant=0.5,xslant=-1,color=\colorNametagGL] at (1.4,5.81) {{\Large\boldmath $\GLpn$}};
	\fill[color=black] (A) circle (.075);
	\fill[color=black] (B) circle (.075);
	\fill[color=black] (C) circle (.075);
	\draw[color=blue,->,opacity=1, very thick] (A) -- (B)
		node[pos=.5,above,sloped] {$\gammahat(t)=A+t(C-A)$}
		node[pos=.5,below,sloped,color=red] {$\gammahat(t_0)\notin\GLpn$};
	\draw[color=black,->,opacity=1, very thick] (A) -- (C)
		node[pos=.7,below] {$\gamma$};
	\fill[color=red] (center) circle (.075);
	\node[color=black, below left] at (A) {$A$};
	\node[color=black, below left] at (B) {$C$};
	\node[color=black, below left] at (C) {$B$};
\end{scope}

%% file: tikz/deformationDistanceExplanation_alternative.tex
\tikzset{curvestyle/.style={smooth, semithick, variable=\t, samples=25, domain=-1:1.75}}
\tikzset{pointNodeStyle/.style={inner sep = 0, minimum size=2, draw, circle, very thin, shade, shading=ball}}
\tikzset{textNodeStyle/.style={font = {}}}
\tikzset{tangentStyle/.style={->, semithick}}

\tikzset{arrowstyle/.style={
	decoration={markings,mark=at position 1 with {\arrow[scale=1.8]{>}}},
	postaction={decorate},
	shorten >=0.4pt
}}
\tikzset{backarrowstyle/.style={
    decoration={markings,mark=at position 0 with {\arrow[scale=1.8]{<}}},
	postaction={decorate},
	shorten >=0.4pt
}}
\tikzset{doublearrowstyle/.style={
    decoration={markings,mark=at position 0 with {\arrow[scale=1.8]{<}},mark=at position 1 with {\arrow[scale=2.1]{>}}},
	postaction={decorate},
	shorten >=0.4pt
}}

\def\rotPhi{14}

\OmegaSetDefaults
\OmegaSetGridSize{.2}{.2}
\OmegaSetOutlineSampleCount{120}
\OmegaSetGridSampleCount{120}

\OmegaSetOutlineStyle{color = darkgray, opacity = .9, thick}
\OmegaSetShadingStyle{left color = lightgray, right color = black, opacity = \deformationDefaultOpacity}
\OmegaSetGridStyle{very thin, color=darkgray, opacity = .3}

\newcommand{\smallSquare}{%
	\OmegaSetGridStyle{opacity=0}
	\OmegaSetGridSampleCount{2}
	\OmegaSetGridStyle{very thin, color=darkgray, opacity = .3}
	\OmegaSetOutlineToRectangle{-.3}{.1}{-.1}{.3}
	\OmegaDraw
}

\begin{scope}[scale=2.24,declare function ={%
		phistraightx(\x,\y)=\x;
		phistraighty(\x,\y)=1.1*\y*(1-.05*(\x)^2)-.1*(2*\x)^2-.075*(3*\x)^3;
		phix(\x,\y)=\x+.4*\y;
		phiy(\x,\y)=\y+.25*(\x+.5)^2;
		phiTwox(\x,\y)=cos(\rotPhi)*phistraightx(\x,\y) - sin(\rotPhi)*phistraighty(\x,\y);
		phiTwoy(\x,\y)=cos(\rotPhi)*phistraighty(\x,\y) + sin(\rotPhi)*phistraightx(\x,\y);
	}]

	\OmegaSetOutlineToRectangle{-.5}{-.5}{.5}{.5}

	\begin{scope}
		\OmegaSetSharpOutline
		\OmegaSetDeformationToId
		\OmegaDraw

		\OmegaSetShadingStyle{left color = black, right color = black, opacity = .15, ->}
		\smallSquare
		
		\fill[color=black] (-.2,.2) circle (.01) node[above right=-1.19mm] {\footnotesize$x$};
	\end{scope}

	\begin{scope}[xshift=-2.1cm, yshift=-2.94cm]
		\OmegaSetSmoothOutline
		\OmegaSetDeformation{phix}{phiy}
		\OmegaDraw

		\OmegaSetShadingStyle{left color = red, right color = red, opacity = \deformationColorOpacity}
		\smallSquare
	\end{scope}

	\begin{scope}[xshift=2.1cm, yshift=-2.8cm]
		\OmegaSetSmoothOutline
		\OmegaSetDeformation{phiTwox}{phiTwoy}
		\OmegaDraw

		\OmegaSetShadingStyle{left color = blue, right color = blue, opacity = \deformationColorOpacity}
		\smallSquare
	\end{scope}

	\begin{scope}[xshift=-1.26cm, yshift=-2.31cm]
		\OmegaSetSharpOutline
		\OmegaSetOutlineSampleCount{5}
		\OmegaSetDeformation{phix}{phiy}
		
	\begin{scope}[scale=4]
		\OmegaSetShadingStyle{left color = red, right color = red, opacity = \deformationColorOpacity}
		\smallSquare
	\end{scope}
	\end{scope}

	\begin{scope}[xshift=2.73cm, yshift=-2.1cm]
		\OmegaSetSharpOutline
		\OmegaSetOutlineSampleCount{5}
		\OmegaSetDeformation{phiTwox}{phiTwoy}
		
		\begin{scope}[scale=4]
			\OmegaSetShadingStyle{left color = blue, right color = blue, opacity = \deformationColorOpacity}
			\smallSquare
		\end{scope}
	\end{scope}

	\draw[backarrowstyle] (-2.52,-2.24) to[out=120, in=180] node[pos=.3, above left]{$\varphi_1$} (-.7,.21);
	\draw[backarrowstyle] (2.52,-2.24) to[out=60, in=0] node[pos=.3, above right]{$\varphi_2$} (.7,.21);

	\draw[arrowstyle, thick, dotted, color=red] (-.35,.07) -- node[pos=.7,above left=-.7mm] {$\grad\varphi_1(x)$} (-1.4,-.91);
	\draw[arrowstyle, thick, dotted, color=blue] (-.049,.07) -- node[pos=.7,above right=-.7mm] {$\grad\varphi_2(x)$} (1.4,-.91);

	\draw[doublearrowstyle, color=black, dashed] (-1.12,-1.4) to %
		node[pos=.52, above=2]{$\dg(\textcolor{red}{\grad\varphi_1(x)},\textcolor{blue}{\grad\varphi_2(x)})$} (1.12,-1.4);
	\draw[doublearrowstyle, color=black, dashed] (-1.4,-2.8) to %
		node[midway,above,align=center]{$\dist(\varphi_1,\varphi_2) :=$\\$\int_\Omega\dg(\textcolor{red}{\grad\varphi_1(x)},\textcolor{blue}{\grad\varphi_2(x)})\,\dx$} (1.4,-2.8);

	\node[above] at (0,.5) {$\Omega$};
	\node[above] at (0,.5) {$\Omega$};
\end{scope}

%% file: tikz/invarianceExplanation.tex
\OmegaSetDefaults
\OmegaSetGridSize{.2}{.2}
\OmegaSetOutlineSampleCount{42}
\OmegaSetGridSampleCount{10}
\OmegaSetOutlineToEllipse{0}{0}{1.5}{1}{0}

\begin{scope}[scale=.9, declare function ={%
		phix(\x,\y)=\x;%
		phiy(\x,\y)=\y*.5*(1+.75*(\x)^2)+.2*(\x)^2;
		Ax(\x,\y)=1.5*\x;
		Ay(\x,\y)=.75*\y;
		Aphix(\x,\y)=Ax(phix(\x,\y),phiy(\x,\y));
		Aphiy(\x,\y)=Ay(phix(\x,\y),phiy(\x,\y));
	}]

	\begin{scope}
		\OmegaSetDeformationToId
		\OmegaSetSmoothOutline
		\OmegaDraw
	\end{scope}

	\begin{scope}[xshift=8.4cm, yshift=0]
		\OmegaSetDeformation{phix}{phiy}
		\OmegaDraw
	\end{scope}

	\begin{scope}[xshift=0, yshift=-3.5cm]
		\OmegaSetDeformation{Ax}{Ay}
		\OmegaDraw
	\end{scope}

	\begin{scope}[xshift=8.4cm, yshift=-3.5cm]
		\OmegaSetDeformation{Aphix}{Aphiy}
		\OmegaDraw
	\end{scope}

	\tikzset{arrowstyle/.style={->}}

	\draw[arrowstyle] (-2,0) to[out=240, in=120] node[midway, right]{\large$B$} (-2,-2.5);
	\draw[arrowstyle] (10.5,0) to[out=300, in=60] node[midway, left]{\large$B$} (10.5,-2.5);
	
	\node at (0,1.3){\large$\varphi_1(\Omega)$};
	\node at (8.4,1.3){\large$\varphi_2(\Omega)$};
	\node at (0,-2.2){\large$B\cdot\varphi_1(\Omega)$};
	\node at (8.4,-2.2){\large$B\cdot\varphi_2(\Omega)$};
	
	\draw[arrowstyle, <->, color=blue, dashed] (1.75,-.1) to %
	node[pos=.45, below=2]{$\dist(\varphi_1,\varphi_2)$} (6.65,-.1);
	\draw[arrowstyle, <->, color=red, dashed] (2.1,-3) to %
	node[pos=.52, above=2]{$\dist(B\cdot\varphi_1, B\cdot\varphi_1)$} (5.95,-3);
	\node[rotate=90] at (4,-1.6) {\Large $\boldsymbol{=}$};
\end{scope}

%% file: tikz/tangentSpaceTransformation.tex
\begin{scope}[]
	\def\pointNodeSize{4}
	\tikzset{pointNodeStyle/.style={inner sep = 0, minimum size=\pointNodeSize, draw, circle, very thin, shade, shading=ball}}
	\tikzset{textNodeStyle/.style={font = {}}}
	\coordinate (vA) at (7,-5.5);
	\coordinate (wA) at (.9,3);
	\def\alpha{.225}
	\def\beta{.75}
	\coordinate (v) at (7,-6);
	\coordinate (w) at (3,1);
	\def\lambda{.25}
	\def\nu{.84}
	\coordinate (unity) at (0,1.785);
	\def\R{2.5}
	\def\angEl{44.5}

	\filldraw[color=\colorGLprime,ball color=ballColor,opacity=\GLBallOpacity] (0,0) circle (\R);
	\foreach \t in {-80,-70,...,80} { \DrawLatitudeCircle[\R]{\t}}
	\foreach \t in {-5,-25,...,-175} { \DrawLongitudeCircle[\R]{\t+10}}
	\coordinate (A) at (-1.625,-.11){};

	\draw[opacity=.75] ($(A)-{\alpha/2}*(vA)-{\beta/2}*(wA)$) -- ($(A)-{\alpha/2}*(vA)+{\beta/2}*(wA)$) -- ($(A)+{\alpha/2}*(vA)+{\beta/2}*(wA)$) -- ($(A)+{\alpha/2}*(vA)-{\beta/2}*(wA)$) -- cycle;
	\shade[ball color = tangentPlaneColor, opacity=.6] ($(A)-{\alpha/2}*(vA)-{\beta/2}*(wA)$) -- ($(A)-{\alpha/2}*(vA)+{\beta/2}*(wA)$) -- ($(A)+{\alpha/2}*(vA)+{\beta/2}*(wA)$) -- ($(A)+{\alpha/2}*(vA)-{\beta/2}*(wA)$) -- cycle;
	\draw[dashed, color=darkgray, opacity=.3] ($(A)-{\alpha/2}*(vA)$) -- ($(A)+{\alpha/2}*(vA)$);
	\draw[dashed, color=darkgray, opacity=.3] ($(A)-{\beta/2}*(wA)$) -- ($(A)+{\beta/2}*(wA)$);

	\draw[opacity=.75] ($(unity)-{\lambda/2}*(v)-{\nu/2}*(w)$) -- ($(unity)-{\lambda/2}*(v)+{\nu/2}*(w)$) -- ($(unity)+{\lambda/2}*(v)+{\nu/2}*(w)$) -- ($(unity)+{\lambda/2}*(v)-{\nu/2}*(w)$) -- cycle;
	\shade[ball color = tangentPlaneColor, opacity=.6] ($(unity)-{\lambda/2}*(v)-{\nu/2}*(w)$) -- ($(unity)-{\lambda/2}*(v)+{\nu/2}*(w)$) -- ($(unity)+{\lambda/2}*(v)+{\nu/2}*(w)$) -- ($(unity)+{\lambda/2}*(v)-{\nu/2}*(w)$) -- cycle;
	\draw[dashed, color=darkgray, opacity=.3] ($(unity)-{\lambda/2}*(v)$) -- ($(unity)+{\lambda/2}*(v)$);
	\draw[dashed, color=darkgray, opacity=.3] ($(unity)-{\nu/2}*(w)$) -- ($(unity)+{\nu/2}*(w)$);

	\node[pointNodeStyle, ball color = \colorRiemannId, minimum size = {1.5*\pointNodeSize}] (unityNode) at (unity){};
	\node[pointNodeStyle, ball color = \colorRiemannF] (ANode) at (A){};

	\coordinate (M) at ($(A)-{.2*\alpha}*(vA)+{.25*\beta}*(wA)$){};
	\coordinate (N) at ($(A)+{.2*\alpha}*(vA)+{.25*\beta}*(wA)$){};
	\node[pointNodeStyle, ball color=\colorRiemannX] (MNode) at (M){};
	\node[pointNodeStyle, ball color=\colorRiemannY] (NNode) at (N){};
	\draw[color=\colorRiemannX,opacity=.5,->] (ANode) -- (MNode);
	\draw[color=\colorRiemannY,opacity=.5,->] (ANode) -- (NNode);

	\coordinate (AinvM) at ($(unity)-{.2*\lambda}*(v)+{.25*\nu}*(w)$){};
	\coordinate (AinvN) at ($(unity)+{.2*\lambda}*(v)+{.25*\nu}*(w)$){};
	\node[pointNodeStyle, ball color=\colorRiemannX] (AinvMNode) at (AinvM){};
	\node[pointNodeStyle, ball color=\colorRiemannY] (AinvNNode) at (AinvN){};
	\draw[color=\colorRiemannX,opacity=.5,->] (unityNode) -- (AinvMNode);
	\draw[color=\colorRiemannY,opacity=.5,->] (unityNode) -- (AinvNNode);

	\node[textNodeStyle, color=\colorNametagGL] (GLText) at (3.2,.28){$\mathrm{GL^+}(n)$};
	\node[textNodeStyle, color=\colorNametagGL, left, align=center] (glText) at ($(A)-{.5*\alpha}*(vA)-{.3*\beta}*(wA)$){$T_{A}\mathrm{GL^+}(n)$\\$=A\cdot\mathfrak{gl}(n)$};
	\node[textNodeStyle, color=\colorNametagGL, right] (glText) at ($(unity)-{.21*\lambda}*(v)+{.56*\nu}*(w)$){$T_{\id}\mathrm{GL^+}(n)=\mathfrak{gl}(n)$};
	\node[textNodeStyle, below left, color=\colorRiemannId] (unityText) at (unityNode){$\id$};
	\node[textNodeStyle, left, color=\colorRiemannF] (AText) at (A){$A$};
	\node[textNodeStyle, right, color=\colorRiemannX, font=\small] (MText) at (M){$X$};
	\node[textNodeStyle, right, color=\colorRiemannY, font=\small] (NText) at (N){$Y$};
	\node[textNodeStyle, left, color=\colorRiemannX, font=\small] (AinvMText) at (AinvM){$A^{-1}X$};
	\node[textNodeStyle, below, color=\colorRiemannY, font=\small] (AinvNText) at (AinvN){$A^{-1}Y$};

	\tikzset{arrowstyle/.style={->, dotted, thick}}
	\draw[arrowstyle] ($(A)-{\alpha/1.6}*(vA)$) .. controls (-3.5,2) and (-2,3.5) .. node[midway, above left]{$A^{-1}$}($(unity)-{\lambda/1.6}*(v)$);

	\tikzset{formulastyle/.style={align=left, font={\small}}, color=darkgray}
	\node at ($(unity)+(0,2)$) {$g_A(X,Y) = \isoprod{A^{-1}X,A^{-1}Y}$};
	
\end{scope}

%% file: tikz/lagrangeEuler.tex
\OmegaSetDefaults
\OmegaSetGridSize{.25}{.25}
\OmegaSetOutlineSampleCount{42}
\OmegaSetGridSampleCount{10}
\OmegaSetSmoothOutline

\coordinate (omegaOneCenter) at (-2,0);
\coordinate (aboveOmegaOneCenter) at (-1.9,0.2);
\coordinate (belowOmegaOneCenter) at (-1.9,-2.8);
\coordinate (xiOneCenter) at (2,0);
\coordinate (aboveXiOneCenter) at (1.9,0.2);
\coordinate (belowXiOneCenter) at (1.9,-2.8);
\coordinate (lagrange) at (-4.9,.42);
\coordinate (euler) at (4.9,-2.45);

\newcommand{\omegaRadius}{(1cm)}
\newcommand{\xiRadii}{(2cm and 1cm)}
\newcommand{\xiAngle}{40}

\pgfmathdeclarefunction*{phiLinx}{2}{\pgfmathparse{#1+#2+2}}
\pgfmathdeclarefunction*{phiLiny}{2}{\pgfmathparse{1.26*#2-.14*#1}}
\pgfmathdeclarefunction*{phiLinInvx}{2}{\pgfmathparse{(1/1.4)*(1.26*(#1-2)-(#2+2.8))-2}}
\pgfmathdeclarefunction*{phiLinInvy}{2}{\pgfmathparse{(1/1.4)*(0.14*(#1-2)+(#2+2.8))-2.8}}

\begin{scope}
	\OmegaSetOutlineToCircle{-2}{0}{1}
	\OmegaSetDeformationToId
	\OmegaDraw

	\OmegaSetOutlineToCircle{0}{0}{1}
	\OmegaSetDeformation{phiLinx}{phiLiny}
	\OmegaDraw

	\draw[->,thick,color=\colorLagrange] (aboveOmegaOneCenter) to [out=45,in=135] node[above,pos=0.5]{$F$} (aboveXiOneCenter);
\end{scope}

\begin{scope}%
	\pgfmathdeclarefunction*{outlinex}{1}{\pgfmathparse{phiLinx(cos(#1),sin(#1))}}
	\pgfmathdeclarefunction*{outliney}{1}{\pgfmathparse{phiLiny(cos(#1),sin(#1))-2.8}}
	\OmegaSetOutline{outlinex}{outliney}{360}
	\OmegaSetDeformationToId
	\OmegaDraw

	\OmegaSetDeformation{phiLinInvx}{phiLinInvy}
	\OmegaDraw

	\node[above left] at (-2.8,.63) {$\Omega$};
	\node[below right] at (2.8,-.14) {$\varphi(\Omega)$};

	\draw[->,thick,color=\colorEuler] (belowXiOneCenter) to [out=225,in=315] node[above,pos=0.5]{$F^{{}^{-1}}$} (belowOmegaOneCenter);

	\node[align=center,below,color=\colorLagrange] (lagrangeNode) at (lagrange){Lagrangian frame\\(material setting)};
	\node[align=center,below,color=\colorEuler] (eulerNode) at (euler){Eulerian frame\\(spatial setting)};
\end{scope}

%% file: tikz/riemannianDistanceToSOn.tex
\begin{scope}[]
	\def\pointNodeSize{4}
	\tikzset{pointNodeStyle/.style={inner sep = 0, minimum size=\pointNodeSize, draw, circle, very thin, shade, shading=ball}}
	\tikzset{textNodeStyle/.style={font = {}}}
	\coordinate (v) at (7,-6);
	\coordinate (w) at (3,1);
	\def\lambda{.5} 
	\def\nu{1.5}
	\coordinate (unity) at (0,1.785);
	\def\R{2.5}
	\def\angEl{44.5}

	\filldraw[color = \colorGLprime, ball color=ballColor, opacity = \GLBallOpacity] (0,0) circle (\R);
	\foreach \t in {-80,-70,...,80} { \DrawLatitudeCircle[\R]{\t}}
	\foreach \t in {-5,-25,...,-175} { \DrawLongitudeCircle[\R]{\t+10}}
	\DrawSO[\R]{25}

	\coordinate (F) at (1.4,-1.113){};
	\coordinate (R) at (-2.23,-.42){};
	\DrawGeodesic[\R]{10}{-56}{153.5}
	\draw[->, thick, color=\colorRiemannTangent, dashed] (F) to (-.21,-1.8) node[below right] {$F\,\xi$};
	\node[pointNodeStyle, ball color = \colorRiemannF] (FNode) at (F){};
	\node[pointNodeStyle, ball color = \colorRiemannSO] (RNode) at (R){};

	\node[pointNodeStyle, ball color = \colorRiemannId, minimum size = {1.5*\pointNodeSize}] (unityNode) at (unity){};

	\node[textNodeStyle, color=\colorRiemannSO] (SOText) at (-.5,1){$\mathrm{SO}(n)$};
	\node[textNodeStyle, above left, color=\colorRiemannId] (unityText) at (unityNode){$\id$};

	\def\deltaT{0.05}
	\def\minOpacity{.2}
	\def\opacityExponent{2}
	\def\maxOpacity{.7}
	\pgfmathsetmacro\opacityFactor{(\maxOpacity-\minOpacity)/pow(.5,\opacityExponent)}

	\node[textNodeStyle, color=\colorNametagGL] (GLText) at (-3.3,.35){$\mathrm{GL^+}(n)$};
	\node[textNodeStyle, right, color=\colorRiemannF] (FText) at (F){$F$};
	\node[textNodeStyle, right, color=\colorRiemannSO] (FText) at (R){$\widehat{Q}$};

	\tikzset{arrowstyle/.style={<-, dotted, thick}}
	\tikzset{formulastyle/.style={align=left, font={\footnotesize}}}
	
	\draw[arrowstyle, color=\colorRiemannF] (.1,-1.2) .. controls (1.3,-.2) and (2,0) .. (3.2,0) node[formulastyle, color=\colorRiemannF, right]{$\mathrm{dist}_{\mathrm{geod}}^2(F, \mathrm{SO}(n))$};
\end{scope}

%% file: tikz/allDistances.tex
\begin{scope}[]
	\def\pointNodeSize{4}
	\tikzset{pointNodeStyle/.style={inner sep = 0, minimum size=\pointNodeSize, draw, circle, very thin, shade, shading=ball}}
	\tikzset{textNodeStyle/.style={font = {}}}
	\coordinate (v) at (7,-6);
	\coordinate (w) at (3,1);
	\def\lambda{.5} 
	\def\nu{1.5}
	\coordinate (unity) at (0,1.785);
	\def\R{2.5}
	\def\angEl{44.5}

	\filldraw[color = \colorGLprime, ball color=ballColor, opacity = \GLBallOpacity] (0,0) circle (\R);
	\foreach \t in {-80,-70,...,80} { \DrawLatitudeCircle[\R]{\t}}
	\foreach \t in {-5,-25,...,-175} { \DrawLongitudeCircle[\R]{\t+10}}

	\DrawSO[\R]{25}

	\coordinate (F) at (1.4,-1.1){};
	\coordinate (R) at (-2.23,-.42){};
	\node[pointNodeStyle, ball color = \colorRiemannF] (FNode) at (1.4,-1.1){};
	\node[pointNodeStyle, ball color = \colorRiemannSO] (RNode) at (-2.23,-.42){};
	\DrawGeodesic[\R]{10}{-56}{153.5}

	\node[pointNodeStyle, ball color = \colorRiemannId, minimum size = {1.5*\pointNodeSize}] (unityNode) at (unity){};

	\node[textNodeStyle, color=\colorRiemannSO] (SOText) at (-.5,1){$\mathrm{SO}(n)$};
	\node[textNodeStyle, below=7,right, color=\colorRiemannId] (unityText) at (unityNode){$\id$};
	\node[textNodeStyle, color=\colorRiemannPolar, right] (RText) at ($(R)+(0,.2)$){$R=\mathrm{polar}(F)$};

	\draw[opacity=.75] ($(unity)-{\lambda/2}*(v)-{\nu/2}*(w)$) -- ($(unity)-{\lambda/2}*(v)+{\nu/2}*(w)$) -- ($(unity)+{\lambda/2}*(v)+{\nu/2}*(w)$) -- ($(unity)+{\lambda/2}*(v)-{\nu/2}*(w)$) -- cycle;
	\shade[ball color = tangentPlaneColor, opacity=.6] ($(unity)-{\lambda/2}*(v)-{\nu/2}*(w)$) -- ($(unity)-{\lambda/2}*(v)+{\nu/2}*(w)$) -- ($(unity)+{\lambda/2}*(v)+{\nu/2}*(w)$) -- ($(unity)+{\lambda/2}*(v)-{\nu/2}*(w)$) -- cycle;

	\tikzset{connectParametersLeft/.style={semithick, color=\colorRiemannSO, opacity=.5, dotted}}
	\draw[connectParametersLeft] ($(unity)-{.2*\nu}*(w)$) -- ($(unity)-{.2*\nu}*(w)-(0,.13)$);
	\draw[connectParametersLeft] ($(unity)-{.25*\nu}*(w)$) -- ($(unity)-{.25*\nu}*(w)-(0,.22)$);
	\draw[connectParametersLeft] ($(unity)-{.3*\nu}*(w)$) -- ($(unity)-{.3*\nu}*(w)-(0,.35)$);
	\draw[connectParametersLeft] ($(unity)-{.35*\nu}*(w)$) -- ($(unity)-{.35*\nu}*(w)-(0,.49)$);
	\draw[connectParametersLeft] ($(unity)-{.4*\nu}*(w)$) -- ($(unity)-{.4*\nu}*(w)-(0,.7)$);
	\draw[connectParametersLeft, thick] ($(unity)-{.45*\nu}*(w)$) -- ($(unity)-{.45*\nu}*(w)-(0,1)$);
	\tikzset{connectParametersRight/.style={thin, color=\colorRiemannSO, opacity=.7, dotted}}
	\draw[connectParametersRight] ($(unity)+{.2*\nu}*(w)$) -- ($(unity)+{.2*\nu}*(w)-(0,.14)$);
	\draw[connectParametersRight] ($(unity)+{.25*\nu}*(w)$) -- ($(unity)+{.25*\nu}*(w)-(0,.26)$);
	\draw[connectParametersRight] ($(unity)+{.3*\nu}*(w)$) -- ($(unity)+{.3*\nu}*(w)-(0,.36)$);
	\draw[connectParametersRight] ($(unity)+{.35*\nu}*(w)$) -- ($(unity)+{.35*\nu}*(w)-(0,.52)$);
	\draw[connectParametersRight] ($(unity)+{.4*\nu}*(w)$) -- ($(unity)+{.4*\nu}*(w)-(0,.72)$);
	\draw[connectParametersRight, very thin] ($(unity)+{.45*\nu}*(w)$) -- ($(unity)+{.45*\nu}*(w)-(0,1)$);

	\draw[thick, color=\colorRiemannso] ($(unity)-{\nu/2.1}*(w)$) -- ($(unity)+{\nu/2.1}*(w)$);

	\coordinate (gradu) at ($(unity)+{.2*\lambda}*(v)+{.25*\nu}*(w)$){};
	\coordinate (skewgradu) at ($(unity)+{.25*\nu}*(w)$){};
	\node[pointNodeStyle, ball color=\colorRiemanngradu] (graduNode) at (gradu){};
	\draw[color=\colorRiemanngradu,opacity=.91,->] (unityNode) -- coordinate[midway](graduVector) (graduNode);
	\node[pointNodeStyle, ball color = \colorRiemannso] (skewgraduNode) at (skewgradu){};
	\draw[color=\colorRiemannsym,opacity=1,very thick] (graduNode) -- node[pos=.4, minimum size = 12](distglCenter){} (skewgraduNode);
	\draw[color=\colorRiemannso,opacity=1,->] (unityNode) -- coordinate[midway](skewgraduVector) (skewgraduNode);

	\def\deltaT{0.05}
	\def\minOpacity{.2}
	\def\opacityExponent{2}
	\def\maxOpacity{.7}
	\pgfmathsetmacro\opacityFactor{(\maxOpacity-\minOpacity)/pow(.5,\opacityExponent)}
	\foreach \t in {0,\deltaT,...,1}
			\draw[dashed, color=black, opacity={\minOpacity+\opacityFactor*pow(abs((\t+\deltaT/2)-.5),\opacityExponent)}] (${1-\t}*(F)+\t*(R)$) -- (${1-(\t+\deltaT)}*(F)+{\t+\deltaT}*(R)$);

	\node[textNodeStyle, color=\colorNametagGL] (GLText) at (-3.3,.35){$\mathrm{GL^{\!+}}(n)$};
	\node[textNodeStyle, color=\colorNametagGL, right] (glText) at ($(unity)-{.35*\lambda}*(v)+{.525*\nu}*(w)$){$T_{\id}\mathrm{GL^{\!+}}(n)=\mathfrak{gl}(n) \cong \Rnn$};
	\node[textNodeStyle, color=\colorRiemannso, right] (soText) at ($(unity)+{.518*\nu}*(w)$){$T_{\id}\mathrm{SO}(n)=\mathfrak{so}(n)$};
	\node[textNodeStyle, right, color=\colorRiemannF] (FText) at (F){$F$};
	\node[textNodeStyle, below, color=\colorRiemanngradu, font=\small] (graduText) at (graduVector){$\grad u$};
	\node[textNodeStyle, above=9, left=-12, color=\colorRiemannso, font=\small] (skewgraduText) at (skewgraduVector){$\mathrm{skew}\nabla u$};

	\tikzset{arrowstyle/.style={<-, dotted, thick}}
	\tikzset{formulastyle/.style={align=left, font={\footnotesize}}}
	\draw[arrowstyle, color=\colorRiemannsym] (distglCenter) .. controls (3,2) and (4,1.8) .. (4.55,1.6) node[formulastyle, color=\colorRiemannsym, right]{\\[7mm]$\mathrm{dist}_{\mathrm{euclid},\,\mathfrak{gl}}^2(\nabla u, \mathfrak{so}(n))$\\[1mm]$\hspace{4mm}=\mu\,|\!|\mathrm{dev}_n\,\mathrm{sym} \nabla u|\!|^2 + \frac{\kappa}{2}\,[\mathrm{tr}\nabla u]^2$};
	\draw[arrowstyle, color=darkgray] (.1,-.7) .. controls (1,-.2) and (3,-.2) .. (4.34,-.2) node[formulastyle, color=darkgray, right]{\\[4mm]$\mathrm{dist}_{\mathrm{euclid}}^2(F, \mathrm{SO}(n))$\\[1mm] $\hspace{4mm}=|\!|U-\id|\!|^2 = |\!|\sqrt{F^T\!F}-\id|\!|^2$};
	\draw[arrowstyle, color=\colorRiemannF] (.1,-1.55) .. controls (1,-2) and (3,-2) .. (4.2,-2) node[formulastyle, color=\colorRiemannF, right]{\\[4mm]$\mathrm{dist}_{\mathrm{geod}}^2(F, \mathrm{SO}(n))$\\[1mm] $\hspace{4mm}=\mu\,|\!|\mathrm{dev}_n\,\mathrm{log}\,U|\!|^2 + \frac{\kappa}{2}\,[\mathrm{tr}(\mathrm{log}U)]^2$};
\end{scope}

%% file: tikz/invarianceExplanation_rightInvariance.tex
\OmegaSetDefaults
\OmegaSetGridSize{.2}{.2}
\OmegaSetOutlineSampleCount{42}
\OmegaSetGridSampleCount{10}
\OmegaSetOutlineToUnitCircle

\begin{scope}[scale=.7, declare function ={%
		Fonex(\x,\y)=\x;
		Foney(\x,\y)=1.4*\y;
		Ftwox(\x,\y)=.91*\x+.7*\y;
		Ftwoy(\x,\y)=.91*\y;
		Ax(\x,\y)=1.5*\x;
		Ay(\x,\y)=.75*\y;
		FoneAx(\x,\y)=Fonex(Ax(\x,\y),Ay(\x,\y));
		FoneAy(\x,\y)=Foney(Ax(\x,\y),Ay(\x,\y));
		FtwoAx(\x,\y)=Ftwox(Ax(\x,\y),Ay(\x,\y));
		FtwoAy(\x,\y)=Ftwoy(Ax(\x,\y),Ay(\x,\y));
	}]

	\begin{scope}[yshift=-1cm]
		\OmegaSetShadingStyle{left color = \colorUndeformedLeft, right color = \colorUndeformedRight, opacity = \deformationDefaultOpacity}
		\OmegaSetDeformationToId
		\OmegaSetSmoothOutline
		\OmegaDraw
	\end{scope}

	\begin{scope}[xshift=7cm, yshift=-1cm]
		\OmegaSetShadingStyle{left color = \colorBdeformedLeft, right color = \colorBdeformedRight, opacity = \deformationDefaultOpacity}
		\OmegaSetDeformation{Ax}{Ay}
		\OmegaDraw
	\end{scope}

	\begin{scope}[xshift=0, yshift=-4.9cm]
		\OmegaSetShadingStyle{left color = \colorUndeformedLeft, right color = \colorUndeformedRight, opacity = \deformationDefaultOpacity}
		\OmegaSetDeformation{Fonex}{Foney}
		\OmegaDraw
	\end{scope}

	\begin{scope}[xshift=7cm, yshift=-4.9cm]
		\OmegaSetShadingStyle{left color = \colorBdeformedLeft, right color = \colorBdeformedRight, opacity = \deformationDefaultOpacity}
		\OmegaSetDeformation{FoneAx}{FoneAy}
		\OmegaDraw
	\end{scope}
	
	\begin{scope}[xshift=0, yshift=-9.8cm]
		\OmegaSetShadingStyle{left color = \colorUndeformedLeft, right color = \colorUndeformedRight, opacity = \deformationDefaultOpacity}
		\OmegaSetDeformation{Ftwox}{Ftwoy}
		\OmegaDraw
	\end{scope}

	\begin{scope}[xshift=7cm, yshift=-9.8cm]
		\OmegaSetShadingStyle{left color = \colorBdeformedLeft, right color = \colorBdeformedRight, opacity = \deformationDefaultOpacity}
		\OmegaSetDeformation{FtwoAx}{FtwoAy}
		\OmegaDraw
	\end{scope}

	\tikzset{arrowstyle/.style={->}}
	\draw[arrowstyle] (1,0.1) to[out=30, in=150] node[midway, below]{\large$B$} (5.9,0.1);
	
	\draw[arrowstyle] (-1.26,-1) to[out=240, in=120] node[midway, right]{\large$F_1$} (-1.26,-3.85);
	\draw[arrowstyle] (8.96,-1) to[out=300, in=60] node[midway, left]{\large$F_1$} (8.96,-3.85);
	
	\draw[arrowstyle] (-1.4,-1) to[out=240, in=120] node[midway, left]{\large$F_2$} (-1.4,-8.75);
	\draw[arrowstyle] (9.1,-1) to[out=300, in=60] node[midway, right]{\large$F_2$} (9.1,-8.75);
	
	\node at (0,0.3){\large$\Omega$};
	\node at (7,0.3){\large$B\cdot\Omega$};
	\node at (0,-3.15){\large$F_1\cdot\Omega$};
	\node at (7,-3.5){\large$F_1\cdot B\cdot\Omega$};
	\node at (-1.89,-9.8){\large$F_2\cdot\Omega$};
	\node at (9.73,-9.8){\large$F_2\cdot B\cdot\Omega$};
	
	\begin{scope}[yshift=-4.9cm]
		\draw[arrowstyle, <-, color=blue, dashed] (0,-1.61) to (0,-2.24);
		\node[color=blue] at (0,-2.59) {$\dg(F_1, F_2)$};
		\draw[arrowstyle, ->, color=blue, dashed] (0,-3.01) to (0,-3.92);
		
		\draw[arrowstyle, <-, color=red, dashed] (7,-1.61) to (7,-2.24);
		\node[color=red] at (7,-2.59) {$\dg(F_1\cdot B, F_2\cdot B)$};
		\draw[arrowstyle, ->, color=red, dashed] (7,-3.01) to (7,-3.92);
		
		\node[color=black] at (3.5,-2.59) {{\boldmath$=$}};
	\end{scope}
\end{scope}

%% file: tikz/thirdOrderConstants.tex
\tikzset{graphStyle/.style={smooth, color=black, very thick, opacity=.5}}
\begin{axis}[axisStyle, xmin=.6, xmax=1.68, xlabel={$\lambda$}, ylabel={$\Biot$}, legend style={legendStyle,xshift=-56mm,yshift=-28mm}, restrict x to domain=.5:2, xtick={0.7,1,1.4}, ytick={-12}]
	\addplot[graphStyle, samples=100, domain=.67:1.6, color=red, opacity=1,dotted] {Neohooke_t_uniaxial(x)};
	\addplot[graphStyle, samples=100, domain=.68:1.6, color=black, dashed, opacity=1] {ogden_t_uniaxial(x)};
	\addplot[graphStyle, samples=100, domain=.68:1.6, color=blue, opacity=1] {Hencky_t_uniaxial(x)};
	\addplot[markStyle, color=black, very thick] table[x=lambda1,y=t1] {\simpleElongationTable};
	\legend{\small Neo-Hooke,\small Ogden,\small Hencky, \footnotesize Experimental data for rubber};
	\addplot[markStyle, color=black, very thick] table[x=lambda3,y=t_compression] {\equibiaxialAndCompressionTable};
\end{axis}
\node at (0,0){};

%% file: tikz/expHenckyThreeDimensionalEoS.tex
\begin{scope}[scale=1, trim axis left, trim axis right]
	\pgfplotsset{yticklabel style={text width=1em,align=right},xlabel={$\det F$}, ylabel={$\frac13 \tr(\sigma)$}}
	\pgfplotsset{axisStyle/.style={axis x line=middle, axis y line=middle}}
	\tikzset{graphStyle/.style={smooth, color=blue, very thick}}
	
	\def\kappavar{1}
	\def\khatvar{4}

	\begin{scope}
		\begin{axis}[axisStyle, x post scale=.7, xmin=0, xmax=3.9, at={(0,0)}, yticklabels={},scaled y ticks = false,x label style={at={(current axis.right of origin)},anchor=north, below},]
			\addplot[graphStyle, samples=1000, domain=.455:2.73, dashed, color=red] {\kappavar*(ln(x)/x)*exp(\khatvar*ln(x)*ln(x))} node[pos=.84,left]{$\frac13 \tr(\sigma_{\mathrm{eH}})$};
			\addplot[graphStyle, samples=1000, domain=.105:3.78, color=blue] {\kappavar*(ln(x)/x)} node[pos=.9975,above left]{$\frac13 \tr(\sigma_{\mathrm{H}})$};
		\end{axis}
	\end{scope}
\end{scope}

%% file: tikz/oneDimensionalHenckyAndExpHenckyEnergyPlot.tex
\begin{scope}[scale=1, trim axis left, trim axis right]
	\pgfplotsset{yticklabel style={text width=1em,align=right},ytick={-100},xtick={1,5,25},xlabel={$\lambda$}}
	\pgfplotsset{axisStyle/.style={axis x line=middle, axis y line=middle}}
	\tikzset{graphStyle/.style={smooth, color=blue, very thick}}

	\begin{scope}
		\begin{axis}[axisStyle, ymax=18,xmin=0, y post scale=.75, at={(0,0)}]
			\addplot[graphStyle, samples=1000, domain=.021:40] {(ln(x))^2} node[below right, midway] {$\WH(\lambda)=\ln^2(\lambda)$};
			\addplot[graphStyle, samples=1000, domain=.1855:5.4, dashed, color=red] {exp((ln(x))^2)-1}node [right,pos=.84]{$\WeH(\lambda)=e^{\ln^2(\lambda)}$};
		\end{axis}
	\end{scope}

	\begin{scope}
		\begin{axis}[axisStyle, ymax=12,xmin=0, xmax=7, y post scale=.75, at={(7.7cm,0)},every axis x label/.style={
    at={(current axis.right of origin)}, anchor=north}]
			\addplot[graphStyle, samples=1000, domain=.4:7] {2*ln(x)/x} node[above, pos=.77] {$\sigmaH$};
			\addplot[graphStyle, samples=1000, domain=.5:5.4, dashed, color=red] {(2*ln(x)/x)*exp((ln(x))^2)}node [right,pos=.84]{$\sigmaeH$};
		\end{axis}
	\end{scope}

\node at(0,0){};
\end{scope}

%% file: tikz/tensionCompressionSymmetry.tex
\OmegaSetDefaults
\OmegaSetGridSize{.25}{.25}
\OmegaSetOutlineSampleCount{42}
\OmegaSetGridSampleCount{10}
\OmegaSetSmoothOutline

\coordinate (omegaOneCenter) at (-2,0);
\coordinate (aboveOmegaOneCenter) at (-1.9,0.2);
\coordinate (belowOmegaOneCenter) at (-1.9,-2.8);
\coordinate (xiOneCenter) at (2,0);
\coordinate (aboveXiOneCenter) at (1.9,0.2);
\coordinate (belowXiOneCenter) at (1.9,-2.8);
\coordinate (lagrange) at (-2,-4.76);
\coordinate (euler) at (2,-4.76);

\newcommand{\omegaRadius}{(1cm)}
\newcommand{\xiRadii}{(2cm and 1cm)}
\newcommand{\xiAngle}{40}

\def\aconst{1}
\def\bconst{.42}

\pgfmathdeclarefunction*{phiLinx}{2}{\pgfmathparse{\aconst*#1+2}}
\pgfmathdeclarefunction*{phiLiny}{2}{\pgfmathparse{#2+\bconst*#1*#1}}
\pgfmathdeclarefunction*{phiLinInvx}{2}{\pgfmathparse{(#1-2)/\aconst-2}}
\pgfmathdeclarefunction*{phiLinInvy}{2}{\pgfmathparse{#2-\bconst*(#1-2)*(#1-2)/(\aconst*\aconst)}}
\begin{scope}
	\OmegaSetOutlineToCircle{-2}{0}{1}
	\OmegaSetDeformationToId
	\OmegaDraw

	\OmegaSetOutlineToCircle{0}{0}{1}
	\OmegaSetDeformation{phiLinx}{phiLiny}
	\OmegaDraw

	\draw[->,thick] (aboveOmegaOneCenter) to [out=45,in=135] node[above,pos=0.5]{$\varphi$} (aboveXiOneCenter);
\end{scope}

\begin{scope}%
	\pgfmathdeclarefunction*{outlinex}{1}{\pgfmathparse{phiLinx(cos(#1),sin(#1))}}
	\pgfmathdeclarefunction*{outliney}{1}{\pgfmathparse{phiLiny(cos(#1),sin(#1))-2.8}}
	\OmegaSetOutline{outlinex}{outliney}{360}
	\OmegaSetDeformationToId
	\OmegaDraw

	\OmegaSetDeformation{phiLinInvx}{phiLinInvy}
	\OmegaDraw

	\node at (-2,-1.4) {$\Omega$};
	\node at (2,-1.4) {$\varphi(\Omega)$};

	\draw[->,thick] (belowXiOneCenter) to [out=225,in=315] node[above,pos=0.5]{$\varphi^{{}^{-1}}$} (belowOmegaOneCenter);
\end{scope}

%% file: tikz/geodesicConvexity.tex
\def\tikzconstPangle{-40}
\def\tikzconstQangle{110}
\def\tikzconstMidAngle{(\tikzconstPangle+\tikzconstQangle)/2+10}
\pgfmathsetmacro\QposX{3+1.5*cos(\tikzconstPangle}
\pgfmathsetmacro\QposY{3+1.5*sin(\tikzconstPangle)}
\pgfmathsetmacro\RposX{3+1.5*cos(\tikzconstQangle}
\pgfmathsetmacro\RposY{3+1.5*sin(\tikzconstQangle)}
\definecolor{euclidcolor}{rgb}{.5,0,1}
\foreach \p in {0,1}{
\begin{scope}[xshift={190*\p},yshift={0*\p}]
	\begin{scope}[yslant=0.5,xslant=-1,xshift=160]
		\fill[left color=planeColorLeft,right color=planeColorRight] (1,1) rectangle (5,5);
		\draw[color=black, opacity=.2,step=.5] (0.8,0.8) grid (5.2,5.2);
		\draw[color=\colorFlatSOconvexity, thick] (3,3) circle (1.5);
		\foreach \x in {0,.01,...,.2}
		\fill[color=white, very thick, opacity={1-5*\x}] (3,3) circle (\x);
		\node (F) at (.5,5.5) {};
		\node (R) at (\RposX,\RposY) {};
		\node (Q) at (\QposX,\QposY) {};
		\node[color=\colorFlatSOconvexity,below] (SOn) at (2.1,1.75) {\large $\SOn$};
		\draw [\colorConvexitySOgeodesic,ultra thick,domain=\tikzconstPangle:\tikzconstQangle,opacity={1-\p}] plot ({3+1.5*cos(\x)}, {3+1.5*sin(\x)});
		\draw [\colorConvexityGLgeodesic,ultra thick,opacity={1-\p}] (\RposX, \RposY) to[out=-110,in=168] (\QposX, \QposY);
		\draw [\colorConvexityGLgeodesic,ultra thick,domain=\tikzconstPangle:\tikzconstQangle,opacity={\p}] plot ({3+1.5*cos(\x)}, {3+1.5*sin(\x)});
		\node (geodesicText) at ({3+1.5*cos(\tikzconstMidAngle)}, {3+1.5*sin(\tikzconstMidAngle)}) {};
	\end{scope}
	\begin{scope}
		\node[below right,yslant=0.5,xslant=-1,color=\colorNametagGL] at (1.435,5.81) {{\boldmath $\GLpn$}};
		\fill[color=\colorFlatSOconvexity] (R) circle (.075);
		\fill[color=\colorFlatSOconvexity] (Q) circle (.075);
		\node[color=\colorFlatSOconvexity, below left] at (Q) {$Q$};
		\node[color=\colorFlatSOconvexity, below] at (R) {$R$};
		\node[above=2, color=\colorConvexitySOgeodesic,opacity={1-\p}] at (geodesicText) {{$\muc\cdot\dist_{\SOn}(Q,R) = \mu_c\hspace{0.5pt}\norm{\!\log (Q^TR)}^2$}};
		\node[below=13.3, color=\colorConvexityGLgeodesic,opacity={1-\p}] at (geodesicText) {{$\quad\;\;\dist_{\GLpn}(Q,R)$\;?}};
		\node[above=2, color=\colorConvexityGLgeodesic,opacity={\p}] at (geodesicText) {{$\dist_{\GLpn}(Q,R) = \mu_c\hspace{0.5pt}\norm{\!\log (Q^TR)}^2$\;?}};
	\end{scope}
\end{scope}
}

%% file: tikz/tensorsAndTangentSpaces.tex
\tikzset{curvestyle/.style={smooth, semithick, variable=\t, samples=25, domain=-1:1.75}}
\tikzset{pointNodeStyle/.style={inner sep = 0, minimum size=2, draw, circle, very thin, shade, shading=ball}}
\tikzset{textNodeStyle/.style={font = {}}}
\tikzset{tangentStyle/.style={->, semithick}}

\def\rotOut{-20}
\def\rotPhi{5}

\definecolor{Ocolor}{rgb}{.4,.4,1}
\definecolor{phiOcolor}{rgb}{1,.4,.4}
\definecolor{Otangentcolor}{rgb}{0,0,1}
\definecolor{phiOtangentcolor}{rgb}{1,0,0}

\OmegaSetDefaults
\OmegaSetGridSize{.2}{.2}
\OmegaSetOutlineSampleCount{21}
\OmegaSetGridSampleCount{5}

\begin{scope}[declare function ={%
		outlinestraightx(\t)=2*cos(\t);
		outlinestraighty(\t)=2*(sin(\t))*.5*(1+.75*(cos(\t))^2)+.25*(cos(\t))^2+.75*(cos(\t))^3;
		outlinex(\t)=cos(\rotOut)*outlinestraightx(\t) - sin(\rotOut)*outlinestraighty(\t);
		outliney(\t)=cos(\rotOut)*outlinestraighty(\t) + sin(\rotOut)*outlinestraightx(\t);
		phistraightx(\x,\y)=\x;
		phistraighty(\x,\y)=1.1*\y*(1-.05*(\x)^2)-.1*(\x)^2-.075*(\x)^3;
		phix(\x,\y)=cos(\rotPhi)*phistraightx(\x,\y) - sin(\rotPhi)*phistraighty(\x,\y);
		phiy(\x,\y)=cos(\rotPhi)*phistraighty(\x,\y) + sin(\rotPhi)*phistraightx(\x,\y);
		gammax(\t)=\t-.75;
		gammay(\t)=.1-.3*(\t)^2;
	}]

	\OmegaSetOutline{outlinex}{outliney}{360}
	\OmegaSetOutlineStyle{color = darkgray, opacity = .9, thick}
	\OmegaSetGridStyle{very thin, color=darkgray, opacity = .3, ->}
	\OmegaSetSmoothOutline

	\begin{scope}
		\OmegaSetDeformationToId
		\OmegaDraw
		\draw[curvestyle, color=Ocolor] plot ({gammax(\t)}, {gammay(\t)});
	\end{scope}

	\begin{scope}[xshift=5.5cm, yshift=0cm]
		\OmegaSetDeformation{phix}{phiy}
		\OmegaDraw
		\draw[curvestyle, color=phiOcolor] plot ({phix(gammax(\t), gammay(\t))}, {phiy(gammax(\t), gammay(\t))});
		\node[pointNodeStyle, ball color=phiOcolor] (phiXtNode) at ({phix(gammax(.1), gammay(.1))},{phiy(gammax(.1), gammay(.1))}){};
	\end{scope}

	\node[pointNodeStyle, ball color=Ocolor] (XtNode) at ({gammax(.1)}, {gammay(.1)}){};

	\draw[tangentStyle, color=Otangentcolor] (XtNode) -- ($(XtNode) + (2,-.12)$) coordinate[midway] (XdottNode);
	\draw[tangentStyle, color=phiOtangentcolor] (phiXtNode) -- ($(phiXtNode) + (1.5,.05)$) node[pos=0, sloped, below] {$\varphi(x(t))$} node[pos=.42, sloped, above] (phiXdottNode){$\nabla\varphi(x(t)).\dot x(t)$};

	\node[below, color=Ocolor] at (XtNode) {$x(t)$};
	\node[above, color=Otangentcolor] at (XdottNode) {$\dot x(t)$};
	\node at (-1.5,1.5){$\Omega$};
	\node at (6,1.5){$\varphi(\Omega)$};

	\draw[->] (1.2,1.2) to[out=45, in=135] node[midway, below]{$\varphi$} (3.8,1.2);
\end{scope}

	\node[color=Otangentcolor, right, align=left] at (-1.4,-2){$\dot x(t)\in T_{x(t)}\Omega $\\[2mm]$U,C: T_x\Omega \to T_x\Omega$};
	\node[color=phiOtangentcolor, right, align=left] at (3.3,-2){$\nabla\varphi.\dot x(t)\in T_{\varphi(x(t))}\varphi(\Omega)$\\[2mm]$V,B: T_{\varphi(x)}\varphi(\Omega) \to T_{\varphi(x)}\varphi(\Omega)$};
	\node[color=black, right, align=center] at (1,-3.2){$F,R: T_{x}\Omega\to T_{\varphi(x)}\varphi(\Omega)$\\[1mm](two-point tensors)};

\tikzset{arrowstyle/.style={->}}